\pdfoutput=1
\documentclass[a4paper,11pt]{amsart}
\usepackage[hmarginratio={1:1},vmarginratio={1:1},lmargin=60.0pt,tmargin=60.0pt]{geometry}

\allowdisplaybreaks

\usepackage[numbers]{natbib}
\usepackage[utf8]{inputenc}
\usepackage{latexsym,exscale,mathtools}
\usepackage{amssymb,amsmath,amsthm,amsfonts,enumitem}
\usepackage[table]{xcolor}
\usepackage{booktabs,array}
\usepackage{etoolbox,needspace}

\setlist[enumerate]{itemsep=0.15cm,label=\emph{\upshape(\alph*)}}
\setlist[enumerate,2]{itemsep=0.15cm,label=\emph{\upshape(\roman*)}}

\usepackage{tikz}
\usetikzlibrary{arrows.meta}
\definecolor{orchid}{RGB}{143,40,194}

\definecolor{lava}{RGB}{207,16,32}
\definecolor{mydarkblue}{RGB}{10,10,170}
\definecolor{mred}{RGB}{255,10,10}
\definecolor{treegray}{gray}{0.88}
\definecolor{treegraylight}{gray}{0.96}
\usepackage{aliascnt}

\let\emph\relax
\DeclareTextFontCommand{\emph}{\bfseries\em}
\tikzset{
  anchorbase/.style={baseline={([yshift=#1]current bounding box.center)}},
  anchorbase/.default={-0.5ex},
  every picture/.append style={line cap=round,line join=round,anchorbase},
  a/.style={draw=black,line width=1pt},
  m/.style={draw=mred,line width=2.5pt},
  dot/.style={circle,fill=black,inner sep=1.35pt},
  coupon/.style={draw=black,fill=treegray!45,inner sep=2.6pt,
    rounded corners=1pt,font=\scriptsize},
  over/.style={preaction={
    draw=white,
    line width=6pt,
    shorten <=1.2mm,
    shorten >=1.2mm
  }},
  region/.style={fill=treegray,draw=none},
  arr/.style={-{Stealth[length=1.7mm]},line width=.8pt}
}
\newcommand{\C}{\mathbb C}
\newcommand{\Z}{\mathbb Z}
\newcommand{\id}{\operatorname{id}}
\newcommand{\Hom}{\operatorname{Hom}}
\newcommand{\End}{\operatorname{End}}
\newcommand{\FPdim}{\operatorname{FPdim}}
\newcommand{\Mat}{\operatorname{Mat}}
\newcommand{\TY}{\operatorname{TY}}
\renewcommand{\Vec}{\operatorname{Vec}}
\newcommand{\Add}{\operatorname{Add}}
\newcommand{\ev}{\operatorname{ev}}
\newcommand{\coev}{\operatorname{coev}}
\newcommand{\tr}{\operatorname{tr}}
\newcommand{\Out}{\operatorname{Out}}
\newcommand{\D}{\mathcal D(A,\chi,\tau)}
\def\NewTheorem#1#2{%
\newaliascnt{#1}{equation}%
\newtheorem{#1}[#1]{#2}%
\aliascntresetthe{#1}%
\expandafter\def\csname #1autorefname\endcsname{#2}%
}
\def\equationautorefname~#1\null{(#1)\null}

\numberwithin{equation}{subsection}

\theoremstyle{plain}
\NewTheorem{proposition}{Proposition}
\NewTheorem{theorem}{Theorem}
\NewTheorem{corollary}{Corollary}
\AtEndEnvironment{corollary}{\null\hfill$\square$}%
\NewTheorem{lemma}{Lemma}
\theoremstyle{definition}
\NewTheorem{definition}{Definition}
\AtEndEnvironment{definition}{\null\hfill$\Diamond$}%
\NewTheorem{example}{Example}
\AtEndEnvironment{example}{\null\hfill$\Diamond$}%
\theoremstyle{remark}
\NewTheorem{remark}{Remark}
\AtEndEnvironment{remark}{\null\hfill$\Diamond$}%

\usepackage[hypertexnames=false]{hyperref}
\usepackage{bookmark}
\hypersetup{
pdftoolbar=true,
pdfmenubar=true,
pdffitwindow=false,
pdfstartview={FitH},
pdftitle={The Tambara--Yamagami calculus and finite semigroups},
pdfauthor={Super Cao and Daniel Tubbenhauer},
pdfsubject={},
pdfcreator={Super Cao and Daniel Tubbenhauer},
pdfproducer={Super Cao and Daniel Tubbenhauer},
pdfkeywords={},
pdfnewwindow=true,
colorlinks=true,
linkcolor=mydarkblue,
citecolor=teal,
filecolor=magenta,
urlcolor=orchid,
linkbordercolor=lava,
citebordercolor=teal,
urlbordercolor=orchid,
linktocpage=true
}
\expandafter\def\csname subjclassname@2020\endcsname{2020 Mathematics Subject Classification}

\BeforeBeginEnvironment{definition}{\begin{samepage}}
\AfterEndEnvironment{definition}{\end{samepage}}
\BeforeBeginEnvironment{lemma}{\begin{samepage}}
\AfterEndEnvironment{lemma}{\end{samepage}}
\BeforeBeginEnvironment{proposition}{\begin{samepage}}
\AfterEndEnvironment{proposition}{\end{samepage}}
\BeforeBeginEnvironment{theorem}{\begin{samepage}}
\AfterEndEnvironment{theorem}{\end{samepage}}

\newcommand{\edge}[2]{\ifboolexpr{test {\ifstrequal{#1}{m}} or test {\ifstrequal{#1}{m^*}} or test {\ifstrequal{#1}{m^{**}}}}{\draw[m] #2;}{\ifstrequal{#1}{1}{\draw[gray,densely dotted] #2;}{\draw[a] #2;}}}
\newcommand{\wire}[1]{\begin{tikzpicture}[scale=.7]
 \edge{#1}{(0,0)--(0,1.15)}\node[below,font=\scriptsize] at (0,0){$#1$};\end{tikzpicture}}
\newcommand{\fuse}[3]{\begin{tikzpicture}[scale=.78,every node/.style={font=\scriptsize}]
 \edge{#1}{(-.48,0)--(0,.52)}\edge{#2}{(.48,0)--(0,.52)}\edge{#3}{(0,.52)--(0,1)}
 \node[dot] at (0,.52){};\node[below] at (-.48,0){$#1$};\node[below] at (.48,0){$#2$};\node[above] at (0,1){$#3$};\end{tikzpicture}}
\newcommand{\splitv}[3]{\begin{tikzpicture}[scale=.78,every node/.style={font=\scriptsize}]
 \edge{#1}{(-.48,1)--(0,.48)}\edge{#2}{(.48,1)--(0,.48)}\edge{#3}{(0,.48)--(0,0)}
 \node[dot] at (0,.48){};\node[above] at (-.48,1){$#1$};\node[above] at (.48,1){$#2$};\node[below] at (0,0){$#3$};\end{tikzpicture}}
\newcommand{\lefttree}[5]{\begin{tikzpicture}[scale=.8,every node/.style={font=\scriptsize}]
 \edge{#1}{(-.9,0)--(-.45,.5)}\edge{#2}{(0,0)--(-.45,.5)}
 \edge{#4}{(-.45,.5)--(0,1.05)}\edge{#3}{(.9,0)--(0,1.05)}\edge{#5}{(0,1.05)--(0,1.5)}
 \node[dot] at (-.45,.5){};\node[dot] at (0,1.05){};
 \node[below] at (-.9,0){$#1$};\node[below] at (0,0){$#2$};\node[below] at (.9,0){$#3$};
 \node[left] at (-.25,.78){$#4$};\node[above] at (0,1.5){$#5$};\end{tikzpicture}}
\newcommand{\righttree}[5]{\begin{tikzpicture}[scale=.8,every node/.style={font=\scriptsize}]
 \edge{#2}{(0,0)--(.45,.5)}\edge{#3}{(.9,0)--(.45,.5)}
 \edge{#4}{(.45,.5)--(0,1.05)}\edge{#1}{(-.9,0)--(0,1.05)}\edge{#5}{(0,1.05)--(0,1.5)}
 \node[dot] at (.45,.5){};\node[dot] at (0,1.05){};
 \node[below] at (-.9,0){$#1$};\node[below] at (0,0){$#2$};\node[below] at (.9,0){$#3$};
 \node[right] at (.25,.78){$#4$};\node[above] at (0,1.5){$#5$};\end{tikzpicture}}
\newcommand{\leftsplit}[5]{\begin{tikzpicture}[scale=.8,every node/.style={font=\scriptsize}]
 \edge{#1}{(-.9,1.5)--(-.45,1)}\edge{#2}{(0,1.5)--(-.45,1)}
 \edge{#4}{(-.45,1)--(0,.45)}\edge{#3}{(.9,1.5)--(0,.45)}\edge{#5}{(0,.45)--(0,0)}
 \node[dot] at (-.45,1){};\node[dot] at (0,.45){};
 \node[above] at (-.9,1.5){$#1$};\node[above] at (0,1.5){$#2$};\node[above] at (.9,1.5){$#3$};
 \node[left] at (-.25,.75){$#4$};\node[below] at (0,0){$#5$};\end{tikzpicture}}
\newcommand{\rightsplit}[5]{\begin{tikzpicture}[scale=.8,every node/.style={font=\scriptsize}]
 \edge{#2}{(0,1.5)--(.45,1)}\edge{#3}{(.9,1.5)--(.45,1)}
 \edge{#4}{(.45,1)--(0,.45)}\edge{#1}{(-.9,1.5)--(0,.45)}\edge{#5}{(0,.45)--(0,0)}
 \node[dot] at (.45,1){};\node[dot] at (0,.45){};
 \node[above] at (-.9,1.5){$#1$};\node[above] at (0,1.5){$#2$};\node[above] at (.9,1.5){$#3$};
 \node[right] at (.25,.75){$#4$};\node[below] at (0,0){$#5$};\end{tikzpicture}}
\newcommand{\channel}[3]{\begin{tikzpicture}[scale=.72,every node/.style={font=\scriptsize}]
 \edge{#1}{(-.48,0)--(0,.45)}\edge{#2}{(.48,0)--(0,.45)}\edge{#3}{(0,.45)--(0,1.05)}
 \edge{#1}{(0,1.05)--(-.48,1.5)}\edge{#2}{(0,1.05)--(.48,1.5)}
 \node[dot] at (0,.45){};\node[dot] at (0,1.05){};\node[right] at (0,.75){$#3$};
 \node[below] at (-.48,0){$#1$};\node[below] at (.48,0){$#2$};\end{tikzpicture}}
\newcommand{\crossing}[2]{\begin{tikzpicture}[scale=.8,every node/.style={font=\scriptsize}]
 \edge{#2}{(.5,0)--(-.5,1.1)}
 \ifboolexpr{test {\ifstrequal{#1}{m}} or test {\ifstrequal{#1}{m^*}} or test {\ifstrequal{#1}{m^{**}}}}{\draw[m,over](-.5,0)--(.5,1.1);}{\draw[a,over](-.5,0)--(.5,1.1);}
 \node[below] at (-.5,0){$#1$};\node[below] at (.5,0){$#2$};
 \node[above] at (-.5,1.1){$#2$};\node[above] at (.5,1.1){$#1$};\end{tikzpicture}}
\newcommand{\resolvedcross}[3]{\begin{tikzpicture}[scale=.72,every node/.style={font=\scriptsize}]
 \edge{#1}{(-.5,0)--(0,.45)}\edge{#2}{(.5,0)--(0,.45)}\edge{#3}{(0,.45)--(0,1.05)}
 \edge{#2}{(0,1.05)--(-.5,1.5)}\edge{#1}{(0,1.05)--(.5,1.5)}
 \node[dot] at (0,.45){};\node[dot] at (0,1.05){};\node[right] at (0,.75){$#3$};
 \node[below] at (-.5,0){$#1$};\node[below] at (.5,0){$#2$};
 \node[above] at (-.5,1.5){$#2$};\node[above] at (.5,1.5){$#1$};\end{tikzpicture}}
\newcommand{\twistpic}[1]{%
\begin{tikzpicture}[scale=.8]
 \edge{#1}{(0,0)--(0,1.35)}
 \node[coupon,minimum width=.48cm,minimum height=.47cm]
   at (0,.68){$\theta$};
 \node[below,font=\scriptsize] at (0,0){$#1$};
\end{tikzpicture}}
\newcommand{\doublecross}[2]{\begin{tikzpicture}[scale=.8,every node/.style={font=\scriptsize}]
 \edge{#1}{(-.45,0)--(.45,.8)--(-.45,1.6)}\edge{#2}{(.45,0)--(-.45,.8)--(.45,1.6)}
 \ifboolexpr{test {\ifstrequal{#1}{m}} or test {\ifstrequal{#1}{m^*}} or test {\ifstrequal{#1}{m^{**}}}}{\draw[m,over](-.45,0)--(.45,.8);}{\draw[a,over](-.45,0)--(.45,.8);}
 \ifboolexpr{test {\ifstrequal{#2}{m}} or test {\ifstrequal{#2}{m^*}} or test {\ifstrequal{#2}{m^{**}}}}{\draw[m,over](-.45,.8)--(.45,1.6);}{\draw[a,over](-.45,.8)--(.45,1.6);}
 \node[below] at (-.45,0){$#1$};\node[below] at (.45,0){$#2$};\end{tikzpicture}}
\newcommand{\braidthree}[1]{\begin{tikzpicture}[scale=.72]
 \foreach \k [count=\h from 0] in {#1}{
  \pgfmathsetmacro{\yy}{.65*\h}\pgfmathsetmacro{\xx}{.65*(\k-1)}
  \foreach \j in {0,1,2}{\ifnum\j=\numexpr\k-1\relax\else\ifnum\j=\k\relax\else
   \draw[m](.65*\j,\yy)--(.65*\j,\yy+.65);\fi\fi}
  \draw[m](\xx+.65,\yy)--(\xx,\yy+.65);
  \draw[m,over](\xx,\yy)--(\xx+.65,\yy+.65);}
 \foreach \j in {0,1,2}{\node[below,font=\scriptsize] at (.65*\j,0){$m$};}\end{tikzpicture}}
\newcommand{\closedtwobraid}[1]{\begin{tikzpicture}[scale=.63]
 \pgfmathsetmacro{\hh}{.55*#1+.3}
 \draw[m](.7,\hh) to[out=90,in=90](1.4,\hh)--(1.4,0)
    to[out=-90,in=-90](.7,0);
 \draw[m](0,\hh) to[out=90,in=90](2.1,\hh)--(2.1,0)
    to[out=-90,in=-90](0,0);
 \draw[m](0,0)--(0,.3);\draw[m](.7,0)--(.7,.3);
 \foreach \j in {1,...,#1}{
  \pgfmathsetmacro{\yy}{.3+.55*(\j-1)}
  \draw[m](.7,\yy)--(0,\yy+.55);
  \draw[m,over](0,\yy)--(.7,\yy+.55);}
 \node[coupon,inner sep=1.6pt] at (0,.14){$\rho$};
 \node[coupon,inner sep=1.6pt] at (.7,.14){$\rho$};
\end{tikzpicture}}
\newcommand{\ztlid}[1]{\begin{tikzpicture}[scale=.68]
 \foreach \j in {1,...,#1}{\draw[m](.7*\j,0)--(.7*\j,1.2);}
 \end{tikzpicture}}
\newcommand{\ztl}[2]{\begin{tikzpicture}[scale=.68]
 \foreach \k [count=\h from 0] in {#2}{
  \pgfmathsetmacro{\yy}{.85*\h}\pgfmathsetmacro{\xx}{.7*\k}
  \foreach \j in {1,...,#1}{\ifnum\j=\k\relax\else
   \ifnum\j=\numexpr\k+1\relax\else
    \draw[m](.7*\j,\yy)--(.7*\j,\yy+.85);\fi\fi}
  \draw[m](\xx,\yy) to[out=90,in=180](\xx+.35,\yy+.3)
    to[out=0,in=90](\xx+.7,\yy);
  \draw[m](\xx,\yy+.85) to[out=-90,in=180](\xx+.35,\yy+.55)
    to[out=0,in=-90](\xx+.7,\yy+.85);}
 \end{tikzpicture}}
\newcommand{\zlocal}[2]{\begin{tikzpicture}[scale=.68]
 \foreach \k/\lab [count=\h from 0] in {#2}{
  \pgfmathsetmacro{\yy}{.82*\h}
  \foreach \j in {1,...,#1}{\draw[m](.7*\j,\yy)--(.7*\j,\yy+.82);}
  \node[coupon,minimum width=.68cm,minimum height=.28cm] at (.7*\k+.35,\yy+.41){$\lab$};}
 \end{tikzpicture}}
\newcommand{\zchannels}[2]{\begin{tikzpicture}[scale=.68]
 \foreach \k/\lab [count=\h from 0] in {#2}{
  \pgfmathsetmacro{\yy}{1.25*\h}\pgfmathsetmacro{\xx}{.7*\k}
  \foreach \j in {1,...,#1}{\ifnum\j=\k\relax\else
   \ifnum\j=\numexpr\k+1\relax\else
    \draw[m](.7*\j,\yy)--(.7*\j,\yy+1.25);\fi\fi}
  \draw[m](\xx,\yy)--(\xx+.35,\yy+.35)--(\xx+.7,\yy);
  \ifnum\pdfstrcmp{\lab}{1}=0
   \draw[gray,densely dotted](\xx+.35,\yy+.35)--(\xx+.35,\yy+.9);
  \else\draw[a](\xx+.35,\yy+.35)--(\xx+.35,\yy+.9);\fi
  \draw[m](\xx,\yy+1.25)--(\xx+.35,\yy+.9)--(\xx+.7,\yy+1.25);
  \node[dot] at (\xx+.35,\yy+.35){};
  \node[dot] at (\xx+.35,\yy+.9){};
  \node[right,font=\scriptsize] at (\xx+.35,\yy+.625){$\lab$};}
 \end{tikzpicture}}
\newcommand{\zmatrixunit}[2]{\begin{tikzpicture}[scale=.7,every node/.style={font=\scriptsize}]
 \draw[m](-.8,0)--(-.4,.45)--(0,0);
 \edge{#2}{(-.4,.45)--(0,.95)}
 \draw[m](.8,0)--(0,.95)--(0,1.45)--(.8,2.4);
 \edge{#1}{(0,1.45)--(-.4,1.95)}
 \draw[m](-.8,2.4)--(-.4,1.95)--(0,2.4);
 \foreach \p in {(-.4,.45),(0,.95),(0,1.45),(-.4,1.95)}{\node[dot] at \p{};}
 \node[left] at (-.23,.68){$#2$};\node[left] at (-.23,1.72){$#1$};
 \end{tikzpicture}}
\newcommand{\zloop}[1]{\begin{tikzpicture}[scale=.7]
 \draw[m](0,.65) ellipse (.45 and .65);
 \ifstrequal{#1}{raw}{}{\node[coupon,inner sep=2pt] at (-.45,.65){$#1$};}
 \end{tikzpicture}}
\newcommand{\ztrace}[3]{\begin{tikzpicture}[scale=.59]
 \pgfmathtruncatemacro{\laststrand}{#1-1}
 \foreach \j in {0,...,\laststrand}{
  \pgfmathsetmacro{\xx}{.7*\j}
  \pgfmathsetmacro{\rr}{.7*(2*#1-1-\j)}
  \draw[m](\xx,0)--(\xx,1.65) to[out=90,in=90](\rr,1.65)
   --(\rr,0) to[out=-90,in=-90](\xx,0);
  \node[coupon,inner sep=1.4pt] at (\xx,.2){$#3$};}
 \pgfmathsetlengthmacro{\ww}{.7*(#1-1)*.59cm+.22cm}
 \pgfmathsetmacro{\cc}{.35*(#1-1)}
 \node[coupon,minimum width=\ww] at (\cc,.95){$#2$};
 \end{tikzpicture}}
\newcommand{\zsixstate}[4]{\begin{tikzpicture}[scale=.53,every node/.style={font=\scriptsize}]
 \begin{scope}[shift={(0,3.05)},yscale=-1]
  \draw[m](0,0)--(.45,.65)--(.9,0);
  \edge{#1}{(.45,.65)--(1.3,1.15)}
  \draw[m](1.8,0)--(1.3,1.15)--(2.15,1.65)--(2.7,0);
  \edge{#2}{(2.15,1.65)--(3,2.15)}
  \draw[m](3.6,0)--(3,2.15)--(3.85,2.65)--(4.5,0);
  \edge{#3}{(3.85,2.65)--(3.85,3.05)}
  \foreach \p in {(.45,.65),(1.3,1.15),(2.15,1.65),(3,2.15),(3.85,2.65)}{\node[dot] at \p{};}
  \node[below left,inner sep=1pt] at (.875,.9){$#1$};
  \node[below left,inner sep=1pt] at (2.575,1.9){$#2$};
  \node[below] at (3.85,3.05){$#3$};
 \end{scope}
 \ifstrequal{#4}{none}{}{\foreach \k [count=\h from 0] in {#4}{
  \pgfmathsetmacro{\yy}{3.05+.8*\h}
  \foreach \j in {0,...,5}{\draw[m](.9*\j,\yy)--(.9*\j,\yy+.8);}
  \pgfmathsetmacro{\cc}{.9*(\k-1)+.45}
  \node[coupon,minimum width=.66cm] at (\cc,\yy+.4){$s$};}}
 \end{tikzpicture}}
\newcommand{\zpairprojector}[1]{\begin{tikzpicture}[scale=.63,every node/.style={font=\scriptsize}]
 \foreach \x in {0,1.5}{
  \draw[m](\x,0)--(\x+.35,.35)--(\x+.7,0);
  \draw[gray,densely dotted](\x+.35,.35)--(\x+.35,.95);
  \draw[m](\x,1.3)--(\x+.35,.95)--(\x+.7,1.3);
  \node[dot] at (\x+.35,.35){};\node[dot] at (\x+.35,.95){};
  \node[right] at (\x+.35,.65){$1$};}
 \draw[m](3,0)--(3.35,.35)--(3.7,0);
 \edge{#1}{(3.35,.35)--(3.35,.95)}
 \draw[m](3,1.3)--(3.35,.95)--(3.7,1.3);
 \node[dot] at (3.35,.35){};\node[dot] at (3.35,.95){};
 \node[right] at (3.35,.65){$#1$};
 \end{tikzpicture}}
\newcommand{\zinversecross}{\begin{tikzpicture}[scale=.8]
 \draw[m](-.5,0)--(.5,1.1);
 \draw[m,over](.5,0)--(-.5,1.1);
 \end{tikzpicture}}
\newcommand{\zsignedbraid}[1]{\begin{tikzpicture}[scale=.7]
 \foreach \k [count=\h from 0] in {#1}{
  \pgfmathsetmacro{\yy}{.8*\h}
  \ifnum\k=1
   \draw[m](.8,\yy)--(0,\yy+.8);
   \draw[m,over](0,\yy)--(.8,\yy+.8);
  \else
   \draw[m](0,\yy)--(.8,\yy+.8);
   \draw[m,over](.8,\yy)--(0,\yy+.8);
  \fi}
 \end{tikzpicture}}
\newcommand{\zbraidedstate}[1]{\begin{tikzpicture}[scale=.7,every node/.style={font=\scriptsize}]
 \draw[m](-.9,1.5)--(-.45,1)--(0,1.5);
 \edge{#1}{(-.45,1)--(0,.45)}
 \draw[m](.9,1.5)--(0,.45)--(0,0);
 \node[dot] at (-.45,1){};\node[dot] at (0,.45){};
 \node[left] at (-.25,.75){$#1$};\node[below] at (0,0){$m$};
 \foreach \h in {0,1}{\pgfmathsetmacro{\yy}{1.5+.75*\h}
  \draw[m](-.9,\yy)--(-.9,\yy+.75);
  \draw[m](.9,\yy)--(0,\yy+.75);
  \draw[m,over](0,\yy)--(.9,\yy+.75);}
 \end{tikzpicture}}

\newcommand{\cftree}[8]{\begin{tikzpicture}[scale=.55,
 every node/.style={font=\scriptsize}]
 \ifcase#1
  \edge{#2}{(-1.2,1.9)--(-.8,1.45)}\edge{#3}{(-.4,1.9)--(-.8,1.45)}
  \edge{#6}{(-.8,1.45)--(-.4,.9)}\edge{#4}{(.4,1.9)--(-.4,.9)}
  \edge{#7}{(-.4,.9)--(0,.35)}\edge{#5}{(1.2,1.9)--(0,.35)}
  \node[left] at (-.63,1.15){$#6$};\node[left] at (-.2,.62){$#7$};
  \node[dot] at (-.8,1.45){};\node[dot] at (-.4,.9){};
 \or
  \edge{#3}{(-.4,1.9)--(0,1.45)}\edge{#4}{(.4,1.9)--(0,1.45)}
  \edge{#6}{(0,1.45)--(-.4,.9)}\edge{#2}{(-1.2,1.9)--(-.4,.9)}
  \edge{#7}{(-.4,.9)--(0,.35)}\edge{#5}{(1.2,1.9)--(0,.35)}
  \node[right] at (-.16,1.18){$#6$};\node[left] at (-.2,.62){$#7$};
  \node[dot] at (0,1.45){};\node[dot] at (-.4,.9){};
 \or
  \edge{#3}{(-.4,1.9)--(0,1.45)}\edge{#4}{(.4,1.9)--(0,1.45)}
  \edge{#6}{(0,1.45)--(.4,.9)}\edge{#5}{(1.2,1.9)--(.4,.9)}
  \edge{#7}{(.4,.9)--(0,.35)}\edge{#2}{(-1.2,1.9)--(0,.35)}
  \node[left] at (.16,1.18){$#6$};\node[right] at (.2,.62){$#7$};
  \node[dot] at (0,1.45){};\node[dot] at (.4,.9){};
 \or
  \edge{#4}{(.4,1.9)--(.8,1.45)}\edge{#5}{(1.2,1.9)--(.8,1.45)}
  \edge{#6}{(.8,1.45)--(.4,.9)}\edge{#3}{(-.4,1.9)--(.4,.9)}
  \edge{#7}{(.4,.9)--(0,.35)}\edge{#2}{(-1.2,1.9)--(0,.35)}
  \node[right] at (.63,1.15){$#6$};\node[right] at (.2,.62){$#7$};
  \node[dot] at (.8,1.45){};\node[dot] at (.4,.9){};
 \or
  \edge{#2}{(-1.2,1.9)--(-.8,1.3)}\edge{#3}{(-.4,1.9)--(-.8,1.3)}
  \edge{#4}{(.4,1.9)--(.8,1.3)}\edge{#5}{(1.2,1.9)--(.8,1.3)}
  \edge{#6}{(-.8,1.3)--(0,.35)}\edge{#7}{(.8,1.3)--(0,.35)}
  \node[left] at (-.47,.88){$#6$};\node[right] at (.47,.88){$#7$};
  \node[dot] at (-.8,1.3){};\node[dot] at (.8,1.3){};
 \fi
 \edge{#8}{(0,.35)--(0,0)}\node[dot] at (0,.35){};
 \node[above] at (-1.2,1.9){$#2$};\node[above] at (-.4,1.9){$#3$};
 \node[above] at (.4,1.9){$#4$};\node[above] at (1.2,1.9){$#5$};
 \node[below] at (0,0){$#8$};
\end{tikzpicture}}
\newcommand{\cpentagon}[1]{\begingroup
 \ifcase#1
  \def\ctA{\cftree{0}{a}{b}{m}{c}{ab}{m}{m}}
  \def\ctB{\cftree{1}{a}{b}{m}{c}{m}{m}{m}}
  \def\ctC{\cftree{2}{a}{b}{m}{c}{m}{m}{m}}
  \def\ctD{\cftree{3}{a}{b}{m}{c}{m}{m}{m}}
  \def\ctE{\cftree{4}{a}{b}{m}{c}{ab}{m}{m}}
  \def\ceA{1}\def\ceB{\chi(a,c)}\def\ceC{\chi(b,c)}
  \def\ceD{\chi(ab,c)}\def\ceE{1}
 \or
  \def\ctA{\cftree{0}{m}{a}{m}{m}{m}{x}{m}}
  \def\ctB{\cftree{1}{m}{a}{m}{m}{m}{x}{m}}
  \def\ctC{\cftree{2}{m}{a}{m}{m}{m}{ay}{m}}
  \def\ctD{\cftree{3}{m}{a}{m}{m}{y}{ay}{m}}
  \def\ctE{\cftree{4}{m}{a}{m}{m}{m}{y}{m}}
  \def\ceA{d(a,x)}\def\ceB{T_{ay,x}}\def\ceC{1}
  \def\ceD{T_{y,x}}\def\ceE{1}
 \or
  \def\ctA{\cftree{0}{m}{m}{a}{m}{x}{xa}{m}}
  \def\ctB{\cftree{1}{m}{m}{a}{m}{m}{xa}{m}}
  \def\ctC{\cftree{2}{m}{m}{a}{m}{m}{y}{m}}
  \def\ctD{\cftree{3}{m}{m}{a}{m}{m}{y}{m}}
  \def\ctE{\cftree{4}{m}{m}{a}{m}{x}{m}{m}}
  \def\ceA{1}\def\ceB{T_{y,xa}}\def\ceC{d(a,y)}
  \def\ceD{1}\def\ceE{T_{y,x}}
 \or
  \def\ctA{\cftree{0}{a}{b}{c}{m}{ab}{abc}{m}}
  \def\ctB{\cftree{1}{a}{b}{c}{m}{bc}{abc}{m}}
  \def\ctC{\cftree{2}{a}{b}{c}{m}{bc}{m}{m}}
  \def\ctD{\cftree{3}{a}{b}{c}{m}{m}{m}{m}}
  \def\ctE{\cftree{4}{a}{b}{c}{m}{ab}{m}{m}}
  \def\ceA{\omega(a,b,c)}
  \def\ceB{u(a,bc)}
  \def\ceC{u(b,c)}
  \def\ceD{u(ab,c)}
  \def\ceE{u(a,b)}
 \fi
 \begin{tikzpicture}[
  every node/.style={font=\small}]
  \node (A) at (-3,.5){\ctA};\node (B) at (0,2.5){\ctB};
  \node (C) at (3,.5){\ctC};\node (D) at (1.9,-1.9){\ctD};
  \node (E) at (-1.9,-1.9){\ctE};
  \draw[arr](A)--node[above left]{$\ceA$}(B);
  \draw[arr](B)--node[above right]{$\ceB$}(C);
  \draw[arr](C)--node[right]{$\ceC$}(D);
  \draw[arr](A)--node[left]{$\ceD$}(E);
  \draw[arr](E)--node[below]{$\ceE$}(D);
 \end{tikzpicture}\endgroup}
\newcommand{\coverlap}[2]{\begin{tikzpicture}[scale=.65,
 every node/.style={font=\scriptsize}]
 \draw[m](0,0)--(0,.35);
 \edge{#1}{(0,.35)--(-.55,.95)}
 \draw[m](-.55,.95)--(-1.1,1.45);\draw[m](-.55,.95)--(0,1.45);
 \draw[m](0,.35)--(1.1,1.45)--(.55,1.95);
 \draw[m](0,1.45)--(.55,1.95);
 \edge{#2}{(.55,1.95)--(0,2.55)}
 \draw[m](-1.1,1.45)--(0,2.55)--(0,2.9);
 \foreach \p in {(0,.35),(-.55,.95),(.55,1.95),(0,2.55)}{\node[dot] at \p{};}
 \node[left] at (-.3,.65){$#1$};\node[right] at (.3,2.25){$#2$};
 \node[below] at (0,0){$m$};\node[above] at (0,2.9){$m$};
\end{tikzpicture}}

\newcommand{\overedge}[2]{\ifboolexpr{test {\ifstrequal{#1}{m}} or test {\ifstrequal{#1}{m^*}} or test {\ifstrequal{#1}{m^{**}}}}{\draw[m,over] #2;}{\draw[a,over] #2;}}
\newcommand{\cupcap}[4][none]{\begin{tikzpicture}[scale=.8,
 every node/.style={font=\scriptsize}]
 \ifstrequal{#2}{cup}{
  \edge{#3}{(-.55,.9) to[out=-90,in=180](0,.2)}
  \edge{#4}{(0,.2) to[out=0,in=-90](.55,.9)}
  \node[above] at (-.55,.9){$#3$};\node[above] at (.55,.9){$#4$};
 }{
  \edge{#3}{(-.55,0) to[out=90,in=180](0,.7)}
  \edge{#4}{(0,.7) to[out=0,in=90](.55,0)}
  \node[below] at (-.55,0){$#3$};\node[below] at (.55,0){$#4$};}
 \ifstrequal{#1}{none}{}{\node[coupon] at (.5,.45){$#1$};}
 \end{tikzpicture}}
\newcommand{\bubble}[4]{\begin{tikzpicture}[scale=.73,
 every node/.style={font=\scriptsize}]
 \edge{#3}{(0,0)--(0,.4)}\edge{#4}{(0,1.4)--(0,1.8)}
 \edge{#1}{(0,.4) to[out=145,in=-145](0,1.4)}
 \edge{#2}{(0,.4) to[out=35,in=-35](0,1.4)}
 \node[dot] at (0,.4){};\node[dot] at (0,1.4){};
 \node[left] at (-.43,.9){$#1$};\node[right] at (.43,.9){$#2$};
 \node[below] at (0,0){$#3$};\node[above] at (0,1.8){$#4$};
 \end{tikzpicture}}
\newcommand{\pivloop}[2]{\begin{tikzpicture}[scale=.72,
 every node/.style={font=\scriptsize}]
 \edge{#1}{(0,.7) ellipse (.48 and .7)}
 \node[left] at (-.48,.7){$#1$};
 \ifstrequal{#1}{m}{\node[right] at (.48,.7){$m^*$};}{\node[right] at (.48,.7){$#1^{-1}$};}
 \ifstrequal{#2}{none}{}{\node[coupon] at (0,1.4){$#2$};}
 \end{tikzpicture}}
\newcommand{\treepair}[5]{\begin{tikzpicture}[scale=.72,
 every node/.style={font=\scriptsize}]
 \path[region](-.8,.25)--(.8,.25)--(0,.95)--cycle;
 \path[fill=treegraylight](-.8,2.3)--(.8,2.3)--(0,1.6)--cycle;
 \foreach \x in {-.7,0,.7}{\draw[a](\x,0)--(\x,.25)--(0,.95);
  \draw[a](0,1.6)--(\x,2.3)--(\x,2.55);}
 \edge{#3}{(0,.95)--(0,1.6)}
 \node[right] at (0,1.28){$#3$};
 \node[fill=treegray,inner sep=1pt] at (0,.48){$#4$};
 \node[fill=treegraylight,inner sep=1pt] at (0,2.07){$#5$};
 \node[below] at (0,0){$#1$};\node[above] at (0,2.55){$#2$};
 \end{tikzpicture}}
\newcommand{\matrixunit}[2]{\zmatrixunit{#1}{#2}}
\newcommand{\stateaction}[7]{\begin{tikzpicture}[scale=.7,
 every node/.style={font=\scriptsize}]
 \edge{#1}{(-.9,1.5)--(-.45,1)}\edge{#2}{(0,1.5)--(-.45,1)}
 \edge{#4}{(-.45,1)--(0,.45)}\edge{#3}{(.9,1.5)--(0,.45)}
 \edge{#5}{(0,.45)--(0,0)}
 \node[dot] at (-.45,1){};\node[dot] at (0,.45){};
 \node[left] at (-.27,.75){$#4$};\node[below] at (0,0){$#5$};
 \edge{#1}{(-.9,1.5)--(-.9,2.3)}\edge{#2}{(0,1.5)--(0,2.3)}
 \edge{#3}{(.9,1.5)--(.9,2.3)}
 \node[coupon,minimum width=.76cm] at ({.9*(#6-1)-.45},1.9){$#7$};
 \node[above] at (-.9,2.3){$#1$};\node[above] at (0,2.3){$#2$};
 \node[above] at (.9,2.3){$#3$};\end{tikzpicture}}
\newcommand{\pairfusion}[3]{\begin{tikzpicture}[scale=.72,
 every node/.style={font=\scriptsize}]
 \draw[m](-1.2,0)--(-.8,.5)--(-.4,0);
 \draw[m](.4,0)--(.8,.5)--(1.2,0);
 \edge{#1}{(-.8,.5)--(0,1.35)}\edge{#2}{(.8,.5)--(0,1.35)}
 \edge{#3}{(0,1.35)--(0,1.8)}
 \foreach \p in {(-.8,.5),(.8,.5),(0,1.35)}{\node[dot] at \p{};}
 \node[left] at (-.5,.9){$#1$};\node[right] at (.5,.9){$#2$};
 \node[above] at (0,1.8){$#3$};
 \foreach \x in {-1.2,-.4,.4,1.2}{\node[below] at (\x,0){$m$};}
 \end{tikzpicture}}
\newcommand{\quotlabel}[2]{\ifstrequal{#1}{#2}{1}{\ifstrequal{#1}{1}{#2}{#1^{-1}#2}}}
\newcommand{\quotedge}[3]{\ifstrequal{#1}{#2}{\edge{1}{#3}}{\ifstrequal{#1}{1}{\edge{#2}{#3}}{\edge{a}{#3}}}}
\newcommand{\fourunit}[3]{\begin{tikzpicture}[scale=.61,
 every node/.style={font=\scriptsize}]
 \foreach \h in {0,1}{\begin{scope}[shift={(0,3.1*\h)},yscale={1-2*\h}]
  \draw[m](-1.2,0)--(-.8,.45)--(-.4,0);
  \draw[m](.4,0)--(.8,.45)--(1.2,0);
  \ifnum\h=0
   \edge{#2}{(-.8,.45)--(0,1.2)}\quotedge{#2}{#3}{(.8,.45)--(0,1.2)}
   \node[left] at (-.55,.78){$#2$};\node[right] at (.55,.78){$\quotlabel{#2}{#3}$};
  \else
   \edge{#1}{(-.8,.45)--(0,1.2)}\quotedge{#1}{#3}{(.8,.45)--(0,1.2)}
   \node[left] at (-.55,.78){$#1$};\node[right] at (.55,.78){$\quotlabel{#1}{#3}$};
  \fi
  \foreach \p in {(-.8,.45),(.8,.45),(0,1.2)}{\node[dot] at \p{};}
 \end{scope}}
 \edge{#3}{(0,1.2)--(0,1.9)}\node[right] at (0,1.55){$#3$};
 \end{tikzpicture}}
\newcommand{\channeltrace}[3]{\begin{tikzpicture}[scale=.72,
 every node/.style={font=\scriptsize}]
 \draw[m](0,-.25)--(0,0)--(.45,.5)--(.9,0);
 \edge{#2}{(.45,.5)--(.45,1.2)}
 \draw[m](0,2.05)--(0,1.75)--(.45,1.2)--(.9,1.75);
 \node[dot] at (.45,.5){};\node[dot] at (.45,1.2){};
 \node[right] at (.45,.85){$#2$};
 \draw[m](.9,1.75) to[out=90,in=90](1.8,1.75)--(1.8,0)
  to[out=-90,in=-90](.9,0);
 \node[coupon,inner sep=1.6pt] at (1.38,2.02){$#3$};
 \node[right] at (1.8,.85){$m^*$};
 \ifnum#1=1
  \draw[m](0,2.05) to[out=90,in=90](2.65,2.05)--(2.65,-.25)
   to[out=-90,in=-90](0,-.25);
  \node[coupon,inner sep=1.6pt] at (.1,2.29){$#3$};
 \else
  \node[below] at (0,-.25){$m$};\node[above] at (0,2.05){$m$};
 \fi\end{tikzpicture}}
\newcommand{\partialchannel}[2]{\channeltrace{0}{#1}{#2}}
\newcommand{\closedchannel}[2]{\channeltrace{1}{#1}{#2}}
\newcommand{\rotatedcup}[1]{\begin{tikzpicture}[scale=.65,
 every node/.style={font=\scriptsize}]
 \draw[m](1.6,2.45)--(1.6,1.05) to[out=-90,in=-90](.8,1.05)
  --(.8,1.85) to[out=90,in=90](0,1.85)--(0,.25)
  to[out=-90,in=-90](2.4,.25)--(2.4,2.45);
 \node[coupon] at (2.4,1.75){$#1$};
 \node[above] at (.4,2.19){$\ev_m$};
 \node[below] at (1.2,-.5){$\coev_{m^*}$};
 \node[below] at (1.2,.76){$i_1$};
 \node[left] at (0,1.1){$m^*$};\node[right] at (2.4,1.1){$m^{**}$};
 \node[above] at (1.6,2.45){$m$};\node[above] at (2.4,2.45){$m$};
 \end{tikzpicture}}
\newcommand{\hexdraw}[6]{%
 \ifnum#1=0
  \edge{#3}{(1,0)--(1.5,.5)}\edge{#4}{(2,0)--(1.5,.5)}
  \edge{#5}{(1.5,.5)--(1.5,.75)}\edge{#2}{(0,0)--(0,.75)}
  \ifnum#6=0
   \edge{#5}{(1.5,.75)--(.5,1.65)}\overedge{#2}{(0,.75)--(2,1.65)}
  \else
   \edge{#2}{(0,.75)--(2,1.65)}\overedge{#5}{(1.5,.75)--(.5,1.65)}\fi
  \edge{#5}{(.5,1.65)--(.5,2.2)}\edge{#2}{(2,1.65)--(2,2.2)}
  \node[dot] at (1.5,.5){};\node[right] at (1.5,.58){$#5$};
 \else
  \ifnum#6=0
   \edge{#3}{(1,0)--(0,.8)}\overedge{#2}{(0,0)--(1,.8)}
   \edge{#4}{(2,0)--(2,.8)--(1,1.4)}\overedge{#2}{(1,.8)--(2,1.4)}
  \else
   \edge{#2}{(0,0)--(1,.8)}\overedge{#3}{(1,0)--(0,.8)}
   \edge{#2}{(1,.8)--(2,1.4)}\overedge{#4}{(2,0)--(2,.8)--(1,1.4)}\fi
  \edge{#3}{(0,.8)--(0,1.4)--(.5,1.9)}
  \edge{#4}{(1,1.4)--(.5,1.9)}\edge{#5}{(.5,1.9)--(.5,2.2)}
  \edge{#2}{(2,1.4)--(2,2.2)}\node[dot] at (.5,1.9){};
 \fi
 \node[below] at (0,0){$#2$};\node[below] at (1,0){$#3$};
 \node[below] at (2,0){$#4$};\node[above] at (.5,2.2){$#5$};
 \node[above] at (2,2.2){$#2$};}
\newcommand{\hexmove}[6]{\begin{tikzpicture}[scale=.67,
 every node/.style={font=\scriptsize}]
 \ifnum#1=0\hexdraw{#2}{#3}{#4}{#5}{#6}{0}
 \else\begin{scope}[xscale=-1]\hexdraw{#2}{#5}{#4}{#3}{#6}{1}\end{scope}\fi
 \end{tikzpicture}}
\newcommand{\fourtree}[4]{\cftree{#1}{m}{m}{m}{m}{#2}{#3}{#4}}
\newcommand{\pentagonpic}{\begin{tikzpicture}[every node/.style={font=\scriptsize}]
 \node (A) at (-3,.8){\fourtree{0}{a}{m}{x}};
 \node (B) at (0,2.6){\fourtree{1}{b}{m}{x}};
 \node (C) at (3,.8){\fourtree{2}{b}{m}{x}};
 \node (D) at (1.9,-1.4){\fourtree{3}{c}{m}{x}};
 \node (E) at (-1.9,-1.4){\fourtree{4}{a}{c}{x}};
 \draw[arr](A)--node[above left]{$F_{b,a}$}(B);
 \draw[arr](B)--node[above right]{$\chi(b,x)$}(C);
 \draw[arr](C)--node[right]{$F_{c,b}$}(D);
 \draw[arr](A)--node[left]{$1$}(E);
 \draw[arr](E)--node[below]{$1$}(D);
 \end{tikzpicture}}

\newcommand{\oppoverlap}[2]{\begin{tikzpicture}[scale=.65,xscale=-1,
 every node/.style={font=\scriptsize}]
 \draw[m](0,0)--(0,.35);
 \edge{#1}{(0,.35)--(-.55,.95)}
 \draw[m](-.55,.95)--(-1.1,1.45);\draw[m](-.55,.95)--(0,1.45);
 \draw[m](0,.35)--(1.1,1.45)--(.55,1.95);
 \draw[m](0,1.45)--(.55,1.95);
 \edge{#2}{(.55,1.95)--(0,2.55)}
 \draw[m](-1.1,1.45)--(0,2.55)--(0,2.9);
 \foreach \p in {(0,.35),(-.55,.95),(.55,1.95),(0,2.55)}{\node[dot] at \p{};}
 \node[left] at (-.3,.65){$#1$};\node[right] at (.3,2.25){$#2$};
 \node[below] at (0,0){$m$};\node[above] at (0,2.9){$m$};
\end{tikzpicture}}

\newcommand{\sameoverlap}[2]{\begin{tikzpicture}[scale=.65,
 every node/.style={font=\scriptsize}]
 \draw[m](0,0)--(0,.35);\edge{#1}{(0,.35)--(-.55,.95)}
 \draw[m](-1.1,1.45)--(-.55,.95)--(0,1.45);
 \draw[m](0,.35)--(1.1,1.45)--(0,2.55)--(0,2.9);
 \draw[m](-1.1,1.45)--(-.55,1.95)--(0,1.45);
 \edge{#2}{(-.55,1.95)--(0,2.55)}
 \foreach \p in {(0,.35),(-.55,.95),(-.55,1.95),(0,2.55)}{\node[dot] at \p{};}
 \node[left] at (-.3,.65){$#1$};\node[left] at (-.3,2.25){$#2$};
 \node[below] at (0,0){$m$};\node[above] at (0,2.9){$m$};\end{tikzpicture}}

\newcommand{\partialbox}[2]{\begin{tikzpicture}[scale=.85]
  \draw[m](0,-.3)--(0,1.8);\draw[m](.7,0)--(.7,1.5);
  \draw[m](.7,1.5) to[out=90,in=90](1.6,1.5)--(1.6,0)
     to[out=-90,in=-90](.7,0);
  \node[coupon,minimum width=1.1cm] at (.35,.75){$#1$};
  \node[coupon] at (.7,1.3){$#2$};
  \node[right,font=\scriptsize] at (1.6,.75){$m^*$};
  \node[above,font=\scriptsize] at (1.15,1.84){$\ev_{m^*}$};
  \node[below,font=\scriptsize] at (1.15,-.32){$\coev_m$};
  \node[below,font=\scriptsize] at (0,-.3){$m$};
  \node[above,font=\scriptsize] at (0,1.8){$m$};
 \end{tikzpicture}}

\newcommand{\drinfeldpic}{\begin{tikzpicture}[scale=.8,every node/.style={font=\scriptsize}]
  \draw[m](0,0)--(0,.65);
  \draw[m](.8,.65) to[out=-90,in=-90](1.6,.65)--(1.6,2.25);
  \draw[m](.8,.65)--(0,1.3);\draw[m,over](0,.65)--(.8,1.3);
  \draw[m](0,1.3) to[out=90,in=90](.8,1.3);
  \node[above] at (.4,1.82){$\ev_m$};\node[below] at (1.2,.18){$\coev_{m^*}$};
  \node[right] at (.8,.61){$m^*$};
  \node[below] at (0,0){$m$};\node[above] at (1.6,2.25){$m^{**}$};
 \end{tikzpicture}}

\newcommand{\positivecurl}{\begin{tikzpicture}[scale=.8,every node/.style={font=\scriptsize}]
  \draw[m](0,0)--(0,.7);
  \draw[m](.8,.7) to[out=-90,in=-90](1.6,.7)--(1.6,1.75);
  \draw[m](.8,.7)--(0,1.35)--(0,2.5);
  \draw[m,over](0,.7)--(.8,1.35);
  \draw[m](.8,1.35)--(.8,1.75) to[out=90,in=90](1.6,1.75);
  \node[coupon] at (.8,1.57){$\rho$};
  \node[right] at (1.6,1.13){$m^*$};
  \node[above] at (1.2,2.12){$\ev_{m^*}$};
  \node[below] at (1.2,.23){$\coev_m$};
  \node[below] at (0,0){$m$};\node[above] at (0,2.5){$m$};
 \end{tikzpicture}}

\newcommand{\rightabsorptionpic}{\begin{tikzpicture}[scale=.75,every node/.style={font=\scriptsize}]
  \draw[m](-.8,0)--(-.8,.8)--(-.2,1.15);
  \draw[a](0,0)--(.4,.55);\draw[m](.8,0)--(.4,.55)--(.4,.8)--(-.2,1.15);
  \draw[a](-.2,1.15)--(-.2,1.6);
  \draw[m](-.8,2.1)--(-.2,1.6)--(.4,2.1);
  \node[dot] at (.4,.55){};\node[dot] at (-.2,1.15){};\node[dot] at (-.2,1.6){};
  \node[right] at (-.2,1.38){$x$};
  \node[below] at (-.8,0){$m$};\node[below] at (0,0){$a$};\node[below] at (.8,0){$m$};
  \node[above] at (-.8,2.1){$m$};\node[above] at (.4,2.1){$m$};
 \end{tikzpicture}}

\newcommand{\leftabsorptionpic}{\begin{tikzpicture}[scale=.75,every node/.style={font=\scriptsize}]
  \draw[m](-.8,0)--(-.55,.55)--(-.55,2.1);
  \draw[a](0,0)--(-.55,.55);\draw[m](.8,0)--(.8,2.1);
  \node[dot] at (-.55,.55){};
  \node[below] at (-.8,0){$m$};\node[below] at (0,0){$a$};\node[below] at (.8,0){$m$};
  \node[above] at (-.55,2.1){$m$};\node[above] at (.8,2.1){$m$};
 \end{tikzpicture}}

\newcommand{\hopfclosurepic}{\begin{tikzpicture}[scale=1,every node/.style={font=\scriptsize}]
  \draw[a](0,0)--(0,1.5);\draw[a](.65,0)--(.65,1.5);
  \draw[a](.65,1.5) to[out=90,in=90](1.35,1.5)--(1.35,0)
      to[out=-90,in=-90](.65,0);
  \draw[a](0,1.5) to[out=90,in=90](2.05,1.5)--(2.05,0)
      to[out=-90,in=-90](0,0);
  \node[coupon,minimum width=1.2cm,minimum height=.5cm] at (.325,.75){$D_{X,Y}$};
  \node[coupon] at (.65,1.3){$j_Y$};\node[coupon] at (0,1.3){$j_X$};
  \node[left] at (0,.17){$X$};\node[left] at (.65,.17){$Y$};
  \node[right] at (1.35,.75){$Y^*$};\node[right] at (2.05,.75){$X^*$};
 \end{tikzpicture}}

\newcommand{\balancedfusionpic}{\begin{tikzpicture}[scale=.8,every node/.style={font=\scriptsize}]
 \draw[a](-.6,0)--(-.6,.55);\draw[a](.6,0)--(.6,.55);
 \draw[a](.6,.55)--(-.6,1.15);\draw[a,over](-.6,.55)--(.6,1.15);
 \draw[a](.6,1.15)--(-.6,1.75);\draw[a,over](-.6,1.15)--(.6,1.75);
 \draw[a](-.6,1.75)--(0,2.2)--(.6,1.75);\draw[a](0,2.2)--(0,2.6);
 \node[coupon] at (-.6,.22){$\theta_r$};\node[coupon] at (.6,.22){$\theta_s$};
 \node[dot] at (0,2.2){};\node[above] at (0,2.6){$u$};
 \node[below] at (-.6,0){$r$};\node[below] at (.6,0){$s$};
\end{tikzpicture}}

\newcommand{\twistfusionpic}{\begin{tikzpicture}[scale=.8,every node/.style={font=\scriptsize}]
 \draw[a](-.6,0)--(0,1.15)--(.6,0);\draw[a](0,1.15)--(0,2.4);
 \node[dot] at (0,1.15){};\node[coupon] at (0,1.85){$\theta_u$};
 \node[below] at (-.6,0){$r$};\node[below] at (.6,0){$s$};\node[above] at (0,2.4){$u$};
\end{tikzpicture}}

\newcommand{\rightsnakepic}{\begin{tikzpicture}[scale=.65]
  \draw[m](.8,0)--(.8,1.05) to[out=90,in=90](0,1.05)
    --(0,.7) to[out=-90,in=-90](-.8,.7)--(-.8,1.85);
  \node[above,font=\scriptsize] at (.4,1.48){$\ev_m$};
  \node[below,font=\scriptsize] at (-.4,.28){$\coev_m$};
  \node[below,font=\scriptsize] at (.8,0){$m$};
  \node[above,font=\scriptsize] at (-.8,1.85){$m$};
  \node[left,font=\scriptsize] at (0.15,.87){$m^*$};
 \end{tikzpicture}}

\newcommand{\leftsnakepic}{\begin{tikzpicture}[scale=.65]
  \draw[m](-.8,0)--(-.8,1.05) to[out=90,in=90](0,1.05)
    --(0,.7) to[out=-90,in=-90](.8,.7)--(.8,1.85);
  \node[above,font=\scriptsize] at (-.4,1.48){$\ev_m$};
  \node[below,font=\scriptsize] at (.4,.28){$\coev_m$};
  \node[below,font=\scriptsize] at (-.8,0){$m^*$};
  \node[above,font=\scriptsize] at (.8,1.85){$m^*$};
  \node[right,font=\scriptsize] at (0,.87){$m$};
 \end{tikzpicture}}

\newcommand{\isowire}[1]{\begin{tikzpicture}[scale=.7]
 \edge{#1}{(0,0) to[out=90,in=-90](.25,.55) to[out=90,in=-90](0,1.15)}
 \node[below,font=\scriptsize] at (0,0){$#1$};\end{tikzpicture}}
\newcommand{\isofuse}[3]{\begin{tikzpicture}[scale=.78,
 every node/.style={font=\scriptsize}]
 \edge{#1}{(-.48,0) to[out=90,in=-150](.15,.5)}
 \edge{#2}{(.48,0) to[out=90,in=-30](.15,.5)}
 \edge{#3}{(.15,.5) to[out=90,in=-90](0,1)}
 \node[dot] at (.15,.5){};\node[below] at (-.48,0){$#1$};
 \node[below] at (.48,0){$#2$};\node[above] at (0,1){$#3$};\end{tikzpicture}}
\newcommand{\isopivsnake}[1]{\begin{tikzpicture}[scale=.65,
 every node/.style={font=\scriptsize}]
 \begin{scope}[xscale=#1]
  \draw[m](-.8,0)--(-.8,1.05) to[out=90,in=90](0,1.05)
   --(0,.7) to[out=-90,in=-90](.8,.7)--(.8,1.85);
  \node[coupon] at (-.8,.7){$\rho$};\node[coupon] at (.8,1.3){$\rho$};
  \node[below] at (-.8,0){$m$};\node[above] at (.8,1.85){$m$};
 \end{scope}\end{tikzpicture}}
\newcommand{\isoflag}[1]{\ifstrequal{#1}{m}{\rho}{\epsilon_{#1}}}
\newcommand{\bentvertex}[4]{\begin{tikzpicture}[scale=.68,
 every node/.style={font=\scriptsize}]
 \ifcase#1
  \edge{#2}{(-.8,2.2)--(-.8,.55) to[out=-90,in=-90](.2,.55)--(.2,.9)--(.7,1.3)}
  \edge{#3}{(1.2,0)--(1.2,.9)--(.7,1.3)}\edge{#4}{(.7,1.3)--(.7,2.2)}
  \node[dot] at (.7,1.3){};\node[coupon] at (-.8,1.2){$\isoflag{#2}$};
  \node[below] at (-.3,.15){$\widetilde\coev_{#2}$};
  \node[left] at (.2,.95){$#2$};\node[above] at (-.8,2.2){$#2^*$};
  \node[above] at (.7,2.2){$#4$};\node[below] at (1.2,0){$#3$};
 \or
  \edge{#3}{(.8,2.2)--(.8,.55) to[out=-90,in=-90](-.2,.55)--(-.2,.9)--(-.7,1.3)}
  \edge{#2}{(-1.2,0)--(-1.2,.9)--(-.7,1.3)}\edge{#4}{(-.7,1.3)--(-.7,2.2)}
  \node[dot] at (-.7,1.3){};\node[below] at (.3,.15){$\coev_{#3}$};
  \node[right] at (-.2,.95){$#3$};\node[above] at (.8,2.2){$#3^*$};
  \node[above] at (-.7,2.2){$#4$};\node[below] at (-1.2,0){$#2$};
 \or
  \edge{#2}{(-.8,0)--(-.8,1.6) to[out=90,in=90](.2,1.6)--(.2,1.05)--(.7,.5)}
  \edge{#3}{(.7,0)--(.7,.5)}\edge{#4}{(.7,.5)--(1.2,1.05)--(1.2,2.2)}
  \node[dot] at (.7,.5){};\node[coupon] at (-.8,1.05){$\isoflag{#2}$};
  \node[above] at (-.3,1.96){$\widetilde\ev_{#2}$};
  \node[right] at (.2,1.3){$#2^*$};\node[below] at (-.8,0){$#2$};
  \node[below] at (.7,0){$#3$};\node[above] at (1.2,2.2){$#4$};
 \or
  \edge{#3}{(.8,0)--(.8,1.6) to[out=90,in=90](-.2,1.6)--(-.2,1.05)--(-.7,.5)}
  \edge{#2}{(-.7,0)--(-.7,.5)}\edge{#4}{(-.7,.5)--(-1.2,1.05)--(-1.2,2.2)}
  \node[dot] at (-.7,.5){};\node[above] at (.3,1.96){$\ev_{#3}$};
  \node[left] at (-.2,1.3){$#3^*$};\node[below] at (.8,0){$#3$};
  \node[below] at (-.7,0){$#2$};\node[above] at (-1.2,2.2){$#4$};
 \fi\end{tikzpicture}}
\newcommand{\isopivotal}[1]{\begin{tikzpicture}[scale=.7,
 every node/.style={font=\scriptsize}]
 \draw[a](0,0)--(0,2);
 \ifnum#1=0
  \node[coupon] at (0,.55){$f$};\node[coupon] at (0,1.45){$j_Y$};
  \node[right] at (0.2,1){$Y$};
 \else
  \node[coupon] at (0,.55){$j_X$};\node[coupon] at (0,1.45){$f^{**}$};
  \node[right] at (0.2,1){$X^{**}$};\fi
 \node[below] at (0,0){$X$};\node[above] at (0,2){$Y^{**}$};\end{tikzpicture}}

\newcommand{\semipentagon}[4]{\begin{tikzpicture}[
 every node/.style={font=\scriptsize},scale=1.2]
 \node (A) at (-3,0){\cftree{0}{#1}{#2}{#3}{#4}{#1#2}{#1#2#3}{#1#2#3#4}};
 \node (B) at (0,1.9){\cftree{1}{#1}{#2}{#3}{#4}{#2#3}{#1#2#3}{#1#2#3#4}};
 \node (C) at (3,0){\cftree{2}{#1}{#2}{#3}{#4}{#2#3}{#2#3#4}{#1#2#3#4}};
 \node (D) at (1.8,-2){\cftree{3}{#1}{#2}{#3}{#4}{#3#4}{#2#3#4}{#1#2#3#4}};
 \node (E) at (-1.8,-2){\cftree{4}{#1}{#2}{#3}{#4}{#1#2}{#3#4}{#1#2#3#4}};
 \draw[arr](A)--node[above left]{$\omega(#1,#2,#3)$}(B);
 \draw[arr](B)--node[above right]{$\omega(#1,#2#3,#4)$}(C);
 \draw[arr](C)--node[right]{$\omega(#2,#3,#4)$}(D);
 \draw[arr](A)--node[left]{$\omega(#1#2,#3,#4)$}(E);
 \draw[arr](E)--node[below]{$\omega(#1,#2,#3#4)$}(D);
\end{tikzpicture}}

\def\makeautorefname#1#2{\csdef{#1autorefname}{#2}}
\makeautorefname{section}{Section}%
\makeautorefname{subsection}{Section}%
\makeautorefname{subsubsection}{Section}%

\makeautorefname{theorem}{Theorem}%
\makeautorefname{lemma}{Lemma}%
\makeautorefname{proposition}{Proposition}%
\makeautorefname{corollary}{Corollary}%
\makeautorefname{example}{Example}%
\makeautorefname{remark}{Remark}%
\makeautorefname{definition}{Definition}%

\begin{document}
\title[The Tambara--Yamagami calculus and finite semigroups]{The Tambara--Yamagami calculus and finite semigroups}
\author[S. Cao and D. Tubbenhauer]{Super Cao and Daniel Tubbenhauer}
\address{S.C.: The University of Sydney, School of Mathematics and Statistics F07, Office Carslaw 806, NSW 2006, Australia}
\email{super.cao@sydney.edu.au}
\address{D.T.: The University of Sydney, School of Mathematics and Statistics F07, Office Carslaw 827, NSW 2006, Australia, \href{http://www.dtubbenhauer.com}{www.dtubbenhauer.com}, \href{https://orcid.org/0000-0001-7265-5047}{ORCID 0000-0001-7265-5047}}
\email{daniel.tubbenhauer@sydney.edu.au}
\date{}

\begin{abstract}
Tambara--Yamagami categories are among the simplest fusion categories beyond the pointed case: one adds a single noninvertible simple object to a finite group of invertible ones. We ask what remains of their (diagrammatic) calculus when the group is replaced by a finite semigroup, and develop the resulting (diagrammatic) calculus.
\end{abstract}


\addtocontents{toc}{\protect\setcounter{tocdepth}{1}}
\maketitle
\tableofcontents

\section{Introduction}

Pointed fusion categories are controlled by group-theoretic data: the
simple objects form a finite group under tensor product, and the
associator is encoded by a $3$-cocycle. Tambara--Yamagami categories are
among the simplest fusion categories which are not pointed. They add
one noninvertible simple object to this group.

In this paper we ask a basic question: what remains of the
Tambara--Yamagami calculus if the group is replaced by a finite
semigroup?

\subsection{Some history}

A Tambara--Yamagami (TY) category has a finite group $A$ of invertible
simple objects and one further simple object $m$, with
\begin{equation}\label{eq:fusion}
 a\otimes b=ab,\qquad
 a\otimes m=m=m\otimes a,\qquad
 m\otimes m=\bigoplus_{a\in A}a.
\end{equation}
Every product of two simples is multiplicity-free, but $m^{\otimes3}$
contains $|A|$ copies of $m$. Fusing (meaning multiplication) the first two copies of $m$ first
labels these copies by the intermediate summand $a\in A$; fusing the last
two first gives another such basis. The associator is the change-of-basis
matrix between them. Tambara and Yamagami classified these categories over
$\C$ \cite{TaYa-fusion-rules}. The group must be abelian, and the
associators are determined by a symmetric nondegenerate bicharacter and a
scalar whose square is the reciprocal of the group order. In the bases
above, the associator on $m^{\otimes3}$ is, up to this scalar, the character
table of $A$, hence a finite Fourier matrix.

Their simplicity makes TY categories a useful testing ground for other
structures on tensor categories. Braidings were classified in
\cite{Si-braided-near-group}, and Frobenius--Schur indicators were
computed in \cite{Shi-TY-indicators}. TY categories also sit
inside the larger class of near-group categories, where the square of
the noninvertible simple may contain copies of that simple
\cite{EvGa-near-group-doubles,Si-near-group}; see also
\cite{Th-braided-near-group} for braidings. Generalized TY categories
allow several noninvertible simples, whose pairwise products still decompose into invertible simples
\cite{Li-generalized-TY,Na-faithful-gradings}.

These generalizations keep a group of invertible simple objects and
change the noninvertible part of the category. Our question goes in a
different direction. We keep one additional simple object, but weaken
the algebraic structure carried by the other simples: a semigroup
replaces the group, so these objects need no longer be invertible.

\begin{remark}
Before going further, let us fix some language. We use the terminology of
\cite{EtGeNiOs-tensor-categories} and string diagrams throughout, and assume
some familiarity with both. (For accounts emphasizing the diagrammatic
point of view, see e.g. \cite{Tu-qt,TuVi-monoidal-tqft}.) Roughly, tensor product places diagrams side
by side, composition stacks them, and decompositions into simple summands are
recorded by trivalent vertices; see, e.g.,
\cite{JoSt-tensor-calculus,Se-graphical-languages}. Different fusion trees
give different bases, and associators give the corresponding changes of basis. This change of fusion tree is called recoupling.

As the reader might have noticed at this point, we often use terminology
from the physics literature, for example fusion, channels, and recoupling;
see, e.g., \cite{Wa-topological}.

Diagrammatic generators-and-relations presentations have a long history,
going back at least to the 1930s \cite{Br,RTW}. More closely related to the
present paper are the later theories of spiders and webs, e.g.,
\cite{Ku-spiders,MoPeSn-categories-trivalent-vertex,Ya-invariant-graphs} and general constructions in semisimple tensor categories as, e.g., in \cite{Ba-fusion-diagrams,Ya-polygonal}. (For the reader familiar with some of these, none of the pictures below will look surprising.) We also borrow some of
the exposition from \cite{DeTu-dicyclic,LaTu-minimal-webs,RoTu-symmetric-webs}.

Of course, and that is part of the point, the categories we consider are semisimple and very easy. So it is no surprise that several related, and even closely related, constructions (diagrammatic or ``essentially diagrammatic'') are known, maybe too many to cite here.
\end{remark}

Our diagrammatic language is the following. We use black strands for the
(semi)group part and a heavier red strand for the additional object $m$:
\[
\begin{tikzpicture}[]
  \draw[a] (0,0)--(0,1);
  \node[below] at (0,0) {$a$};
  \draw[m] (1.5,0)--(1.5,1);
  \node[below] at (1.5,0) {$m$};
\end{tikzpicture}
,
\qquad
a\in A.
\]
We refer to these as the black and red calculus. The fusion rules, associators, rigidity, braidings etc. are encoded in pictures of the form (here $\omega,d$ and $\chi$ are certain scalars)
\[
 \fuse{r}{s}{rs},\fuse{s}{m}{m},
 \lefttree{r}{s}{t}{rs}{rst}
 =\omega(r,s,t)^{-1}\righttree{r}{s}{t}{st}{rst},
 \pivloop{m}{none}=d_m,
 \doublecross{a}{m}=\chi(a,a)\wire{a}\,\wire{m},
\]
and we get a calculus of (colored) trees, potentially with extra structure.

\begin{remark}
The paper can be read without color: in black and white, the red strand is distinguished by being thicker and lightly shaded.
\end{remark}

\subsection{Why semigroups?}

There is a simple reason to try semigroups: one gets new behavior.
Already the idempotent monoid \(E_2=\{1,e\}\) with \(e^2=e\) shows genuinely
new features. The nonunit \(e\) has no dual, but it can still cross, and there is also a red extension with \(em=me=m,m^2=e\). Both give generally new pictures:
\[
 \crossing{e}{e}=\resolvedcross{e}{e}{e},\quad
 \wire{m}\,\wire{m}=\channel{m}{m}{e}.
\]
For example, having a middle edge labeled \(e\) makes the red object behave locally like a TY object, but its cup and cap
use the local unit \(e\), not the global unit \(1\). The ambient category
is therefore not rigid.

This is useful rather than pathological. Nonrigid semisimple monoidal
categories are much less constrained than fusion categories (already in
rank two, compare \cite{Os-rank-two} with \cite{SuChZh-rank-two}; see also e.g. \cite{ChZh-reconstruction} and,
for recent rigidity results, \cite{EtPe-rigidity}). The
examples here are small enough that one can still calculate explicitly
which parts of the TY calculus survive without rigidity.

\subsection{This paper's contribution}

The black calculus itself works for every finite semigroup: multiplication
trees describe semigroup-graded vector spaces, and associativity remains
invertible. The real question is what happens after adding the extra TY
object.

There is an immediate obstruction to keeping the literal TY rules. As we will see (cf. \autoref{thm:semigroup-obstruction}), the formulas
\(sm=ms=m\) and \(m^2=\bigoplus_{x\in S}x\) force \(S\) to be a group.
If instead one insists on the literal TY rule, 
this failure is already visible for the idempotent monoid from
\((e\otimes m)\otimes m\cong1\oplus e\), whereas
\(e\otimes(m\otimes m)\cong e\oplus e\).

Our main solution is to keep invertible associators but restrict the
channels occurring in \(m^2\). Their support is forced to be a two-sided ideal which is itself a group (a group ideal), with constant multiplicity; see
\autoref{lem:stationary-channels}. We treat multiplicity one. If
\(A\subset S\) is such an ideal, with identity \(e\), then
\(\phi(s)=se=es\) defines a retraction \(\phi:S\to A\). When \(A\) is
abelian and carries TY data, the usual TY coefficients extend to the
ambient semigroup by evaluating black labels through \(\phi\), while the
original semigroup labels remain on the boundary.

\begin{remark}
Another feature of the semigroup setting is that notions standard in
semigroup theory, such as group ideals (see, for example,
\cite{ClPr-semigroups,Gr-structure-semigroups}), appear naturally in the
categorical structure. Such phenomena do not
arise in the usual group-based theory of tensor categories.
\end{remark}

We give a generators-and-relations presentation of the resulting
category and give bases of all morphism
spaces; see \autoref{thm:presentation}. Positive tensor powers of \(m\)
remain entirely inside the local TY subcategory, so their endomorphism
algebras are the ordinary TY algebras, realized directly by the channel
calculus.

The ambient category nevertheless behaves differently. Cups and caps for
\(m\) use the local unit \(e\), so they need not define ambient duals.
Crossings can extend further: we determine exactly when the local TY
braiding extends to the full semigroup category.

\begin{remark}
There is a second way around the obstruction: keep all red-red channels
but allow reassociation to be noninvertible. This leads towards skew
monoidal structures
\cite{LaSt-skew-monoidal,Sz-skew-bialgebroids}. We give one explicit
example for the idempotent monoid $E_2$, but do not pursue this direction further.
\end{remark}

All calculations are over an algebraically closed field of arbitrary
characteristic. The local TY classification in this generality is already
contained in \cite[Example 4.6 and Corollary 4.10]{Li-generalized-TY};
our purpose is to derive it diagrammatically and for
the semigroup extension. 
(Over a nonclosed field, simple objects need not be split. This is a different question; for TY categories over the reals, see e.g. \cite{PlSaSc-real-TY}.)

\begin{remark}
Interestingly, the only field restriction is that \(|A|\) be
nonzero in the ground field; there is no corresponding restriction on
\(|S|\).
\end{remark}

This gives behavior absent over \(\C\). For example, in characteristic
two a TY category for an odd-order group is still semisimple and rigid,
but has zero global dimension and is therefore not separable; compare
\cite{Et-faithful-lifting,EtNiOs-fusion}. We also show that passing to
finite tensor categories does not rescue the literal TY Grothendieck
rules in the excluded characteristics; see
\autoref{prop:TY-splitting}. This contrasts with genuinely
nonsemisimple near-group categories as e.g. in
\cite{EtOs-finite-tensor,Se-nonsemisimple-near-group}.

\subsection{Some interesting new categories}

There are many interesting non-group choices of \(A\subset S\) one could
try. One may adjoin a new global unit to an abelian group \(A\); take
\(A\) to be the periodic group ideal of a finite cyclic monoid; take
\(S=A\times T\) for any finite semigroup \(T\) with zero; or use a
two-level construction \(S=H\sqcup A\), with the action of \(H\) on
\(A\) induced by a homomorphism \(H\to A\). More generally, one can
inflate the elements of \(A\) by adding several ambient labels with the
same image under \(\phi\).

Diagram monoids give particularly concrete candidates (see e.g. \cite{CoTu-mobius,HaRa-partition-algebras,HeTu-affine-diagram,khovanov-monoidal-2024} for such monoids). In rook and
Brauer monoids one can choose \(A\) inside a maximal subgroup at a fixed
idempotent and enlarge it by suitable ambient diagrams. The same idea
can be tried for annular and affine diagram monoids, and for the
Möbius strip diagram monoids, where the additional nonorientable
decorations provide further ambient labels. In such examples the map
\(\phi:S\to A\) often has a direct diagrammatic interpretation by
forgetting, deleting, or closing the parts of the diagram outside the
chosen local sector.
For this construction, the chosen ambient diagrams must form a finite subsemigroup in which the chosen abelian subgroup is a two-sided ideal. Being a maximal subgroup at an idempotent does not by itself ensure this.

These examples suggest that the construction is far from being restricted to ad hoc semigroups (such as adjoining a global unit to $A$). For brevity, we do not discuss these examples further.
\smallskip

\noindent\textbf{Acknowledgments.}
This paper is part of the first author's PhD thesis.
SC was supported by the
Postgraduate Research Scholarship based on the ARC Future Fellowship FT230100489, and DT by the ARC Future Fellowship FT230100489. It remains an open question for DT whether it's a good or bad sign that zero lines of TikZ code were written by hand for this paper.

\noindent\textbf{AI declaration.}
During the preparation of this manuscript, the authors used ChatGPT, Codex, Claude, and Gemini (several mid-2026 models). Interactive drafting was employed for two specific components: mathematical statements were refined through back-and-forth prompts, and TikZ figures were generated by converting the authors' textual descriptions and hand-drawn concepts into LaTeX code. Additionally, the models were used for standard language editing, literature verification, and exploratory discussions. The core ideas and research questions are solely due to the authors, who hold full responsibility for the manuscript's conclusions.

\section{The black strand calculus}\label{sec:pointed}

We first describe the black strands.

\subsection{Conventions}

The following are used throughout.

\begin{itemize}

\item The ground field $k$ is algebraically closed, of arbitrary
characteristic. Categories and functors are $k$-linear, tensor
products are bilinear, and Hom spaces are finite-dimensional.
We use ``fusion category'' also in positive characteristic
\cite[Section 9]{EtNiOs-fusion}. Semigroups are finite and nonempty.

\item We briefly recall the cohomology used below, see \cite[Section 1]{No-semigroup-cohomology} for details. The scalars record reassociation; changing multiplication bases changes them by a coboundary. A $3$-cocycle on a
semigroup $S$, with trivial coefficients $k^\times$, is a function
$\omega:S^{\times3}\to k^\times$ satisfying the pentagon equation
\begin{equation}\label{eq:black-cocycle}
 \omega(s,t,u)\omega(r,st,u)\omega(r,s,t)
 =\omega(rs,t,u)\omega(r,s,tu).
\end{equation}
For a $2$-cochain $z:S^{\times2}\to k^\times$, its coboundary is
\begin{equation}\label{eq:black-gauge}
 dz(r,s,t)=\frac{z(s,t)z(r,st)}{z(rs,t)z(r,s)},
 \qquad \omega'=\omega\,dz.
\end{equation}
The last equality defines when two cocycles are cohomologous.
Their classes form $H^3(S,k^\times)$. For monoids we use
normalized cochains and cocycles: their value is $1$ whenever an
argument is the identity. Every cohomology class has a normalized
representative.

\item We distinguish the product label $rs\in S$ from the two-letter word
$r\otimes s$. However, we often write $mm$ instead of $m\otimes m$.

\item Diagrams are read from bottom to top: in $fg=f\circ g$,
$f$ sits above $g$. Tensor product places diagrams side by side:
\[
\begin{tikzpicture}[]
  \draw[a] (0,0) -- (0,1.8);
  \node[coupon] at (0,.6) {$g$};
  \node[coupon] at (0,1.2) {$f$};
  \node at (0.9,.9) {$=\,fg$,};
%
  \draw[a] (3.35,0) -- (3.35,1.8);
  \draw[a] (4.05,0) -- (4.05,1.8);
  \node[coupon] at (3.35,.9) {$f$};
  \node[coupon] at (4.05,.9) {$g$};
  \node at (5.05,.9) {$=\,f\otimes g$.};
\end{tikzpicture}
\]
With words as objects and juxtaposition as tensor product, the diagram
categories are strict.

\item Associators have direction
$\alpha_{X,Y,Z}:(X\otimes Y)\otimes Z\to X\otimes(Y\otimes Z)$.
Our duality convention follows \cite[Section 3.1]{Shi-TY-indicators}.

\item In the TY sections, $A$ denotes the relevant abelian group and
$N=|A|$, viewed as a scalar in $k$ when it occurs in a formula.
Unless stated otherwise, $N>1$. Complex examples use the positive
square root $\sqrt N$ and write $\nu=\tau\sqrt N\in\{1,-1\}$.

\end{itemize}

\subsection{Semigroup-graded vector spaces}

Fix a semigroup $S$ and a cocycle $\omega$.
An $S$-graded vector space is a finite-dimensional vector space
$V=\bigoplus_{s\in S}V_s$. Its morphisms preserve the grading, so
\[
 \Hom_{\Vec_S^\omega}(V,W)=\bigoplus_{s\in S}\Hom_k(V_s,W_s).
\]
The tensor product is the ordinary vector space tensor product,
graded by multiplication in $S$:
\[
 (V\otimes W)_t=\bigoplus_{rs=t}V_r\otimes W_s.
\]
Thus a tensor of vectors of degrees $r,s$ has degree $rs$.

\begin{definition}
The category $\mathcal V=\Vec_S^\omega$ is this category of graded vector
spaces, with associator on homogeneous vectors
$v\in V_r$, $w\in W_s$ and $z\in Z_t$ given by
\begin{equation}\label{eq:Vec-associator}
 \alpha_{r,s,t}:(v\otimes w)\otimes z
 \longmapsto\omega(r,s,t)\cdot\big(v\otimes(w\otimes z)\big).
\end{equation}
For a monoid its tensor unit is $k$ in degree $1$, with the usual
unit maps. We abbreviate $\Vec_S^1$ to $\Vec_S$.
\end{definition}

\autoref{eq:black-cocycle} says that reassociating four homogeneous
vectors using \autoref{eq:Vec-associator} in either order gives the same result. For a monoid, the
normalization ensures compatibility with the tensor unit.
Hence $\Vec_S^\omega$ is semigroupal (monoidal without unit), and monoidal when $S$ is a monoid.

Write $k_s$ for the one-dimensional space concentrated in degree
$s$. These are the simples, and
\[
 k_r\otimes k_s\cong k_{rs},\qquad
 \Hom_{\mathcal V}(k_r,k_s)=\begin{cases}k,&r=s,\\0,&r\ne s.\end{cases}
\]
Every object is a finite direct sum of these lines.
Choosing one graded line $k_s$ for each $s\in S$, both
$(k_r\otimes k_s)\otimes k_t$ and $k_r\otimes(k_s\otimes k_t)$
are identified with $k_{rst}$. The associator between them is nevertheless
multiplication by $\omega(r,s,t)$, which need not be $1$.

\begin{remark}
The category $\Vec_S^\omega$, after choosing the lines $k_s$ as representatives, is skeletal but not strict (unless $\omega$ is trivial).
\end{remark}

\subsection{Multiplication diagrams}

We now give a presentation by planar diagrams, which is strict, by construction. A strand carries a label in
$S$. The generating morphisms are a straight strand, a merge and
a split:
\[
 \id_s=\wire{s},\qquad
 \mu_{r,s}=\fuse{r}{s}{rs}:r\otimes s\to rs,\qquad
 \Delta_{r,s}=\splitv{r}{s}{rs}:rs\to r\otimes s.
\]
(The inputs at a merge or split are ordered.)

\begin{definition}\label{def:pointed}
The black strand category $\mathcal P=\mathcal P(S,\omega)$
has nonempty words in $S$ as objects. Morphisms are $k$-linear
combinations of diagrams obtained by composing and tensoring the
generators above, modulo progressive planar isotopy (cf. Remark \ref{remark:isotopy}) and the local relations
\begin{equation}\label{eq:pointed-inverse}
 \bubble{r}{s}{rs}{rs}=\wire{rs},\qquad
 \channel{r}{s}{rs}=\wire{r}\,\wire{s},
\end{equation}
\begin{align}
 \lefttree{r}{s}{t}{rs}{rst}
 &=\omega(r,s,t)^{-1}\righttree{r}{s}{t}{st}{rst},
 \label{eq:pointed-assoc}\\[1ex]
 \leftsplit{r}{s}{t}{rs}{rst}
 &=\omega(r,s,t)\rightsplit{r}{s}{t}{st}{rst}.
 \label{eq:pointed-assoc-picture}
\end{align}
Composition stacks diagrams, and tensor product concatenates words
and places diagrams side by side. Thus the category is strict
semigroupal. We write $\mathcal P(S)=\mathcal P(S,1)$.

For a monoid $M$, use words in $M\setminus\{1\}$, including the
empty word, and erase the global unit label (cf. \autoref{eq:pointed-bends}). Set
$\mu_{1,s}=\mu_{s,1}=\Delta_{1,s}=\Delta_{s,1}=\id_s$.
This gives the strict monoidal version $\mathcal P(M,\omega)$.
\end{definition}

\begin{remark}\label{remark:isotopy}
Here progressive isotopy keeps every strand monotone from bottom
to top and preserves the ordered inputs and outputs of every vertex, e.g.:
\[
 \isowire{s}=\wire{s},\qquad
 \isofuse{r}{s}{rs}=\fuse{r}{s}{rs}.
\]
Independent vertices may pass each other in height:
\[
\begin{tikzpicture}[]
  \draw[a] (0,0) -- (0,1.8);
  \draw[a] (.8,0) -- (.8,1.8);
  \node[coupon] at (0,.6) {$f$};
  \node[coupon] at (.8,1.2) {$g$};
  \node[below] at (0,0) {$X$};
  \node[below] at (.8,0) {$Y$};
  \node[above] at (0,1.8) {$X'$};
  \node[above] at (.8,1.8) {$Y'$};
  \node at (1.6,.9) {$=$};
%
  \draw[a] (2.4,0) -- (2.4,1.8);
  \draw[a] (3.2,0) -- (3.2,1.8);
  \node[coupon] at (2.4,1.2) {$f$};
  \node[coupon] at (3.2,.6) {$g$};
  \node[below] at (2.4,0) {$X$};
  \node[below] at (3.2,0) {$Y$};
  \node[above] at (2.4,1.8) {$X'$};
  \node[above] at (3.2,1.8) {$Y'$};
\end{tikzpicture}
.
\]
This is also known as the interchange law.
\end{remark}

\begin{remark}
Note that $P(S,\omega)$ is always strict, independent of $\omega$, but not skeletal.
\end{remark}

The first relations \autoref{eq:pointed-inverse} says that merge and split are inverse maps, and \autoref{eq:pointed-assoc} is the fusion relation.
We will use the following lemma to swap freely between \autoref{eq:pointed-assoc} and \autoref{eq:pointed-assoc-picture}.

\begin{lemma}[Asso and coasso]
Imposing only \autoref{eq:pointed-inverse} and \autoref{eq:pointed-assoc} (or only \autoref{eq:pointed-inverse} and \autoref{eq:pointed-assoc-picture}) gives the same category.
\end{lemma}

\begin{proof}
The splitting relation in
\autoref{eq:pointed-assoc-picture} follows by inverting the fusion
relation using \autoref{eq:pointed-inverse}. The same works the other way around.
\end{proof}

\begin{example}[$\Z/2\Z$ and the idempotent monoid]
\label{ex:order-two-cocycles}
There are two monoids of order two: \(G_2=\{1,g\}\) with \(g^2=1\),
and the idempotent monoid \(E_2=\{1,e\}\) with \(e^2=e\). We will use both
repeatedly.

For a normalized cocycle, only its value on three copies of the nonidentity element can be nontrivial.
For \(G_2\), the cocycle equation gives \(\omega(g,g,g)^2=1\), so
there are two possibilities when \(\operatorname{char}k\neq2\).
Writing \(\lambda=\omega(g,g,g)\), we get
\[
 \lefttree{g}{g}{g}{1}{g}
 =
 \lambda\,\righttree{g}{g}{g}{1}{g}.
\]
For the idempotent monoid \(E_2\), the same equation forces
\(\omega(e,e,e)=1\). Hence
\[
 \lefttree{e}{e}{e}{e}{e}
 =
 \righttree{e}{e}{e}{e}{e},
 \qquad
 \bubble{e}{e}{e}{e}
 =
 \wire{e}.
\]
The second picture is cancellation in \(\End(e)\), not a closed loop.
\end{example}

A fusion tree is a composite of merges with one output,
called its root. Its internal labels are its channels.
Reversing its vertices gives a splitting tree.
For four leaves the splitting relations give
\begin{equation}\label{eq:semigroup-pentagon}
 \semipentagon{r}{s}{t}{u}.
\end{equation}
An arrow labeled by a scalar means that its source tree equals
that scalar times its target tree. 

\begin{lemma}[Pentagon]\label{lem:pointed-pentagon}
The diagram in \autoref{eq:semigroup-pentagon} commutes.
\end{lemma}

\begin{proof}
Equality of the two route
coefficients is exactly \autoref{eq:black-cocycle}.
\end{proof}

\subsection{Equivalence and normal forms}

We now compare the strict diagram category $\mathcal P$ with $\Vec_S^\omega$.

\begin{lemma}[Evaluation]\label{lem:pointed-evaluation}
Sending a word $s_1\cdots s_n$ to the left-associated tensor product
of the lines $k_{s_1},\ldots,k_{s_n}$, and merges and splits to the
standard isomorphisms
\[
 k_r\otimes k_s\rightleftarrows k_{rs},
\]
defines a semigroupal functor
\[
 \mathcal E_\omega:\mathcal P(S,\omega)\longrightarrow\Vec_S^\omega.
\]
For a monoid it is monoidal.
\end{lemma}

\begin{proof}
Cancellation is immediate. For three letters, the two trees are related by
\begin{equation}\label{eq:Vec-associator-picture}
 \alpha_{r,s,t}
 =\omega(r,s,t)\,
 \rightsplit{r}{s}{t}{st}{rst}\circ
 \lefttree{r}{s}{t}{rs}{rst}.
\end{equation}
Writing $P_L,P_R$ for the two fusion trees and $I_L,I_R$ for their
inverses, \autoref{eq:Vec-associator-picture} says
\[
 P_R\alpha I_L=\omega(r,s,t)\id_{rst}.
\]
After identifying the two bracketings with the same word, we obtain
$I_L=\omega I_R$ and $P_L=\omega^{-1}P_R$, exactly the relations
\autoref{eq:pointed-assoc} and \autoref{eq:pointed-assoc-picture}.
The four-letter compatibility is \autoref{lem:pointed-pentagon}.
The tensor maps of $\mathcal E_\omega$ are the corresponding
reassociations; in the monoid case they also include the unit maps.
\end{proof}

For a word $w$, let $\pi(w)$ be its ordered product. We choose the
left comb $\mu_w:w\to\pi(w)$, which merges the letters from
left to right, and write $\Delta_w$ for its inverse. For example,
\[
\mu_{s_1s_2s_3s_4}=
\begin{tikzpicture}[scale=.72,
 every node/.style={font=\scriptsize}]
 \edge{s_1}{(-1.2,0)--(-.8,.45)}
 \edge{s_2}{(-.4,0)--(-.8,.45)}
 \edge{s_1s_2}{(-.8,.45)--(-.4,.9)}
 \edge{s_3}{(.4,0)--(-.4,.9)}
 \edge{s_1s_2s_3}{(-.4,.9)--(0,1.35)}
 \edge{s_4}{(1.2,0)--(0,1.35)}
 \edge{s_1s_2s_3s_4}{(0,1.35)--(0,1.8)}
 \node[dot] at (-.8,.45){};
 \node[dot] at (-.4,.9){};
 \node[dot] at (0,1.35){};
 \node[below] at (-1.2,0){$s_1$};
 \node[below] at (-.4,0){$s_2$};
 \node[below] at (.4,0){$s_3$};
 \node[below] at (1.2,0){$s_4$};
 \node[left] at (-.62,.67){$s_1s_2$};
 \node[left] at (-.18,1.12){$s_1s_2s_3$};
 \node[above] at (0,1.8){$s_1s_2s_3s_4$};
\end{tikzpicture}.
\]
For a monoid put $\pi(\emptyset)=1$ and $\mu_\emptyset=\id_1$.

\begin{lemma}[Multiplication normal form]\label{lem:pointed-normal}
For all allowed words $u,v$,
\begin{equation}\label{eq:pointed-hom}
 \Hom_{\mathcal P}(u,v)=
 \begin{cases}
 k\,\Delta_v\mu_u,&\pi(u)=\pi(v),\\
 0,&\pi(u)\ne\pi(v).
 \end{cases}
\end{equation}
\end{lemma}

\begin{proof}
Every vertex preserves the product of the boundary labels, so the Hom
space is zero when the two products differ. The tree relations change
any fusion tree into the left comb, up to a nonzero scalar; splitting
trees give the inverse changes.

The spaces in \autoref{eq:pointed-hom} contain the generators and are
closed under composition, by cancelling the middle combs, and under
tensor product, by changing the grafted tree back to the chosen comb.
They therefore contain every diagram. Finally,
$\mathcal E_\omega(\Delta_v\mu_u)$ is an isomorphism of graded lines,
so the displayed map is nonzero.
\end{proof}

Let $\Add(\mathcal C)$ denote the category obtained from $\mathcal C$
by adjoining finite direct sums and matrices of morphisms.

\begin{theorem}[Diagrammatic presentation]\label{thm:pointed-equivalence}
Evaluation induces a semigroupal equivalence
\begin{equation}\label{eq:pointed-equivalence}
 \Add\big(\mathcal P(S,\omega)\big)\simeq\Vec_S^\omega,
\end{equation}
which is monoidal for a monoid.
\end{theorem}

\begin{proof}
A direct application of Lemmas \ref{lem:pointed-evaluation} and \ref{lem:pointed-normal}, and the immediate fact that every
graded vector space is a finite direct sum of the lines $k_s$.
\end{proof}

\begin{remark}
The normal form itself is independent of $\omega$. Changing another
tree to the left comb may contribute products of $\omega^{\pm1}$,
but the Hom spaces and their dimensions do not change.
\end{remark}

\subsection{Endomorphism algebras}

The normal form and \autoref{eq:pointed-equivalence} give the endomorphism algebras almost immediately.
Put
 $X_S=\bigoplus_{s\in S}k_s$,
and, for $n\ge1$, define
\[
 \Omega_n(t)=
 \{\mathbf s=(s_1,\ldots,s_n)\in S^n:s_1\cdots s_n=t\},
 \qquad
 d_n(t)=|\Omega_n(t)|.
\]
Thus $d_n(t)$ counts the summands of $X_S^{\otimes n}$ whose product
is $t$. Different tuples are kept separate, even if unit labels erase
to the same diagram.

For $\mathbf s\in\Omega_n(t)$, let
$P_{\mathbf s}:X_S^{\otimes n}\to k_t$ be projection onto the
corresponding summand followed by the chosen comb fusion, and let
$I_{\mathbf s}:k_t\to X_S^{\otimes n}$ be the inverse splitting
followed by summand inclusion.

For $\mathbf u,\mathbf v\in\Omega_n(t)$, put
\begin{equation}\label{eq:semigroup-matrix-pictures}
 E^t_{\mathbf v,\mathbf u}
 =I_{\mathbf v}P_{\mathbf u}
 =\treepair{X_S^n}{X_S^n}{t}{\mathbf u}{\mathbf v}.
\end{equation}
In \autoref{eq:semigroup-matrix-pictures}, the lower tree first projects to the summand indexed by $\mathbf u$
and fuses it to $t$; the upper tree then splits $t$ into the summand
indexed by $\mathbf v$.

\begin{lemma}[Tree matrix units]\label{lem:semigroup-matrix-units}
We have
\[
 E^t_{\mathbf v,\mathbf u}E^q_{\mathbf z,\mathbf w}
 =
 \delta_{t,q}\delta_{\mathbf u,\mathbf z}
 E^t_{\mathbf v,\mathbf w}.
\]
Hence the tree pairs are matrix units.
\end{lemma}

\begin{proof}
Stack the two tree pairs. If the middle roots or tuple labels disagree,
the middle projection and inclusion compose to zero. Otherwise the
middle combs cancel:
\[
 P_{\mathbf u}I_{\mathbf z}
 =\delta_{\mathbf u,\mathbf z}\id_t.
\]
The remaining lower and upper trees are
$I_{\mathbf v}$ and $P_{\mathbf w}$.
\end{proof}

\begin{theorem}[Block decomposition]\label{thm:semigroup-cells}
For every cocycle $\omega$ and $n\ge1$,
\begin{equation}\label{eq:semigroup-end}
 B_n(S):=\End(X_S^{\otimes n})
 \cong
 \bigoplus_{t:d_n(t)>0}\Mat_{d_n(t)}(k).
\end{equation}
In particular, the algebra is independent of $\omega$ up to
isomorphism.
\end{theorem}

\begin{proof}
Fuse every tuple to its product. This gives
\[
 X_S^{\otimes n}
 \cong
 \bigoplus_{t\in S} k_t^{\oplus d_n(t)}.
\]
By \autoref{lem:semigroup-matrix-units}, the maps
$E^t_{\mathbf v,\mathbf u}$ are exactly the matrix units of this
decomposition.
\end{proof}

For every $t$ with $d_n(t)>0$, let
\begin{equation}\label{eq:semigroup-simples}
 L_t=\Hom(k_t,X_S^{\otimes n})
 =
 \bigoplus_{\mathbf s\in\Omega_n(t)}k I_{\mathbf s}.
\end{equation}

\begin{proposition}[Simple modules]\label{cor:semigroup-simples}
Then the $L_t$ are the pairwise nonisomorphic simple left
$B_n(S)$-modules, and
\[
 \dim L_t=d_n(t).
\]
\end{proposition}

\begin{proof}
The root-$t$ block in \autoref{eq:semigroup-end} acts on the column
space spanned by the $I_{\mathbf s}$ with
$\mathbf s\in\Omega_n(t)$, and all other blocks act by zero.
\end{proof}

Cutting a tree pair at its root leaves two independently varying tree indices and the scalar algebra $\End(k_t)=k$ at the cut. This is the following sandwich cellular structure.

\begin{proposition}[Cellularity]\label{prop:semigroup-cellularity}
The tree pairs form a sandwich cellular basis of $B_n(S)$, with roots
$t$ satisfying $d_n(t)>0$, index sets $\Omega_n(t)$, and sandwich
algebras $H_t=k$.

They are also a cellular basis using
 $\big(E^t_{\mathbf v,\mathbf u}\big)^\star
 =E^t_{\mathbf u,\mathbf v}$.
\end{proposition}

\begin{proof}
The multiplication rule in
\autoref{lem:semigroup-matrix-units} acts independently on the two
tree indices, with pairing
\[
 P_{\mathbf u}I_{\mathbf v}
 =\delta_{\mathbf u,\mathbf v}\id_t.
\]
There are no higher terms, so this is the sandwich cell datum of
\cite{Tu-sandwich-cellular}. Swapping the two indices is the usual
matrix transpose and gives the cellular anti-involution of
\cite{GrLe-cellular}.
\end{proof}

\begin{remark}
The transpose above is defined in the chosen comb basis. For general
$\omega$, it need not be reflection of an arbitrary diagram.
\end{remark}

Finally, the block sizes can be read directly from the multiplication
table.

\begin{lemma}[Counting products]\label{lem:semigroup-count}
For $n\ge1$,
\begin{equation}\label{eq:semigroup-count}
 d_1(t)=1,\qquad
 d_{n+1}(t)=\sum_{sr=t}d_n(s),\qquad
 \dim B_n(S)=\sum_t d_n(t)^2.
\end{equation}
\end{lemma}

\begin{proof}
For the recurrence in \autoref{eq:semigroup-count}, choose the product
$s$ of its first $n$ letters and a final letter $r$ with $sr=t$.
The dimension formula follows from
\autoref{thm:semigroup-cells}.
\end{proof}

For a monoid one may also put $d_0(t)=\delta_{t,1}$ and $B_0(M)=k$.
For a semigroup without identity there is no empty tensor power, although
each $B_n(S)$ has its own identity $\id_{X_S^{\otimes n}}$.
The same argument works for any finite sum of boundary words.

\begin{example}[$\Z/2\Z$ and the idempotent monoid]
For \(S=G_2=\{1,g\}\) with \(g^2=1\), exactly half of the words in
\(G_2^n\) have product \(1\) and half have product \(g\). Thus
\[
 B_n(G_2)\cong
 \Mat_{2^{n-1}}(k)\oplus\Mat_{2^{n-1}}(k).
\]

For the idempotent monoid \(E_2=\{1,e\}\) with \(e^2=e\), only the word
\((1,\ldots,1)\) has product \(1\); every other word has product \(e\).
Hence
\[
 B_n(E_2)\cong k\oplus\Mat_{2^n-1}(k),
\]
in contrast to the \(\Z/2\Z\) case.
\end{example}

\subsection{Duals and braidings}\label{sec:semigroup-structures}

Duality needs invertible labels. Crossings do not.

\begin{proposition}[Units and duals]\label{prop:semigroup-structures}
The category $\Vec_S^\omega$ has a tensor unit exactly when $S$ is a
monoid. For a monoid $M$, the simple $k_s$ has a left or right dual
exactly when $s\in M^\times$. Thus $\Vec_M^\omega$ is rigid exactly
when $M$ is a group.
\end{proposition}

\begin{proof}
A tensor unit must be a graded line $k_e$, and the unit maps force
$es=s=se$ for every $s$. Conversely, a normalized cocycle gives the
usual unit maps when $M$ is a monoid.

If $k_t$ is a dual of $k_s$, evaluation and coevaluation require
$ts=1=st$, so $s$ is a unit. Conversely, units have the explicit
duals below.
\end{proof}

For $g\in M^\times$, put $g^*=g^{-1}$ and
$\varepsilon_g=\omega(g,g^{-1},g)$. Define
\begin{equation}\label{eq:pointed-duality}
 \coev_g=\Delta_{g,g^{-1}},\qquad
 \ev_g=\varepsilon_g^{-1}\mu_{g^{-1},g}.
\end{equation}
Thus
\begin{equation}\label{eq:pointed-bends}
 \cupcap{cup}{g}{g^{-1}}
 =\splitv{g}{g^{-1}}{1},
 \qquad
 \cupcap{cap}{g^{-1}}{g}
 =\varepsilon_g^{-1}\fuse{g^{-1}}{g}{1}.
\end{equation}
The label changes from $g$ to $g^{-1}$ when a strand passes through a
bend. (One may add orientations if one wants to record this explicitly.)

\begin{lemma}[Snake identities]\label{lem:pointed-duality}
The maps \autoref{eq:pointed-duality} (the diagrams \autoref{eq:pointed-bends}) satisfy
\[
\begin{tikzpicture}[scale=.75,
  every node/.style={font=\scriptsize}]
 \draw[a] (.6,0)--(.6,1)
   to[out=90,in=90](0,1)--(0,.55)
   to[out=-90,in=-90](-.6,.55)--(-.6,1.6);
 \node[below] at (.6,0){$g$};
 \node[above] at (-.6,1.6){$g$};
 \node[left] at (-0.5,.8){$g^{-1}$};
\end{tikzpicture}
=\wire{g},
\qquad
\begin{tikzpicture}[scale=.75,
  every node/.style={font=\scriptsize}]
 \draw[a] (-.6,0)--(-.6,1)
   to[out=90,in=90](0,1)--(0,.55)
   to[out=-90,in=-90](.6,.55)--(.6,1.6);
 \node[below] at (-.6,0){$g^{-1}$};
 \node[above] at (.6,1.6){$g^{-1}$};
 \node[right] at (0,.8){$g$};
\end{tikzpicture}
=\wire{g^{-1}}.
\]
Hence the full subcategory supported on $M^\times$ is rigid.
\end{lemma}

\begin{proof}
The cocycle equation gives
$\omega(g^{-1},g,g^{-1})=\varepsilon_g^{-1}$.
The first snake has coefficient
$\varepsilon_g^{-1}\omega(g,g^{-1},g)=1$; the second has coefficient
$\varepsilon_g^{-1}\omega(g^{-1},g,g^{-1})^{-1}=1$.
The same construction for $g^{-1}$ gives the opposite dual.
\end{proof}

Turning a trivalent vertex upside down need not give our chosen splitting
vertex exactly. Put
 $\kappa(g,h)=
 \frac{\omega(h,h^{-1},g^{-1})}
      {\omega(g,h,(gh)^{-1})}$.

\begin{lemma}[Turning a merge]\label{lem:pointed-mate}
For units $g,h$, turning the merge $\mu_{g,h}$ upside down gives
\[
\begin{tikzpicture}[]
%
  \coordinate (v) at (0,1.05);
  \draw[a] (-.35,.55) -- (v);
  \draw[a] (.35,.55) -- (v);
  \draw[a] (v) -- (0,1.55);
%
  \draw[a]
    (-.35,.55)
    .. controls (-.50,.20) and (-.90,.20) .. (-1.00,.55)
    -- (-1.00,2.05);
%
  \draw[a]
    (.35,.55)
    .. controls (.20,-.05) and (-1.45,-.05) .. (-1.55,.55)
    -- (-1.55,2.05);
%
  \draw[a]
    (0,1.55)
    .. controls (.25,1.95) and (1.10,1.95) .. (1.20,1.55)
    -- (1.20,0);
  \node[above] at (-1.55,2.05) {$h^{-1}$};
  \node[above] at (-1.00,2.05) {$g^{-1}$};
  \node[below] at (1.20,0) {$(gh)^{-1}$};
  \node at (1.85,1.05) {$=$};
  \node at (2.55,1.05) {$\kappa(g,h)$};
%
  \coordinate (w) at (3.50,1.05);
  \draw[a] (3.50,0) -- (w);
  \draw[a] (w) -- (3.00,1.75);
  \draw[a] (w) -- (4.00,1.75);
  \node[above] at (3.00,1.75) {$h^{-1}$};
  \node[above] at (4.00,1.75) {$g^{-1}$};
  \node[below] at (3.50,0) {$(gh)^{-1}$};
\end{tikzpicture}
.
\]
\end{lemma}

\begin{proof}
Nest the $h$-cup inside the $g$-cup. Reassociating uses
$\alpha_{h,h^{-1},g^{-1}}$ and
$\alpha^{-1}_{g,h,(gh)^{-1}}$, whose quotient is $\kappa(g,h)$.
\end{proof}

We use the star to write the operation displayed in Lemma \ref{lem:pointed-mate}, which then, in algebra, says
 $(\mu_{g,h})^*
 =\kappa(g,h)\Delta_{h^{-1},g^{-1}}$.

\begin{proposition}[Pivotal structure on the units]
On the subcategory supported on $M^\times$, the maps
\[
 j_g=\varepsilon_g^{-1}\id_g:g\longrightarrow g^{**}
\]
define a spherical pivotal structure. Every simple has pivotal
dimension $1$.
\end{proposition}

\begin{proof}
The cocycle equation gives
\[
 \frac{\kappa(g,h)}{\kappa(h^{-1},g^{-1})}
 =\frac{\varepsilon_g\varepsilon_h}{\varepsilon_{gh}},
\]
which is exactly the compatibility of $j$ with tensor products.
The left and right dimensions are both $1$; for example,
\[
 \pivloop{g}{j_g}
 =\ev_{g^{-1}}(j_g\otimes\id_{g^{-1}})\coev_g=1,
\]
which uses the standard picture to illustrate the trace.
\end{proof}

\begin{remark}
For $\omega=1$, all bend factors, $\kappa$ and $j$ are $1$.
A nonunit label has no categorical trace of this kind.
\end{remark}

We now turn to crossings.

\begin{proposition}[Crossings without duals]\label{prop:semigroup-braiding}
If $\Vec_S^\omega$ has a braiding, then $S$ is commutative.
For commutative $S$, its braidings are exactly the functions
$b:S^{\times2}\to k^\times$ satisfying
\begin{equation}\label{eq:semigroup-hexagon}
\begin{aligned}
 b(r,st)&=b(r,s)b(r,t)
 \frac{\omega(s,r,t)}
 {\omega(r,s,t)\omega(s,t,r)},\\
 b(rs,t)&=b(r,t)b(s,t)
 \frac{\omega(r,s,t)\omega(t,r,s)}
 {\omega(r,t,s)}.
\end{aligned}
\end{equation}
The crossing on homogeneous vectors is
$c(v_r\otimes w_s)=b(r,s)w_s\otimes v_r$, or
\[
 \crossing{r}{s}
 =b(r,s)\resolvedcross{r}{s}{rs}.
\]
For a monoid, $b(1,s)=b(s,1)=1$. The braiding is symmetric exactly
when $b(r,s)b(s,r)=1,\forall r,s$.
\end{proposition}

\begin{proof}
The two equations in \autoref{eq:semigroup-hexagon} say that a crossing
may be pulled past either side of a merge:
\[
\begin{aligned}
 \hexmove{0}{0}{r}{s}{t}{st}
 &=\hexmove{0}{1}{r}{s}{t}{st},\\[1ex]
 \hexmove{1}{0}{r}{s}{t}{rs}
 &=\hexmove{1}{1}{r}{s}{t}{rs}.
\end{aligned}
\]
An invertible crossing $k_{rs}\to k_{sr}$ requires $rs=sr$.
It is then multiplication by some nonzero scalar $b(r,s)$ followed
by the ordinary flip.

Moving $r$ past $s$ and $t$ in the first hexagon gives
\[
 \omega(r,s,t)^{-1}b(r,s)\omega(s,r,t)b(r,t)
 \omega(s,t,r)^{-1},
\]
which gives the first equation in \autoref{eq:semigroup-hexagon}.
The other hexagon gives the second. Conversely, these two equations
give the hexagons on homogeneous vectors, hence on all graded spaces.
The unit and symmetry statements follow immediately.
\end{proof}

\begin{remark}
When $\omega=1$, the two equations simply say that $b$ is a
bicharacter. For general $\omega$, commutativity alone is not enough.
For groups, the pairs $(\omega,b)$ are the abelian $3$-cocycles; see
\cite[Section 8.4]{EtGeNiOs-tensor-categories}.
\end{remark}

\begin{example}[$\Z/2\Z$ and the idempotent monoid]
Let \(G_2=\{1,g\}\) with \(g^2=1\). A normalized cocycle is determined by
\(\lambda=\omega(g,g,g)\), with \(\lambda^2=1\). If
\(q=b(g,g)\), the hexagon gives \(q^2=\lambda\), and
\[
 \crossing{g}{g}
 =q\,\resolvedcross{g}{g}{1}.
\]
Here \(g\) is its own dual.

For the idempotent monoid \(E_2=\{1,e\}\) with \(e^2=e\), the cocycle is
trivial and the same hexagon gives \(q=q^2\), hence \(q=1\):
\[
 \crossing{e}{e}
 =\resolvedcross{e}{e}{e}.
\]
Thus its braiding is unique and symmetric, but \(e\) has no dual.
\end{example}

\subsection{Classification}

We finish by explaining which choices of $\omega$ give the same
category. We use monoid cohomology with trivial multiplicative coefficients, see e.g. \cite{CeLe-cohomology-monoids,No-semigroup-cohomology}.

\begin{lemma}[Changing the vertices]\label{lem:black-gauge}
Let $z:S\times S\to k^\times$ be a $2$-cochain and replace the
multiplication vertices by
\[
 \fuse{r}{s}{rs}
 \longmapsto
 z(r,s)\fuse{r}{s}{rs},
 \qquad
 \splitv{r}{s}{rs}
 \longmapsto
 z(r,s)^{-1}\splitv{r}{s}{rs}.
\]
Then the new cocycle is given by \autoref{eq:black-gauge}:
\[
 \omega'=\omega\,dz.
\]
\end{lemma}

\begin{proof}
Compare the two three-leaf trees. The right tree contributes
$z(s,t)z(r,st)$, while the left tree contributes
$z(r,s)^{-1}z(rs,t)^{-1}$. Hence
\[
 \omega'(r,s,t)
 =
 \omega(r,s,t)
 \frac{z(s,t)z(r,st)}
      {z(rs,t)z(r,s)}
 =
 \omega(r,s,t)\,dz(r,s,t).
\]
The cancellation relations are unchanged because a merge and the
matching split are rescaled by reciprocal factors.
\end{proof}

Thus changing the chosen multiplication bases changes $\omega$ only
by a coboundary.

\begin{theorem}[Classification]\label{thm:semigroup-cohomology}
Finite semisimple semigroupal categories with simple labels $S$ and
multiplication
\[
 k_r\otimes k_s\cong k_{rs}
\]
are classified, up to equivalences preserving the labels, by
$H^3(S,k^\times)$.

For a monoid, one uses normalized cocycles and monoidal equivalences.
More generally,
$\Vec_S^\omega$ and $\Vec_T^\eta$ are equivalent precisely when there
is a semigroup isomorphism $f:S\to T$ such that
\[
 [\omega]=f^*[\eta].
\]
\end{theorem}

\begin{proof}
Choose one simple object for each $s\in S$ and a basis of each
one-dimensional multiplication space. Every associator is then a
nonzero scalar $\omega(r,s,t)$, and the pentagon is exactly the
cocycle equation.

Changing the multiplication bases changes $\omega$ by a coboundary,
by \autoref{lem:black-gauge}. Conversely, the tensor maps of an
equivalence give such a change of bases, after relabelling the simples.
This gives the stated classification. In the monoid case the unit
maps may be normalized.
\end{proof}

\begin{example}[$\Z/2\Z$ and the idempotent monoid]
The cocycle calculation in \autoref{ex:order-two-cocycles} gives
\[
 H^3(G_2,k^\times)\cong\{\lambda\in k^\times:\lambda^2=1\},\qquad H^3(E_2,k^\times)=1.
\]
Thus $G_2$ has two classes when $\operatorname{char}k\ne2$ and one in characteristic $2$, while $E_2$ has one.
\end{example}

\begin{remark}[Groups]
For a finite group $G$ this recovers the pointed fusion categories
$\Vec_G^\omega$. With the labels fixed, the classes are
$H^3(G,k^\times)$; forgetting the labels also allows automorphisms of
$G$.
\end{remark}

\section{Adding the red strand}\label{sec:semigroup-question}

The black calculus of \autoref{sec:pointed} works for every finite
semigroup. We now add one new simple object \(m\), drawn as a thicker
red strand.

\subsection{The proposed red calculus}

Tensoring a black label with \(m\) is supposed to leave \(m\) unchanged.
We represent the chosen identifications and their split versions by
\[
 \fuse{s}{m}{m},\quad \splitv{s}{m}{m},
 \qquad
 \fuse{m}{s}{m},\quad \splitv{m}{s}{m}.
\]
We call these absorption vertices. Two red strands are different:
they may fuse through black labels,
\[
 \fuse{m}{m}{x}:m\otimes m\longrightarrow x,
 \qquad
 \splitv{m}{m}{x}:x\longrightarrow m\otimes m.
\]
The black label \(x\) on the internal edge is the channel.

The literal TY rule keeps every black channel once:
\[
 s\otimes m\cong m\cong m\otimes s,
 \qquad
 m\otimes m\cong\bigoplus_{x\in S}x,
 \qquad
 \id_{m\otimes m}=\sum_{x\in S}\channel{m}{m}{x}.
\]
Our main question is whether different fusion trees for these rules can be related by invertible reassociation maps.

\subsection{The full-channel obstruction}\label{sec:semigroup-obstruction}

We start with our two monoids of order two.

\begin{example}[$\Z/2\Z$ and the idempotent monoid]
\label{ex:idempotent-obstruction}
First let \(G_2=\{1,g\}\) with \(g^2=1\). For the word \(gmm\), both
bracketings contain one copy of each simple. At roots \(1\) and \(g\)
the fusion trees are
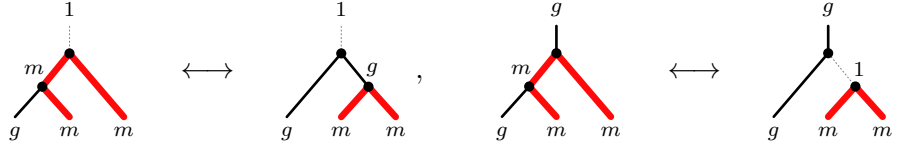
\[
 \lefttree{g}{m}{m}{m}{1}
 \quad\longleftrightarrow\quad
 \righttree{g}{m}{m}{g}{1},
 \qquad
 \lefttree{g}{m}{m}{m}{g}
 \quad\longleftrightarrow\quad
 \righttree{g}{m}{m}{1}{g}.
\]
Thus left multiplication by \(g\) merely swaps the two red-red channels
\(1\) and \(g\).

Now let \(E_2=\{1,e\}\) with \(e^2=e\). The left bracketing still has one
tree at each root, but on the right there is no root-\(1\) tree and there
are two root-\(e\) trees:
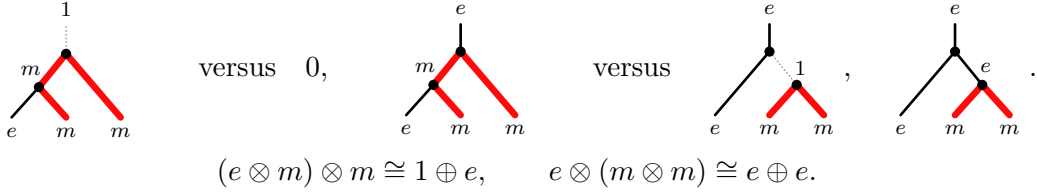
\begin{gather*}
 \lefttree{e}{m}{m}{m}{1}
 \qquad\text{versus}\quad 0,\qquad
 \lefttree{e}{m}{m}{m}{e}
 \qquad\text{versus}\quad
 \righttree{e}{m}{m}{1}{e},
 \quad
 \righttree{e}{m}{m}{e}{e}.
\\
 (e\otimes m)\otimes m\cong1\oplus e,
 \qquad
 e\otimes(m\otimes m)\cong e\oplus e.
\end{gather*}
Hence no invertible change of fusion tree can exist for the idempotent
monoid.
\end{example}

Recall that $S$ is a nonempty finite semigroup.

\begin{theorem}[The full channel rule forces a group]\label{thm:semigroup-obstruction}
Suppose a semisimple $k$-linear category has pairwise nonisomorphic
simple objects $S\sqcup\{m\}$ and an additive tensor product such that
the full subcategory on $S$ is semigroupal and
\[
 s\otimes t=st,\qquad s\otimes m=m=m\otimes s,\qquad
 m\otimes m=\bigoplus_{x\in S}x.
\]
If there are invertible reassociation maps for the triples $smm$ and
$mms$, then $S$ is a group.
\end{theorem}

\begin{proof}
Fix $s,t\in S$. In the two bracketings of $smm$, the root-$t$ trees are
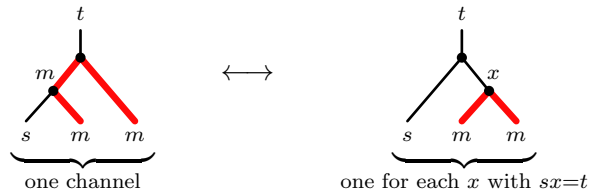
\[
 \underbrace{\lefttree{s}{m}{m}{m}{t}}_{\text{one channel}}
 \qquad\longleftrightarrow\qquad
 \underbrace{\righttree{s}{m}{m}{x}{t}}_
 {\text{one for each }x\text{ with }sx=t}.
\]
Indeed,
\[
 (s\otimes m)\otimes m\cong\bigoplus_{x\in S}x,
 \qquad
 s\otimes(m\otimes m)\cong\bigoplus_{x\in S}sx.
\]
An isomorphism preserves every simple multiplicity. Thus
$|\{x:sx=t\}|=1$ for all $s,t$, so every left translation is bijective.
The same argument for $mms$ makes every right translation bijective.

Choose $s$ and solve $se=s$ and $fs=s$.
From $s(ex)=(se)x=sx$ and $(xf)s=x(fs)=xs$, cancellation gives
$ex=x=xf$ for all $x$. Hence $e=f$ is a two-sided identity.
Solve $xy=e$ and $zx=e$; then $z=z(xy)=(zx)y=y$, so every $x$
has a two-sided inverse.
\end{proof}

\subsection{TY extensions over a semigroup}\label{sec:unitization}

The following is a main concept of this paper.

\begin{definition}
A \emph{TY extension over $S$} consists of a semisimple $k$-linear
category with pairwise nonisomorphic simple objects $S\sqcup\{m\}$ and
additive tensor product such that the full subcategory on $S$ is
semigroupal,
\[
 s\otimes t=st,\qquad s\otimes m\cong m\cong m\otimes s,
 \qquad
 m\otimes m\cong\bigoplus_{x\in S} n_xx,
\]
with $n_x\in\{0,1\}$ not all zero, together with invertible
reassociation maps involving $m$ satisfying the pentagon.
If $S$ is a monoid, we require its identity to be the tensor unit.
\end{definition}

\begin{remark}
\autoref{lem:stationary-channels} below determines the support of $m^2$ and its possible
multiplicities in the case $n_x\in\Z_{\ge0}$ instead of $n_x\in\{0,1\}$. We however decided to study the ``truly semigroupal'' version with $n_x\in\{0,1\}$ instead of the more general $n_x\in\Z_{\ge0}$ version.
\end{remark}

We first determine which black labels can occur in $m^2$. We then classify the local TY data and construct their extension to $S$.

The $smm$ and $mms$ reassociation maps preserve simple multiplicities,
so
\begin{equation}\label{eq:stationary-channels}
 n_t=\sum_{sx=t}n_x=\sum_{xs=t}n_x
 \qquad(s,t\in S).
\end{equation}
A nonempty subset $A\subseteq S$ is a two-sided ideal if $SA\cup AS\subseteq A$. If it is itself a group, it is called a group ideal; see, e.g., \cite{ClPr-semigroups}.

\begin{lemma}\label{lem:stationary-channels}
The nonzero solutions of \autoref{eq:stationary-channels} are exactly
\[
n_x=d\,\delta_{x\in A},
\]
where $d\ge1$ and $A$ is a group ideal of $S$.
\end{lemma}

\begin{proof}
Let $A=\{x:n_x>0\}$. Positivity in
\autoref{eq:stationary-channels} gives $sA=A=As$. Since $S$ is finite,
left and right multiplication by any $s\in S$ therefore permute $A$.
The translation argument of \autoref{thm:semigroup-obstruction},
applied to $A$, shows that $A$ is a group. Its own left translations
are transitive, so all positive multiplicities $n_x$ are equal.

Conversely, let $e$ be the identity of a group ideal $A$. Then
$se=ese=es$, and for $x\in A$,
\[
 sx=(se)x,\qquad xs=x(es).
\]
Thus every left and right translation restricts to a permutation of
$A$, so constant multiplicity on $A$ satisfies
\autoref{eq:stationary-channels}.
\end{proof}

Let $e$ be the identity of the group ideal $A$. Then
\[
 \phi:S\longrightarrow A,\qquad \phi(s)=se=es,
\]
is a retraction: a semigroup homomorphism restricting to the identity on $A$.
Moreover,
\[
 sx=\phi(s)x,\qquad xs=x\phi(s)\qquad(x\in A).
\]
Thus an ambient black label acts on every red-red channel only through
its image under $\phi$.

Suppose that $A$ is abelian and carries ordinary TY data
$(\chi,\tau)$: a symmetric nondegenerate bicharacter
\[
 \chi:A\times A\longrightarrow k^\times
\]
and a scalar $\tau\in k^\times$ with $|A|\tau^2=1$.
These local data and their classification are derived in \autoref{sec:classification}.
The ordinary TY calculus $\TY(A,\chi,\tau)$ on $A$ then extends to the ambient semigroup
by replacing each black input $s$ in a recoupling coefficient by
$\phi(s)$, while keeping $s$ itself on the boundary.

\begin{proposition}[Extension through a group ideal]
\label{prop:group-ideal-extension}
Let $S$ be a finite semigroup with an abelian group ideal $A$, and
choose TY data $(\chi,\tau)$ on $A$. Then there is a TY extension over
$S$ with
\begin{equation}\label{eq:group-ideal-rules}
 st=st,\qquad sm=ms=m,\qquad
 m^2=\bigoplus_{x\in A}x.
\end{equation}
The full subcategory on $A\sqcup\{m\}$ is the ordinary TY category
$\TY(A,\chi,\tau)$, and every positive tensor power of $m$ lies in this
subcategory.
\end{proposition}

\begin{proof}
Let $e$ be the identity of $A$ and put $\phi(s)=se=es$. Since
\[
 \phi(st)=ste=(se)(te)=\phi(s)\phi(t)
\]
and
\[
 sx=\phi(s)x,\qquad xs=x\phi(s)\qquad(x\in A),
\]
all admissible channels involving an ambient black label $s$ are the
same as those for the group element $\phi(s)$.

Use the ordinary TY reassociation maps on $A$, replacing every ambient
black input $s$ by $\phi(s)$. Each mixed pentagon then becomes the
corresponding TY pentagon after applying $\phi$, while the all-black
pentagons are those of the black calculus. Hence these maps define the
required TY extension.

Restricting to $A\sqcup\{m\}$ recovers $\TY(A,\chi,\tau)$, so every
positive tensor power of $m$ remains in this full subcategory.

If $S$ is a monoid, then $\phi(1)=e$; the corresponding unit
associators are identities, so the identity of $S$ is the tensor unit.
\end{proof}

\begin{example}[The idempotent monoid: keep one channel]
\label{ex:idempotent-restricted}
Let \(E_2=\{1,e\}\) with \(e^2=e\). Its group ideal is
\(A=\{e\}\), whose identity is \(e\), and
\[
 \phi(1)=\phi(e)=e.
\]
Thus
\[
 e^2=e,\qquad em=me=m,\qquad m^2=e,
\]
or, diagrammatically,
\[
 \wire{m}\,\wire{m}=\channel{m}{m}{e}.
\]
There is only one red-red channel. The multiplicity mismatch of
\autoref{ex:idempotent-obstruction} has disappeared: both bracketings
have the same simple summands.
\end{example}

\begin{example}[Two copies of the group]
Let
\[
 S=A\times\{0,1\},
 \qquad
 (x,i)(y,j)=(xy,\min(i,j)).
\]
Then $A\times\{0\}$ is a group ideal and
\[
 \phi(x,i)=(x,0).
\]
Thus the two copies of a black label act identically on red channels,
although they remain distinct labels on the boundary. For trivial
$A$, this is the idempotent monoid above.
\end{example}

A useful special case is obtained by adjoining a new global unit to an
ordinary TY category. Let $\mathcal C=\TY(A,\chi,\tau)$, with group
identity $e$, and put
\[
 M=A\sqcup\{1\},
\]
where $1$ is a new identity and the product on $A$ is unchanged.
Thus $A$ is a group ideal in $M$, but its identity $e$ is no longer
the global identity.

Put
\[
 \mathcal C^+=\Vec\oplus\mathcal C,
\]
with simples $\{1\}\sqcup A\sqcup\{m\}$ and no morphisms between $1$
and the old objects. Extend the tensor product by
$1\otimes X=X=X\otimes1$.

\begin{proposition}\label{prop:unitization}
The category $\mathcal C^+$ is a TY extension over $M$, with
\[
 xm=mx=m\quad(x\in M),\qquad
 m^2=\bigoplus_{a\in A}a.
\]
Its tensor unit is the new object $1$. Every positive tensor power of
$m$ lies in $\mathcal C$, and hence has the same endomorphism algebra
in $\mathcal C^+$ as in $\mathcal C$.
\end{proposition}

\begin{proof}
Use the old TY reassociation maps on $\mathcal C$ and identity maps
whenever an input is the new unit. The pentagon is then either an old
TY pentagon or reduces to identities. Positive tensor powers of $m$
contain only objects of $\mathcal C$.
\end{proof}

In this example the old unit $e$ remains visible:
\[
 \wire{m}\,\wire{m}
   =\sum_{a\in A}\channel{m}{m}{a},
 \qquad
 \Delta_e=\splitv{m}{m}{e}:e\longrightarrow mm,
 \qquad
 p_e=\fuse{m}{m}{e}:mm\longrightarrow e.
\]
Thus $p_e\Delta_e=\id_e$, but the $e$-edge cannot be erased: the global
tensor unit is now $1$. This distinction will become important when
strands are bent or closed.

\subsection{A directed alternative}\label{sec:skew-calculus}

There is another way around \autoref{thm:semigroup-obstruction}: keep
all red-red channels, but drop invertibility of reassociation. The
pentagon still makes sense for directed maps
\[
 \alpha_{X,Y,Z}:(XY)Z\longrightarrow X(YZ).
\]
Such directed associators satisfying the pentagon, without unit data, define skew semimonoidal structures; see e.g.
\cite[Section 7]{LaSt-skew-monoidal}. We record the smallest example,
but will not pursue this direction further.

\begin{proposition}[The idempotent monoid with directed reassociation]
\label{prop:skew-idempotent}
Let \(E_2=\{1,e\}\) with \(e^2=e\). On the split semisimple \(k\)-linear
category with simples \(1,e,m\) and rules
\[
 1X=X1=X,\qquad em=me=m,\qquad m^2=1\oplus e,
\]
there are directed reassociation maps satisfying the pentagon.

For \(\eta\in k\) with \(\eta^2=1\), take all maps involving at most one
\(m\) to have coefficient \(1\), and put
\[
 \Pi=\operatorname{diag}(0,1)
\]
in the channel order \(1,e\). The remaining blocks may be chosen as
\[
\begin{array}{c|cc}
 &\text{root }1&\text{root }e\\ \hline
\alpha_{e,m,m}&k\xrightarrow{0}0
 &k\xrightarrow{\binom{0}{1}}k^2\\[1ex]
\alpha_{m,m,e}&0\xrightarrow{0}k
 &k^2\xrightarrow{(0\;1)}k\\[1ex]
\alpha_{m,e,m}&0&1,
\end{array}
\]
together with
\[
 \alpha_{1,m,m}=\alpha_{m,1,m}=\alpha_{m,m,1}=\Pi,
 \qquad
 \alpha_{m,m,m}=\eta\Pi.
\]
Thus reassociation simply kills every \(1\)-channel created by an
\(mm\)-vertex. On the surviving \(e\)-channels the pentagon reduces to
\(\eta^2=1\).
\end{proposition}

\begin{proof}
Every pentagon component outside the surviving \(e\)-channels is zero on
both routes. On the surviving channels all coefficients are \(1\) except
for the four-red pentagon, whose two routes give \(\eta^2\Pi\) and
\(\Pi\). Hence the pentagon is equivalent to \(\eta^2=1\).

For the unit claim, the skew unit axioms would force the reassociation
maps with an input \(1\) to be identities, contradicting the projection
\(\alpha_{1,m,m}=\Pi\). The two displayed decompositions follow directly
from the fusion rules.
\end{proof}

The structure in \autoref{prop:skew-idempotent} is genuinely directed. It has no left skew-monoidal
structure with unit object \(1\), and different parenthesizations need
not even give isomorphic objects. For example,
\[
 \big((mm)m\big)m\cong2\cdot1\oplus2\cdot e,
 \qquad
 (mm)(mm)\cong1\oplus3\cdot e.
\]

Thus the full-channel rule can be retained at the price of
noninvertible reassociation. We will not use this directed structure
below.

\section{Local recoupling}\label{sec:classification}

\aliascntresetthe{theorem}
\aliascntresetthe{lemma}
\aliascntresetthe{proposition}
\aliascntresetthe{corollary}
\aliascntresetthe{definition}
\aliascntresetthe{example}
\aliascntresetthe{remark}

\newcommand{\CS}{\mathcal C_S}
\newcommand{\CA}{\mathcal C_A}
\newcommand{\DS}{\mathcal D(S;A,\chi,\tau)}

\autoref{sec:unitization} showed that, in multiplicity one, the
red-red channels form a group ideal $A\subseteq S$. It also showed
that once the recoupling data on $A\sqcup\{m\}$ are known, they extend
to the ambient semigroup through
\[
 \phi:S\longrightarrow A,\qquad \phi(s)=se=es,
\]
where $e$ is the identity of $A$.

We therefore first solve the local problem on $A\sqcup\{m\}$. This is
essentially the original Tambara--Yamagami calculation
\cite{TaYa-fusion-rules}, now written in our diagrammatic language and
over our ground field. To derive the constraints rather than assume
them, we temporarily write the group as $G$, with identity $1$, and do
not assume that $G$ is abelian. The pentagon will force $G$ to be
abelian and determine the usual TY data $(\chi,\tau)$.
We then return to the group ideal $A$ and extend these coefficients to
the full semigroup through $\phi$.

For the usual TY calculus, the arbitrary-field local
classification is already contained in
\cite[Example 4.6 and Corollary 4.10]{Li-generalized-TY}; the coefficients below are in $k$.
We rederive it in the diagrammatic conventions needed for the
semigroup extension.

\subsection{Vertex bases and matrix conventions}\label{sec:gauge}

Let $G$ be a finite group with $|G|>1$. Suppose that a finite semisimple
$k$-linear monoidal category $\mathcal C$ has simple labels
$\mathcal L_G=G\sqcup\{m\}$ and fusion rules in \autoref{eq:fusion}, with $G$ in place
of $A$ and its identity as tensor unit. 

\begin{remark}
We assume neither rigidity
nor abelianity of $G$.
\end{remark}

Choose one representative of each isomorphism class and take the
unit maps to be identities. The associators are still to be determined;
this categorical model is not assumed strict.

\begin{remark}
Before introducing the notation, it is useful to keep the two basic
pictures in mind. In the strictification, the changes of tree are
\begin{equation*}
\begin{aligned}
 \leftsplit{r}{s}{t}{e}{u}
 =\sum_f F^{rst;u}_{f,e}\,\rightsplit{r}{s}{t}{f}{u},\quad
 \lefttree{r}{s}{t}{e}{u}
 =\sum_f (F^{rst;u})^{-1}_{e,f}\,\righttree{r}{s}{t}{f}{u}.
\end{aligned}
\end{equation*}
Thus reversing the trees interchanges fusion and splitting and inverts
the change-of-basis matrix.

Changing a vertex basis is drawn as
\begin{equation*}
 (p_{rs}^t)'=z_{rs}^t\fuse{r}{s}{t},\qquad
 (i_{rs}^t)'=(z_{rs}^t)^{-1}\splitv{r}{s}{t}.
\end{equation*}
The diagrammatic presentation is proved later; for now these pictures (and similar ones below)
simply record the conventions used in the calculation below.
\end{remark}

The vertices choose the individual summands of a tensor product.
For $r,s\in\mathcal L_G$, let $\Out(r,s)$ be the
set of summands of $r\otimes s$. Each summand has multiplicity one.
Choose projections and inclusions
\[
 p_{rs}^{t}:r\otimes s\longrightarrow t,\qquad
 i_{rs}^{t}:t\longrightarrow r\otimes s
 \qquad(t\in\Out\big(r,s\big))
\]
such that
\begin{equation}\label{eq:resolution}
 p_{rs}^{t}i_{rs}^{u}=\delta_{t,u}\id_t,
 \qquad
 \sum_{t\in\Out(r,s)}i_{rs}^{t}p_{rs}^{t}=\id_{rs}.
\end{equation}
In the first formula a term with $t\ne u$ means the zero map $u\to t$.

Set $p_{mm}^{a}=p_a$ and $i_{mm}^{a}=i_a$.
When either tensor factor is $1$, take
$p_{1s}^{s},i_{1s}^{s},p_{s1}^{s},i_{s1}^{s}$ to be identities.
The maps $p_1:mm\to1$ and $i_1:1\to mm$ remain nontrivial vertices.

For three inputs, the left and right trees give two decompositions.
Let $P_e^L=p_{et}^{u}(p_{rs}^{e}\otimes\id_t)$ and
$P_f^R=p_{rf}^{u}(\id_r\otimes p_{st}^{f})$. Let $I_e^L,I_f^R$
be the reverse composites of the chosen inclusions. Our convention for the associator matrix is
\begin{equation}\label{eq:F-convention}
 F^{rst;u}_{f,e}\id_u=P_f^R\,\alpha_{r,s,t}\,I_e^L,
 \qquad \alpha_{r,s,t}:(rs)t\longrightarrow r(st).
\end{equation}
Thus columns correspond to left channels and rows to right channels.
The associator and its inverse are
\begin{equation}\label{eq:alpha-reconstruction}
 \alpha_{r,s,t}=\sum_{u,e,f}F^{rst;u}_{f,e}I_f^RP_e^L,
 \qquad
 \alpha_{r,s,t}^{-1}=\sum_{u,e,f}(F^{rst;u})^{-1}_{e,f}I_e^LP_f^R.
\end{equation}
After transporting \autoref{eq:alpha-reconstruction} to the word strictification, both bracketings have
the same boundary word. In these pictures the associator is recorded
by a change of tree:
\begin{equation}\label{eq:two-tree-moves}
 I_e^L=\sum_f F^{rst;u}_{f,e}I_f^R,
 \qquad
 P_e^L=\sum_f\bigl((F^{rst;u})^{-1}\bigr)_{e,f}P_f^R.
\end{equation}

\begin{lemma}[Change of vertex bases]\label{prop:gauge}
Choose $z_{rs}^t\in k^\times$ for every admissible vertex, with
$z_{1s}^{s}=z_{s1}^{s}=1$. Under
$p_{rs}^t\mapsto z_{rs}^t p_{rs}^t$ and
$i_{rs}^t\mapsto(z_{rs}^t)^{-1}i_{rs}^t$, the coefficients become
\begin{equation}\label{eq:gauge-F}
 (F')^{rst;u}_{f,e}
 =\frac{z_{st}^{f}z_{rf}^{u}}{z_{rs}^{e}z_{et}^{u}}F^{rst;u}_{f,e}.
\end{equation}
\end{lemma}

\begin{proof}
In \autoref{eq:F-convention}, $P_f^R$ contributes
$z_{st}^{f}z_{rf}^{u}$ and $I_e^L$ contributes
$(z_{rs}^{e}z_{et}^{u})^{-1}$. The factors cancel in
$p_{rs}^ti_{rs}^t$, so \autoref{eq:resolution} is preserved.
\end{proof}

\subsection{Normalizing the vertices}

Name the eight coefficient types before choosing the bases:
\begin{equation}\label{eq:cl-raw}
\begin{gathered}
 \omega(a,b,c)=F^{abc;abc},\quad
 u(a,b)=F^{abm;m},\quad v(a,b)=F^{mab;m},\quad
 \chi(a,b)=F^{amb;m},\\
 B(a,x)=F^{amm;x},\quad d(a,x)=F^{mam;x},\quad
 E(a,x)=F^{mma;x},\quad T_{y,x}=F^{mmm;m}_{y,x}.
\end{gathered}
\end{equation}
Here $\omega$ is the black cocycle of \autoref{sec:pointed}
in the chosen vertex bases. The variables $a,b,c,x,y$ lie in $G$,
and we suppress the unique channel indices of scalar entries. Each coefficient other than an entry of $T$ is nonzero.
The matrix $T$ is invertible, but some entries could initially be zero.
It occurs because each bracketing of $mmm$ contains $|G|$ copies of
$m$; every other triple is multiplicity-free. At this stage, $\chi$
is just a function.

For the pentagon, we start at $(((rs)t)u)$ with first two internal
labels $e,f$ and end at $r(s(tu))$ with last two internal labels
$k,l$. If $v$ is the total output, the equation is
\begin{equation}\label{eq:component-pentagon}
 \sum_h F^{stu;l}_{k,h}F^{rhu;v}_{l,f}F^{rst;f}_{h,e}
 =F^{rsk;v}_{l,e}F^{etu;v}_{k,f}.
\end{equation}
The sum is over the possible labels of $s\otimes t$; a coefficient
with inadmissible indices is zero. The three factors on the left
reassociate the first three inputs, the root, and the last three
inputs. The right side passes through $(rs)(tu)$.

Next, recall the absorption vertices
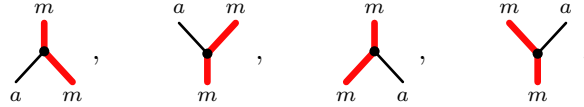
\[
 \fuse{a}{m}{m},\qquad
 \splitv{a}{m}{m},\qquad
 \fuse{m}{a}{m},\qquad
 \splitv{m}{a}{m}.
\]
The pentagon with three black labels and one red label first removes
the freedom in the black cocycle.

\begin{lemma}[Trivializing the black cocycle]
\label{lem:absorption-cocycle}
Choose multiplication and absorption bases, and let
$u(a,b)\in k^\times$ be the coefficient of $\alpha_{a,b,m}$.
Then
\[
 \omega(a,b,c)u(a,bc)u(b,c)=u(ab,c)u(a,b),
 \qquad
 \omega=d(u^{-1}).
\]
Hence rescaling the black vertices by $z(a,b)=u(a,b)$ makes
$\omega=1$ and all coefficients $u(a,b)=1$.
\end{lemma}

\begin{proof}
The $(a,b,c,m)$ pentagon is
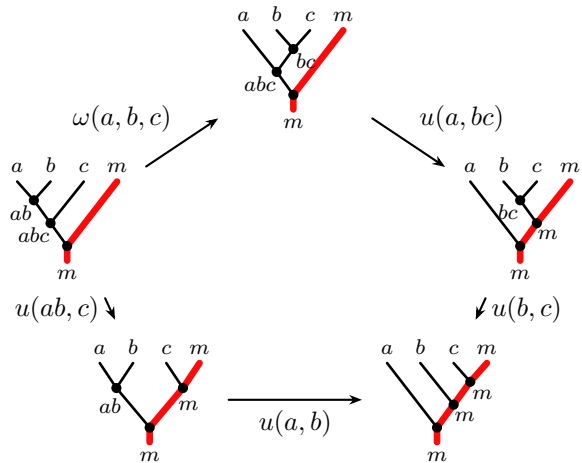
\begin{equation*}
\cpentagon{3}
\end{equation*}
The product along the upper path is
$\omega(a,b,c)u(a,bc)u(b,c)$, while the lower path gives
$u(ab,c)u(a,b)$. Hence
 $\omega(a,b,c)u(a,bc)u(b,c)=u(ab,c)u(a,b)$,
so $\omega=d(u^{-1})$.
Rescaling the black multiplication vertices by $z(a,b)=u(a,b)$ gives
$\omega'=\omega\,du=1$ and $u'(a,b)=1$.
\end{proof}

The remaining rescalings make five of the eight types equal to $1$.

\begin{lemma}[Normal basis]\label{lem:cl-normal}
The vertex bases can be chosen so that
\begin{equation}\label{eq:cl-normal}
 \omega=u=v=B=E=1.
\end{equation}
\end{lemma}
\begin{proof}
All changes below rescale a projection and inversely rescale its
inclusion, as in \autoref{prop:gauge}. First use
\autoref{lem:absorption-cocycle}:
set $z_{ab}^{ab}=u(a,b)$ and leave the other vertex scalars equal
to $1$. This makes both $\omega$ and $u$ equal to $1$.
No assumption on the initial black cocycle is needed.

In these bases, the pentagon on $(a,b,m,m)$ gives
\[
 B(ab,x)=B(a,x)B(b,a^{-1}x).
\]
Put $h(x)=B(x,x)$. Taking $b=a^{-1}x$ shows that
$B(a,x)=h(x)/h(a^{-1}x)$. Rescale the channel vertices by
\begin{equation*}
 p_x'=h(x)\fuse{m}{m}{x},\qquad
 i_x'=h(x)^{-1}\splitv{m}{m}{x}.
\end{equation*}
Then $B'(a,x)=h(a^{-1}x)B(a,x)/h(x)=1$.

The $(a,m,m,b)$ pentagon now reads
$E(b,x)=E(b,a^{-1}x)$, so $E(b,x)=E(b,1)$ is independent of
$x$. Set $z_{mb}^{m}=E(b,1)^{-1}$. This makes $E=1$ and leaves
$\omega=u=B=1$. Finally the $(m,m,a,b)$ pentagon gives $v(a,b)=1$.
Every rescaling respects the unit vertices.
\end{proof}

The two mixed splitting relations in the normal basis are
\begin{equation}\label{eq:cl-mixed-pictures}
\begin{aligned}
 \leftsplit{a}{m}{b}{m}{m}
 =\chi(a,b)\rightsplit{a}{m}{b}{m}{m},\quad
 \leftsplit{m}{a}{m}{m}{x}
 =d(a,x)\rightsplit{m}{a}{m}{m}{x}.
\end{aligned}
\end{equation}
These are associator relations, with no crossings.

\subsection{The mixed pentagons}

The normalized coefficients leave only $\chi$, $d$ and $T$ to
determine.

\begin{lemma}[Pentagon equations]\label{lem:pentagon-table}
In the normal basis \autoref{eq:cl-normal}, the pentagon is equivalent
to the equations in \autoref{tab:pentagon-equations}.
\end{lemma}

\begin{table}[ht]
\centering
\small
\renewcommand{\arraystretch}{1.3}
\begin{tabular}{@{}ll@{}}
\toprule
inputs & pentagon equation\\
\midrule
$abcd,\ abcm,\ mabc$ & $1=1$\\
$abmm,\ ammb,\ mmab$ & $1=1$\\
$abmc$ & $\chi(ab,c)=\chi(a,c)\chi(b,c)$\\
$ambc$ & $\chi(a,bc)=\chi(a,b)\chi(a,c)$\\
$ambm$ & $d(b,x)=\chi(a,b)d(b,a^{-1}x)$\\
$mabm$ & $d(ab,x)=d(a,x)d(b,x)$\\
$mamb$ & $d(a,x)=d(a,xb^{-1})\chi(a,b)$\\
$ammm$ & $T_{y,a^{-1}x}=\chi(a,y)T_{y,x}$\\
$mamm$ & $d(a,x)T_{ay,x}=T_{y,x}$\\
$mmam$ & $T_{y,xa}d(a,y)=T_{y,x}$\\
$mmma$ & $T_{ya^{-1},x}=\chi(x,a)T_{y,x}$\\
$mmmm$ &
$\displaystyle\sum_b T_{c,b}d(b,x)T_{b,a}=\delta_{ac,x}$\\
\bottomrule
\end{tabular}
\caption{Pentagon equations in the normal basis. All group variables
range over $G$. For two or four $m$-inputs, $x$ is the total output;
for three $m$-inputs, $x,y$ are the initial and final group channels.}
\label{tab:pentagon-equations}
\end{table}

\begin{proof}
Substitute the eight types \autoref{eq:cl-raw} into
\autoref{eq:component-pentagon}. A quadruple with zero or four group
inputs has respectively a sum or a single scalar in each output
block. In the other rows, admissibility fixes the intermediate channels.
All normalized coefficients are $1$.
For example, on $(a,b,m,c)$ the long path has coefficients
$1,\chi(a,c),\chi(b,c)$, while the short path has
$\chi(ab,c),1$. On $(m,m,m,m)$, with total output $x$, the long
path sums over the middle-pair channel $b$ and contributes
$T_{c,b}d(b,x)T_{b,a}$. The short path is admissible precisely when
$ac=x$, with both coefficients $1$. The other rows are read in the same way. Together they cover all
$2^4=16$ input patterns.

Here is one of these cases spelled out; the others are similar.
The $(a,b,m,c)$ pentagon has the following splitting trees:
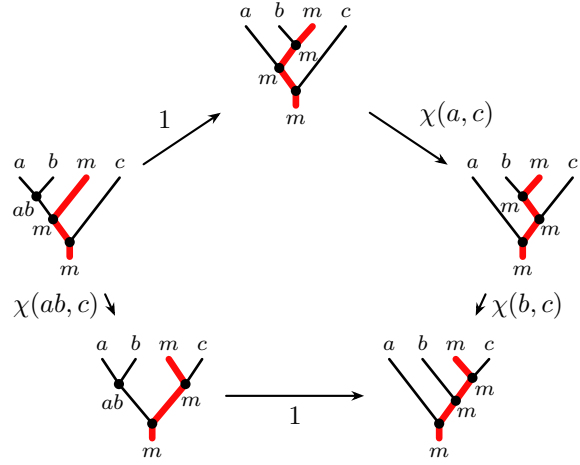
\begin{equation*}
 \cpentagon{0}
\end{equation*}
The product along the upper path is $\chi(a,c)\chi(b,c)$;
along the lower path it is $\chi(ab,c)$. The $(a,m,b,c)$ row
gives multiplicativity in the second variable. Thus $\chi$ is a
bicharacter, normalized by $\chi(1,a)=\chi(a,1)=1$.
\end{proof}

The one-red pentagons make $\chi$ multiplicative. The three-red
pentagons then determine $T$ up to one scalar and force symmetry.

\begin{lemma}[Bicharacter and recoupling matrix]\label{lem:cl-fourier-form}
The function $\chi$ is a symmetric bicharacter. There is a nonzero
scalar $\tau$ such that
\begin{equation}\label{eq:cl-fourier-form}
 T_{y,x}=\tau\chi(x,y)^{-1}.
\end{equation}
Moreover,
\begin{equation}\label{eq:cl-symmetry}
 d(a,x)=\chi(x,a)=\chi(a,x).
\end{equation}
\end{lemma}

\begin{proof}
The $abmc$ and $ambc$ rows of
\autoref{tab:pentagon-equations} make $\chi$ multiplicative in both
variables, and the unit vertices give
$\chi(1,a)=\chi(a,1)=1$.

The $ammm$ and $mmma$ rows give
\[
 T_{y,a}=\chi(a,y)^{-1}T_{y,1},
 \qquad
 T_{a,1}=T_{1,1}.
\]
Hence \autoref{eq:cl-fourier-form} holds with
$\tau=T_{1,1}\neq0$.

The two remaining three-$m$ pentagons are
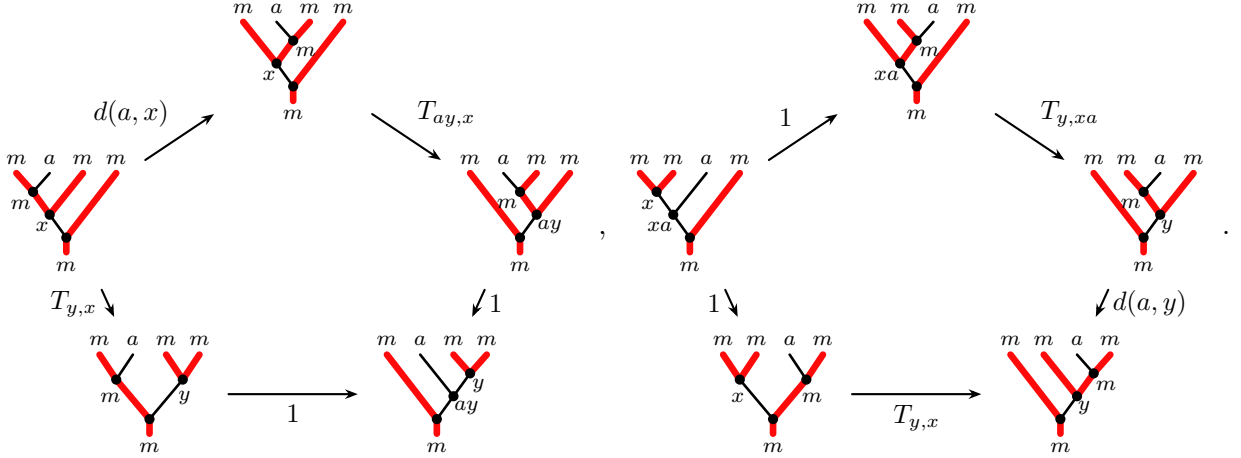
\[
 \cpentagon{1}
,
 \cpentagon{2}.
\]
Reading the two routes gives
\[
 d(a,x)=\frac{T_{y,x}}{T_{ay,x}}=\chi(x,a),
 \qquad
 d(a,y)=\frac{T_{y,x}}{T_{y,xa}}=\chi(a,y).
\]
Thus \autoref{eq:cl-symmetry} holds, and $\chi$ is symmetric.
\end{proof}

\subsection{Nondegeneracy and abelianity}

Invertibility of $T$ means that distinct channels must remain
distinguishable after recoupling.

\begin{proposition}\label{prop:cl-abelian}
The bicharacter $\chi$ is nondegenerate and $G$ is abelian.
\end{proposition}

\begin{proof}
If $\chi(a,b)=\chi(c,b)$ for every $b$, then
\autoref{eq:cl-fourier-form} makes columns $a,c$ of $T$ equal.
Since $T$ is invertible, $a=c$. Thus $\chi$ is nondegenerate.

For abelianity, multiplicativity gives
$\chi(ab,-)=\chi(ba,-)$, hence
\begin{equation*}
 \leftsplit{m}{m}{m}{ab}{m}
 =\sum_{y\in G}\frac{\tau}{\chi(ab,y)}
       \rightsplit{m}{m}{m}{y}{m}
 =\leftsplit{m}{m}{m}{ba}{m}.
\end{equation*}
Contracting with $P_{ab}^L$ gives
\[
 \id_m
 =P_{ab}^LI_{ab}^L
 =P_{ab}^LI_{ba}^L
 =\delta_{ab,ba}\id_m.
\]
Since $m\ne0$, we have $ab=ba$.
\end{proof}

From now on write $A=G$ and $N=|A|$.

\begin{example}[A nonabelian fusion ring]\label{ex:cl-S3}
Take $G=S_3$. Every character $G\to\C^\times$ is trivial on
$h=(123)\in[G,G]$. Hence any bicharacter would satisfy
$\chi(h,b)=\chi(1,b)=1$ for all $b$. The proposed recoupling would
identify
\begin{equation*}
 \leftsplit{m}{m}{m}{h}{m}
 =\sum_{b\in S_3}\tau\,
      \rightsplit{m}{m}{m}{b}{m}
 =\leftsplit{m}{m}{m}{1}{m}.
\end{equation*}
Contract with $P_h^L$: the two sides give $\id_m$ and $0$.
Thus these associative fusion rules have no fusion category
realization over $\C$.
\end{example}

\begin{example}[Nondegeneracy and symmetry]\label{ex:cl-failures}
Work over $\C$.
For $A=G_2=\{1,g\}$, the trivial bicharacter would give
$T=\tau\left(\begin{smallmatrix}1&1\\1&1\end{smallmatrix}\right)$.
Its two columns agree, so recoupling would send $I_1^L-I_g^L$ to zero.
Nondegeneracy does not imply symmetry. On $A=(\mathbb F_2)^2$, let
$\chi(x,y)=(-1)^{x^{\mathsf T}By}$ with
$B=\left(\begin{smallmatrix}1&1\\0&1\end{smallmatrix}\right)$.
This is a nondegenerate bicharacter, but
$\chi(e_1,e_2)=-1$ and $\chi(e_2,e_1)=1$.
The two pentagons in the proof of \autoref{lem:cl-fourier-form} would require $d(e_1,e_2)$ to equal
both $1$ and $-1$. An invertible Fourier matrix alone is therefore
insufficient.
\end{example}

\subsection{The Fourier matrix and the four-red pentagon}

By \autoref{lem:cl-fourier-form}, the three-red recoupling matrix is
\[
 F_{b,a}=\tau\chi(a,b)^{-1}.
\]
We call $F$ the Fourier matrix. (In the normal basis, $F=T$.) Only the four-red pentagon
remains; character orthogonality will determine its normalization.

\begin{lemma}[The Fourier matrix]\label{lem:Fourier}
For $a,c\in A$,
\begin{equation}\label{eq:orthogonality}
 \sum_{b\in A}\chi(ac^{-1},b)=N\delta_{a,c}.
\end{equation}
If $N\tau^2=1$, then
\[
 (F^{-1})_{a,b}=\tau\chi(a,b).
\]
Moreover $F^2=J$, where $J_{c,a}=\delta_{c,a^{-1}}$, and hence
$F^4=I$.
\end{lemma}

\begin{proof}
If $a=c$, every summand in \autoref{eq:orthogonality} is $1$.
Otherwise choose $z$ with $\chi(ac^{-1},z)\ne1$; replacing $b$ by
$zb$ multiplies the sum by this scalar, so the sum is zero.
Under the assumption $N\tau^2=1$, the matrix with entries
$\tau\chi(a,b)$ is therefore the inverse of $F$.
Contracting a splitting tree with its fusion tree expresses the
same identity:
\begin{equation*}
 \sameoverlap{c}{a}
 =\tau^2\sum_b\chi(a,b)\chi(c,b)^{-1}\,\wire{m}
 =\delta_{a,c}\,\wire{m}.
\end{equation*}
Multiplying two copies of $F$
gives $\tau^2\sum_b\chi(ac,b)^{-1}=\delta_{c,a^{-1}}$, proving
$F^2=J$ and $F^4=I$. Two opposite tree moves use $F^{-1}F$, not $F^2$.
\end{proof}

\begin{proposition}[The normalization]\label{prop:cl-tau}
The remaining condition on the normal associators is
$N\tau^2=1$. Conversely, a symmetric nondegenerate bicharacter on a
finite abelian group, together with such a $\tau$, satisfies all
sixteen pentagons in \autoref{tab:pentagon-equations}.
\end{proposition}

\begin{proof}
The four-$m$ pentagon is
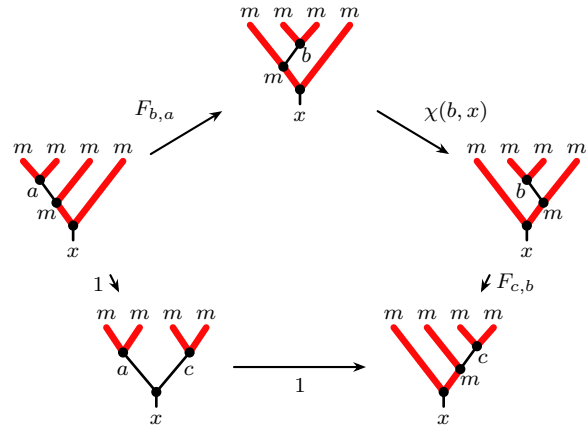
\begin{center}
 \pentagonpic
\end{center}
where $a$ is the initial pair channel, $c$ the final pair channel,
$x$ the total output, and $b$ is summed along the three-step path.

First take $a=c=x=1$. Since $T_{1,b}=T_{b,1}=\tau$ and
$d(b,1)=1$, the two routes give
\begin{equation}\label{eq:cl-tau-forced}
 1=\sum_{b\in A}T_{1,b}d(b,1)T_{b,1}
   =\sum_{b\in A}\tau^2=N\tau^2.
\end{equation}

For arbitrary channels, the three-step path has coefficient
\[
 \sum_b F_{c,b}\chi(b,x)F_{b,a}
 =\tau^2\sum_b
   \chi(a,b)^{-1}\chi(b,x)\chi(b,c)^{-1}
 =\delta_{c,a^{-1}x},
\]
by \autoref{eq:orthogonality}. The other path has coefficient $1$
precisely when $ac=x$, so the two routes agree.

For the converse, multiplicativity checks the one-$m$ and two-$m$
rows of \autoref{tab:pentagon-equations}. Substituting
\[
 T_{y,x}=\frac{\tau}{\chi(x,y)},
 \qquad
 d(a,x)=\chi(a,x)
\]
checks the four three-$m$ rows, while the remaining normalized rows
are identities. Hence all pentagons hold.
In particular, $N\ne0$ in $k$; there are two choices of $\tau$ when
$\operatorname{char}k\ne2$, and one in characteristic $2$.
\end{proof}

\begin{example}[Recoupling for $\Z/3\Z$]\label{ex:Z3-recoupling}
Work over $\C$. Let
\[
 A=\langle g\mid g^3=1\rangle,\qquad
 \zeta=e^{2\pi i/3},\qquad
 \chi(g^j,g^k)=\zeta^{jk},\qquad
 \tau=\frac{\nu}{\sqrt3},
\]
where $\nu\in\{1,-1\}$. In the order $(1,g,g^2)$,
\[
 F=\frac{\nu}{\sqrt3}
 \begin{pmatrix}
 1&1&1\\
 1&\zeta^2&\zeta\\
 1&\zeta&\zeta^2
 \end{pmatrix},
 \qquad
 F^{-1}=\frac{\nu}{\sqrt3}
 \begin{pmatrix}
 1&1&1\\
 1&\zeta&\zeta^2\\
 1&\zeta^2&\zeta
 \end{pmatrix}.
\]

For example, the mixed move is
\[
 \lefttree{g}{m}{g}{m}{m}
 =\zeta^2\righttree{g}{m}{g}{m}{m},
\]
whereas the reversed splitting move has coefficient $\zeta$. Likewise,
\[
 P_g^L=\frac{\nu}{\sqrt3}
       (P_1^R+\zeta P_g^R+\zeta^2P_{g^2}^R),
 \qquad
 I_g^L=\frac{\nu}{\sqrt3}
       (I_1^R+\zeta^2I_g^R+\zeta I_{g^2}^R).
\]
Thus
\[
 P_g^LI_g^L=\frac13(1+1+1)\id_m=\id_m,
\]
whereas using the same Fourier coefficients in both directions would
give $\frac13(1+\zeta+\zeta^2)=0$.

The same cancellation appears in the simplest Fourier bubbles:
\[
 \sameoverlap{g}{g}
 =\frac{1+1+1}{3}\wire{m}
 =\wire{m},
 \qquad
 \sameoverlap{g}{1}
 =\frac{1+\zeta+\zeta^2}{3}\wire{m}
 =0.
\]

Finally, take initial channel $g$ and total output $1$ in the four-$m$
pentagon. Along the long route,
\[
\begin{aligned}
 \fourtree{0}{g}{m}{1}
 &=\frac{\nu}{\sqrt3}\sum_{j=0}^2\zeta^{-j}
       \fourtree{1}{g^j}{m}{1}
 =\frac{\nu}{\sqrt3}\sum_{j=0}^2\zeta^{-j}
       \fourtree{2}{g^j}{m}{1}\\
 &=\sum_{k=0}^2
   \left(\frac13\sum_{j=0}^2\zeta^{-j(1+k)}\right)
       \fourtree{3}{g^k}{m}{1}
 =\fourtree{3}{g^2}{m}{1}.
\end{aligned}
\]
The short route is
\[
 \fourtree{0}{g}{m}{1}
 =\fourtree{4}{g}{g^2}{1}
 =\fourtree{3}{g^2}{m}{1}.
\]
Thus both routes give the same tree with coefficient $1$.
\end{example}

\begin{example}[The cyclic group of order four]
Work over \(\C\).
For $A=\Z/4\Z$ in additive notation, set
$\chi(j,k)=i^{jk}$ and $\tau=\nu/2$. The Fourier matrix is
\begin{equation}\label{eq:Z4}
 F=\frac{\nu}{2}
 \begin{pmatrix}
 1&1&1&1\\1&-i&-1&i\\1&-1&1&-1\\1&i&-1&-i
 \end{pmatrix}.
\end{equation}
Its square is the permutation $j\mapsto-j$, exchanging
$1$ and $3$ and fixing $0,2$. 
\end{example}

We can now summarize the local calculation. The pentagon forces the
group $G$ to be abelian, the function $\chi$ to be a symmetric
nondegenerate bicharacter, and the three-red recoupling matrix to be
$ F_{b,a}=\tau\chi(a,b)^{-1},
 |G|\tau^2=1$.
After our choice of vertex bases, all remaining associators are
determined by this data. Thus we have recovered the usual
Tambara--Yamagami recoupling calculus in our conventions, that we wrap up now.

\subsubsection*{Classification}

The parameters are read from the $amb$ move in
\autoref{eq:cl-mixed-pictures} and the unit-channel overlap
\begin{equation}\label{eq:cl-tau-overlap}
 P_1^R\alpha_{m,m,m}I_1^L=\coverlap{1}{1}=\tau\wire{m}.
\end{equation}
Here $p_1i_1=\id_1$; no choice of cup or cap is involved.

\begin{lemma}\label{lem:cl-invariants}
With the unit vertices fixed, $\chi(a,b)$ and $T_{1,1}$ are invariant
under arbitrary vertex rescalings.
\end{lemma}
\begin{proof}
The rescaling factors in \autoref{eq:gauge-F} are respectively
\begin{equation*}
 \frac{z_{mb}^{m}z_{am}^{m}}{z_{am}^{m}z_{mb}^{m}}=1,
 \qquad
 \frac{z_{mm}^{1}z_{m1}^{m}}{z_{mm}^{1}z_{1m}^{m}}=1.
\end{equation*}
Use $z_{m1}^{m}=z_{1m}^{m}=1$ in the second factor. The projection
and inclusion in \autoref{eq:cl-tau-overlap} acquire reciprocal scalars,
so vertex rescaling cannot remove $\tau$.
\end{proof}

\begin{theorem}[Classification]\label{thm:cl-classification}
Let $|G|>1$. The rules in \autoref{eq:fusion}, with the group identity as
tensor unit, admit a finite semisimple $k$-linear monoidal category
precisely when $G$ is abelian and $|G|\ne0$ in $k$.
Every such category is automatically rigid and is monoidally
equivalent to $\TY(A,\chi,\tau)$ for a symmetric nondegenerate
bicharacter $\chi$ and $|A|\tau^2=1$.

Two data give $k$-linearly monoidally
equivalent categories if and only if there is a group isomorphism
$f:A\to A'$ such that
\begin{equation}\label{eq:cl-isometry}
 \chi'\big(f(a),f(b)\big)=\chi(a,b)\quad(a,b\in A),\qquad \tau'=\tau.
\end{equation}
\end{theorem}
\begin{proof}
\autoref{lem:cl-normal} and \autoref{lem:cl-fourier-form}, together
with \autoref{prop:cl-abelian} and \autoref{prop:cl-tau}, prove
necessity and determine all eight coefficient types in the normal basis.

Conversely, the displayed formulas define associators on the semisimple
category with simple labels $A\sqcup\{m\}$. The scalar and
permutation blocks are invertible. The remaining block has inverse
$(T^{-1})_{x,y}=\tau\chi(x,y)$ by \autoref{lem:Fourier}.
\autoref{prop:cl-tau} proves the pentagon, and the unit
conventions prove the triangle.

For rigidity, take $a^*=a^{-1}$,
$m^*=m$ and
\[
 \coev_a=i_{a,a^{-1}}^1,\quad \ev_a=p_{a^{-1},a}^1,
 \qquad \coev_m=i_1,\quad \ev_m=\tau^{-1}p_1.
\]
The group-labelled snakes have coefficient $1$. The two $m$-snakes
have coefficients $\tau^{-1}T_{1,1}$ and
$\tau^{-1}(T^{-1})_{1,1}$, both equal to $1$. Duality extends to
finite direct sums, so this is a fusion category over $k$.

To obtain such data whenever $|A|\ne0$ in $k$, write $A=\prod_j\Z/n_j\Z$ and choose a
primitive $n_j$-th root $\zeta_j\in k$. Such roots exist because
$k$ is algebraically closed and each $n_j$ is prime to its characteristic.
Then
 $\chi(x,y)=\prod_j\zeta_j^{x_jy_j}$
is symmetric and nondegenerate. Algebraic closure also supplies
$\tau$ with $|A|\tau^2=1$.

A monoidal equivalence must send invertible simples to invertible
simples, inducing $f:A\to A'$, and must send the unique noninvertible
simple $m$ to $m'$. After this relabelling, its effect on the
one-dimensional vertex spaces is a collection of nonzero scalars.
The unit can be normalized, so \autoref{lem:cl-invariants} forces
\autoref{eq:cl-isometry}. Conversely, such an $f$ relabels every vertex
and channel. All associator coefficients then agree, giving a
monoidal equivalence.
\end{proof}

For fixed $A$, we therefore take symmetric nondegenerate
bicharacters up to $\operatorname{Aut}(A)$, together with one choice
of $\tau$ in characteristic $2$ and two otherwise, as in
\cite{TaYa-fusion-rules} and \cite[Example 4.6]{Li-generalized-TY}.

When $A=\{1\}$, the same calculation gives $\tau^2=1$ and
$\alpha_{m,m,m}=\tau\id_m$. Here $m^2=1$, giving one pointed category
on $\Z/2\Z$ in characteristic $2$, and two otherwise. In the equivalence
argument, $m$ is the unique nonunit simple.

\begin{example}[The two order-two data]
Work over $\C$.
For $A=G_2=\{1,g\}$ with the nondegenerate bicharacter $\chi(g,g)=-1$ we have
\begin{equation*}
 T=\frac{\nu}{\sqrt2}
   \begin{pmatrix}1&1\\1&-1\end{pmatrix},\qquad
 \coverlap{1}{1}=\frac{\nu}{\sqrt2}\wire{m},\qquad \nu=\pm1.
\end{equation*}
Thus there are two monoidal classes, distinguished by the displayed
overlap. 
\end{example}

\begin{example}[Relabelling a cyclic group]\label{ex:cl-cyclic}
Work over $\C$.
Write $A=\Z/n\Z$ additively and put $\zeta=e^{2\pi i/n}$.
Every bicharacter is $\chi_k(a,b)=\zeta^{kab}$; it is
nondegenerate exactly when $\gcd(k,n)=1$. The automorphism
$f(a)=ua$, with $u\in(\Z/n\Z)^\times$, is an isometry from
$\chi_k$ to $\chi_{k'}$ precisely when
\begin{equation*}
 k\equiv k'u^2\pmod n.
\end{equation*}
On a recoupling diagram this is the simultaneous relabelling
\begin{equation*}
 \leftsplit{m}{m}{m}{a}{m}
   \xmapsto{\ f\ }
 \leftsplit{m}{m}{m}{ua}{m},\qquad
 \frac{\tau}{\zeta^{kab}}
   =\frac{\tau}{\zeta^{k'(ua)(ub)}}.
\end{equation*}
Both channel labels must be relabelled. For $n=3$, every unit has
square $1$, so $k=1$ and $k=2$ are inequivalent. Together with
$\tau=\pm1/\sqrt3$ this gives four monoidal classes. In particular,
replacing $\chi$ by $\chi^{-1}$ need not preserve the category.
For $n=5$, the two orbits are $\{1,4\}$ and $\{2,3\}$, again giving
four classes. For example, $u=2$ identifies
$k=1$ with $k'=4$ since $4\cdot2^2\equiv1\pmod5$.
\end{example}

Allowing a nonsemisimple tensor category does not remove the
restriction on $|G|$: the same Grothendieck rules already force
products of simples to split.

\begin{proposition}[Splitting the TY products]\label{prop:TY-splitting}
Let $\mathcal T$ be a finite tensor category over $k$ with simples
$G\sqcup\{m\}$, where $G$ is a finite group with $|G|>1$, and with
multiplication \autoref{eq:fusion} in its Grothendieck ring. Then
\[
 m\otimes m\cong\bigoplus_{g\in G}g.
\]
The full additive subcategory generated by the simples is a TY category.
Consequently $G$ is abelian and $|G|\ne0$ in $k$.
\end{proposition}

\begin{proof}
The products $gg^{-1}$, $g^{-1}g$, $gm$ and $mg$ have composition
length one. Hence the group simples are invertible and
$gm\cong mg\cong m$. Duality permutes the simples and preserves
invertibility, so $m^*\cong m$.

For $g\in G$, rigidity gives
\[
 \Hom(g,mm)\cong\Hom(m^*g,m)\cong k,\qquad
 \Hom(mm,g)\cong\Hom(m,gm^*)\cong k.
\]
Choose nonzero maps $i_g:g\to mm$ and $p_g:mm\to g$.
If $p_gi_g=0$, then $g$ occurs both in $\ker p_g$ and in the quotient
$g$, contradicting $[mm:g]=1$. After rescaling,
$p_gi_g=\id_g$. For $g\ne h$, $\Hom(h,g)=0$, hence
 $p_gi_h=\delta_{g,h}\id_g$.
Therefore the induced maps
\[
 I:\bigoplus_{g\in G}g\longrightarrow mm,\qquad
 P:mm\longrightarrow\bigoplus_{g\in G}g
\]
satisfy $PI=\id$. Both objects have composition length $|G|$, so
$I$ is an isomorphism.

Thus all products of simples split, and their duals are simple. The
full additive subcategory generated by them is therefore semisimple,
tensor-closed and rigid, with the original tensor unit. Applying
\autoref{thm:cl-classification} shows that $G$ is abelian and
$|G|\ne0$ in $k$. In particular, if $\operatorname{char}k$ divides
$|G|$, then \autoref{eq:cl-tau-forced} would give $0=1$.
\end{proof}

\begin{remark}
\autoref{prop:TY-splitting} says that the products of simples split and their full additive subcategory is TY. It does not assert that every object of the original finite tensor category is semisimple.

For comparison, in characteristic zero the nonsemisimple near-group
categories of \cite[Sections 2 and 3.1]{Se-nonsemisimple-near-group}
have a simple projective $Q$ whose square involves projective covers.
Their Grothendieck rules are therefore genuinely different from the TY
rules considered here.
\end{remark}

\section{Global recoupling}\label{sec:relations}

We now return to the semigroup $S$.

\subsection{The associators on semigroup labels}
\label{sec:normal-associators}

Recall that a TY datum on $A$ consists
of a symmetric nondegenerate bicharacter $\chi$ and a scalar $\tau$
with $|A|\tau^2=1$. We write
$\CA=\TY(A,\chi,\tau)$
for the corresponding local TY category.

Let $A\subseteq S$ be a group ideal, with identity $e$, and recall the
homomorphism
 $\phi:S\longrightarrow A,\phi(s)=se=es$.
The local coefficients extend to $S$ by applying $\phi$ to every black
label appearing in a coefficient, while leaving the boundary label
itself unchanged.

\begin{definition}\label{def:CS}
The \emph{TY extension associated to}
$(S,A,\chi,\tau)$ is the semisimple $k$-linear category $\CS$ with
simple labels $S\sqcup\{m\}$, fusion rules
\autoref{eq:group-ideal-rules}, and the following associators.

We use the black gauge $\omega=1$ and the normal vertex bases fixed
above. The arrows below display categorical associators in splitting-tree bases. The fusion equalities in the presentation use their inverse coefficients. For $s,t,u\in S$ and $a,b,x\in A$,
\begin{gather*}
\leftsplit{s}{t}{u}{st}{stu}
\xmapsto{\alpha_{s,t,u}}
\rightsplit{s}{t}{u}{tu}{stu},
\qquad
\leftsplit{s}{t}{m}{st}{m}
\xmapsto{\alpha_{s,t,m}}
\rightsplit{s}{t}{m}{m}{m},
\\[1.5ex]
\leftsplit{s}{m}{t}{m}{m}
\xmapsto{\alpha_{s,m,t}}
\chi\big(\phi(s),\phi(t)\big)
\rightsplit{s}{m}{t}{m}{m},
\qquad
\leftsplit{m}{s}{t}{m}{m}
\xmapsto{\alpha_{m,s,t}}
\rightsplit{m}{s}{t}{st}{m},
\\[1.5ex]
\leftsplit{s}{m}{m}{m}{x}
\xmapsto{\alpha_{s,m,m}}
\rightsplit{s}{m}{m}{\phi(s)^{-1}x}{x},
\qquad
\leftsplit{m}{s}{m}{m}{x}
\xmapsto{\alpha_{m,s,m}}
\chi\big(\phi(s),x\big)\,
\rightsplit{m}{s}{m}{m}{x},
\\[1.5ex]
\leftsplit{m}{m}{s}{x\phi(s)^{-1}}{x}
\xmapsto{\alpha_{m,m,s}}
\rightsplit{m}{m}{s}{m}{x},
\qquad
\leftsplit{m}{m}{m}{a}{m}
\xmapsto{\alpha_{m,m,m}}
\sum_{b\in A}\tau\chi(a,b)^{-1}
\rightsplit{m}{m}{m}{b}{m}.
\end{gather*}
Here inverses are taken in $A$. If $S$ is a monoid, then $\phi(1)=e$
and associators involving the global unit are identities.
\end{definition}

From now on we say msemigroupal to mean semigroupal on the semigroup input together with the fusion rules for $m$ in \autoref{prop:group-ideal-extension}.

\begin{proposition}[Transfer of recoupling]\label{prop:ambient-recoupling}
These formulas define the category $\CS$ of
\autoref{prop:group-ideal-extension}. They give a faithful
functor preserving tensor products up to the natural identifications,
that is, a msemigroupal functor
\[
 F_\phi:\CS\longrightarrow\CA,\qquad s\longmapsto\phi(s),\quad m\longmapsto m,
\]
monoidal when $S$ is a monoid, whose restriction to $\CA$
is the identity.
\end{proposition}

\begin{proof}
The functor keeps scalar matrix entries and direct sum copy indices,
so it is faithful even when distinct labels have the same image.
The identities $\phi(st)=\phi(s)\phi(t)$, $sx=\phi(s)x$ and
$xs=x\phi(s)$ for $x\in A$ identify the products and admissible
channels. Applying $F_\phi$ turns both routes of every mixed pentagon into
the corresponding TY routes. All-black components are identities.
Faithfulness therefore proves equality of the two routes in $\CS$.
For a monoid, $\phi(1)=e$ also gives compatibility with the unit.
\end{proof}

Distinct black labels can have the same image, so the functor need
not be full. The inclusion $\CA\subset\CS$ is full and semigroupal;
it preserves the global unit exactly when $S=A$.

\begin{example}[A nontrivial proper group ideal]\label{ex:proper-group-ideal}
Work over $\C$. Let $G_2=\{1,g\}$ with $g^2=1$, and put
$S=G_2\times\{0,1\}$ with
$(x,i)(y,j)=(xy,\min(i,j))$ and $A=G_2\times\{0\}$.
Write $1=(1,1)$, $s=(g,1)$, $e=(1,0)$ and $a=(g,0)$.
Then $1$ is the global unit, $A=\{e,a\}$ is a proper group ideal
with identity $e$, and $\phi(x,i)=(x,0)$. Thus
$\phi(1)=\phi(e)=e$ and $\phi(s)=\phi(a)=a$.

Choose $\chi(a,a)=-1$ and $\tau=1/\sqrt2$. Then $m^2=e\oplus a$ and
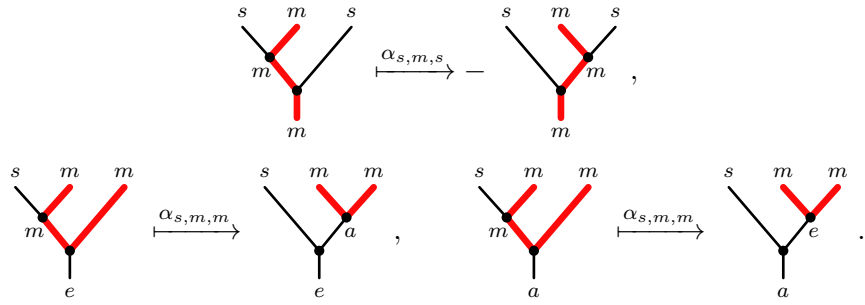
\begin{gather*}
 \leftsplit{s}{m}{s}{m}{m}
 \xmapsto{\alpha_{s,m,s}}
 -\,\rightsplit{s}{m}{s}{m}{m},
\\
 \leftsplit{s}{m}{m}{m}{e}
 \xmapsto{\alpha_{s,m,m}}
 \rightsplit{s}{m}{m}{a}{e},
 \qquad
 \leftsplit{s}{m}{m}{m}{a}
 \xmapsto{\alpha_{s,m,m}}
 \rightsplit{s}{m}{m}{e}{a}.
\end{gather*}
Thus the recoupling coefficients and internal channels only see
$\phi(s)=a$, while the boundary still remembers $s$.

The ambient category $\CS$ has four black simple objects, twice as many
as the local TY category $\CA$, but it is no longer rigid: the global
unit is $1$, whereas $m^2=e\oplus a$ contains only the local unit $e$.
Thus $m$ has no ambient dual.
\end{example}

\begin{proposition}[Field condition for the extension]\label{cor:ambient-field}
A finite semigroup with abelian group ideal $A$ supports the above
construction precisely when $|A|\ne0$ in $k$. Its parameter satisfies
$|A|\tau^2=1$; no invertibility condition on $|S|$ is imposed.
\end{proposition}

\begin{proof}
For $|A|>1$, apply \autoref{thm:cl-classification} and
\autoref{prop:ambient-recoupling}.
For $A=\{e\}$ use $\chi=1$ and $\tau^2=1$.
\end{proof}

The associators determine relations between trees with the same
boundary word. Together with cancellation at a vertex, these relations
present $\CS$, just as multiplication trees present $\Vec_S$.

\subsection{Generators and relations}

Fix $(S,A,\phi,\chi,\tau)$ as above and put $\mathcal L=S\sqcup\{m\}$.
We use the normal basis of \autoref{sec:normal-associators},
with $\omega=1$ on black triples. This choice of vertex bases and
the strictness of the diagram category are separate conventions.
The possible outputs are
\[
 \Out(s,t)=\{st\},\quad \Out(s,m)=\Out(m,s)=\{m\},\quad \Out(m,m)=A
 \qquad(s,t\in S).
\]
For $r,s\in\mathcal L$ and $t\in\Out(r,s)$ introduce
$p_{rs}^t:rs\to t$ and $i_{rs}^t:t\to rs$. The input $rs$ is a
word, whereas a black output $st$ is its single product label.
The four vertex types, with $s,t\in S$ and $a\in A$, are
\begin{equation}\label{eq:generators}
\begin{gathered}
 \fuse{s}{t}{st},\quad\splitv{s}{t}{st},
 \qquad \fuse{s}{m}{m},\quad\splitv{s}{m}{m},\\[1ex]
 \fuse{m}{s}{m},\quad\splitv{m}{s}{m},
 \qquad \fuse{m}{m}{a},\quad\splitv{m}{m}{a}.
\end{gathered}
\end{equation}
Write $p_{mm}^{a}=p_a$ and $i_{mm}^{a}=i_a$.
For a monoid erase only its global unit $1$ and impose
$p_{1r}^r=i_{1r}^r=p_{r1}^r=i_{r1}^r=\id_r$.
If $e\ne1$, retain the local-unit vertices and edges:
\[
 i_e=\splitv{m}{m}{e}:e\longrightarrow mm,\qquad
 p_e=\fuse{m}{m}{e}:mm\longrightarrow e.
\]
For $a\in A$, put $e_a=i_ap_a$.

The vertex relations \autoref{eq:resolution} are drawn as
\begin{equation*}
\bubble{r}{s}{u}{t}=\delta_{t,u}\,\wire{t},\qquad
 \sum_{t\in\Out(r,s)}\channel{r}{s}{t}=\wire{r}\,\wire{s}.
\end{equation*}

The empty diagram represents $\id_1$ in the monoid case.
Progressive isotopy preserves the ordered input and output sides, e.g.:

\begin{equation*}
 \isowire{m}=\wire{m},\qquad
 \isofuse{s}{m}{m}=\fuse{s}{m}{m}.
\end{equation*}
The interchange law of \autoref{remark:isotopy} applies to these vertices and to whole subdiagrams. No bends are included among these generators.

Moreover, the associators in \autoref{sec:normal-associators} act on
splitting trees. The fusion form \autoref{eq:two-tree-moves}
uses $F^{-1}$ and gives the defining tree relations. For $s,t,u\in S$ and
$a,b,x\in A$, the relations are:
\begin{gather}\label{eq:R-aaa}
\begin{gathered}
\lefttree{s}{t}{u}{st}{stu}
 =\righttree{s}{t}{u}{tu}{stu},
 \lefttree{s}{t}{m}{st}{m}
 =\righttree{s}{t}{m}{m}{m},
\\[1ex]
 \lefttree{s}{m}{t}{m}{m}
 =\chi\big(\phi(s),\phi(t)\big)^{-1}\righttree{s}{m}{t}{m}{m},
 \lefttree{m}{s}{t}{m}{m}
 =\righttree{m}{s}{t}{st}{m}.
\end{gathered}
\end{gather}
With two red inputs, the outer black label changes the group
channel, while a middle black label contributes a scalar:
\begin{gather}
 \lefttree{s}{m}{m}{m}{x}
 =\righttree{s}{m}{m}{\phi(s)^{-1}x}{x},\label{eq:R-amm}\\[1ex]
 \lefttree{m}{s}{m}{m}{x}
 =\chi\big(\phi(s),x\big)^{-1}\righttree{m}{s}{m}{m}{x},
 \lefttree{m}{m}{s}{x\phi(s)^{-1}}{x}
 =\righttree{m}{m}{s}{m}{x}.\label{eq:R-mam}
\end{gather}
With three red inputs, the fusion coefficients are the entries of
$F^{-1}$:
\begin{equation}\label{eq:R-mmm}
 \lefttree{m}{m}{m}{a}{m}
 =\sum_{b\in A}\tau\chi(a,b)\righttree{m}{m}{m}{b}{m}.
\end{equation}

\begin{definition}\label{def:D}\label{def:vertex}
The diagram category $\DS$ is the $k$-linear strict
msemigroupal category generated by \autoref{eq:generators}, modulo
\autoref{eq:resolution} and \autoref{eq:R-aaa}--\autoref{eq:R-mmm}.
For nonunital $S$ its objects are nonempty words in $\mathcal L$.
For a monoid it is strict monoidal, with objects the words in
$(S\setminus\{1\})\sqcup\{m\}$ including the empty word, and the
global-unit convention above. Diagrams are taken up to progressive
planar isotopy, preserving every vertex's ordered input and output
sides. As before, $\Add(\DS)$ adjoins finite direct sums and matrices
of morphisms.
\end{definition}

For $S=A$ the category $\DS$ is the TY diagram category $\D$.

\begin{lemma}[Splitting relations]\label{lem:splitting-relations}
The relations of $\DS$ imply the splitting equations
\autoref{eq:two-tree-moves}, obtained by strictifying the associator formulas.
\end{lemma}

\begin{proof}
Multiply the fusion relation by the matrix $F$ to obtain
$P_f^R=\sum_e F_{f,e}P_e^L$. Contract with $I_g^L$ and use
\autoref{eq:resolution} at its two vertices:
\[
 P_f^RI_g^L=\sum_eF_{f,e}P_e^LI_g^L=F_{f,g}\id_u.
\]
Applying the resolution twice also gives
$\id_{rst}=\sum_{u,f}I_f^RP_f^R$. Hence
$I_g^L=\sum_fF_{f,g}I_f^R$, as required.
\end{proof}

\begin{lemma}[Evaluation]\label{lem:evaluation}
Sending a word to its left-associated tensor product in $\CS$ and
vertices to the normal projections and inclusions defines a functor
$\mathcal E:\DS\to\CS$. Its tensor maps are the natural
reassociations, so it is strong semigroupal, and monoidal when $S$
is a monoid.
\end{lemma}

\begin{proof}
The projections and inclusions in $\CS$ satisfy
\autoref{eq:resolution}. By $F=P_R\alpha I_L$, comparing fusion trees
uses $F^{-1}$ and gives exactly \autoref{eq:R-aaa}--\autoref{eq:R-mmm}.
The pentagon proved in \autoref{prop:ambient-recoupling}
makes these comparisons compatible with tensor product.

Equivalently, first evaluate in the word strictification of $\CS$,
where the diagram relations hold as equalities. The functor back to
$\CS$ sends each word to its left-associated product and has the
stated reassociations as tensor maps. For a monoid these also
respect the unit maps.
\end{proof}

\begin{example}[The two order-two monoids]\label{ex:order-two-red-relations}
For \(A=G_2=\{1,g\}\) with \(g^2=1\), nondegeneracy gives
\(\chi(g,g)=-1\) and \(\tau=\nu/\sqrt2\), \(\nu=\pm1\) over
\(\C\). The two red channels resolve the identity,
\[
 \wire{m}\,\wire{m}
 =\channel{m}{m}{1}+\channel{m}{m}{g},
\]
and the three-red splitting relations are
\begin{equation}\label{eq:z2-F-rows}
 \leftsplit{m}{m}{m}{1}{m}
 \!{=}\frac{\nu}{\sqrt2}\left(
     \rightsplit{m}{m}{m}{1}{m}
    \!\!\!{+}\!\!\!\rightsplit{m}{m}{m}{g}{m}\right),\!
 \leftsplit{m}{m}{m}{g}{m}
 \!{=}\frac{\nu}{\sqrt2}\left(
     \rightsplit{m}{m}{m}{1}{m}
    \!\!\!{-}\!\!\!\rightsplit{m}{m}{m}{g}{m}\right).
\end{equation}
Here \(F=F^{-1}\), so the fusion relations have the same coefficients as \autoref{eq:z2-F-rows}.
For example,
\begin{equation*}
 \lefttree{m}{m}{m}{g}{m}
 =\frac{\nu}{\sqrt2}\left(
 \righttree{m}{m}{m}{1}{m}-\righttree{m}{m}{m}{g}{m}\right).
\end{equation*}
The mixed relations remember the bicharacter sign:
\begin{equation*}
\begin{gathered}
 \lefttree{g}{m}{g}{m}{m}=-\righttree{g}{m}{g}{m}{m},
 \lefttree{m}{g}{m}{m}{1}=\righttree{m}{g}{m}{m}{1},
 \lefttree{m}{g}{m}{m}{g}=-\righttree{m}{g}{m}{m}{g}.
\end{gathered}
\end{equation*}

For the idempotent monoid \(E_2=\{1,e\}\), \(e^2=e\), the group ideal
is \(A=\{e\}\). The restricted rule has only one red channel:
\[
 \wire{m}\,\wire{m}=\channel{m}{m}{e}.
\]
Writing \(\tau^2=1\) for its three-red coefficient, the only
nontrivial recoupling is
\begin{equation*}
 \lefttree{m}{m}{m}{e}{m}
 =\tau\,\righttree{m}{m}{m}{e}{m}.
\end{equation*}
Notice that here the three simple objects themselves form the monoid
\[
 M_3=\{1,e,m\},\qquad e^2=e,\quad em=me=m,\quad m^2=e.
\]
Thus this example is exactly a monoid-graded category $\Vec_{{M_3}}^\omega$
of \autoref{sec:pointed}, and its monoidal structures are classified by
$H^3(M_3,k^\times)$.
\end{example}

\subsection{Fusion trees}\label{sec:basis}

A tree successively resolves a word into simple summands. Unlike
the all-black case, its internal channels need not be unique.

\begin{definition}[Fusion trees]
Fix a planar binary tree $T$ with ordered leaves $w=s_1\cdots s_n$.
A labelling is \emph{admissible} if the output at every vertex belongs
to $\Out(r,s)$ for its two inputs. Denote by $\Lambda_T(w;t)$ the
labellings with root $t$. Composing projections gives $P^T_\lambda:w\to t$;
composing inclusions in the reverse order gives $I^T_\lambda:t\to w$.
A one-leaf tree is a strand; a monoid also has the empty tree,
with root $1$ and map $\id_1$.
\end{definition}

\begin{lemma}\label{lem:trees}
For each fixed tree $T$,
\begin{equation*}
 P^T_\lambda I^T_\mu=\delta_{\lambda,\mu}\id_t,
 \qquad
 \sum_{t\in\mathcal L}\sum_{\lambda\in\Lambda_T(w;t)}
 I^T_\lambda P^T_\lambda=\id_w.
\end{equation*}
The first expression is zero when the two root labels differ.
\end{lemma}
\begin{proof}
Stack the splitting tree below the fusion tree. Cancel matching
branches from the leaves towards the root using
\[
 \bubble{r}{s}{u}{t}=\delta_{t,u}\wire{t}.
\]
If a label fails to match, the diagram is zero; otherwise only the
root strand remains. For the second identity, we reverse the induction using
\[
 \wire{r}\,\wire{s}=\sum_{t\in\Out(r,s)}\channel{r}{s}{t}.
\]
Sum at the root and then in each subtree. Strands and empty trees need no replacement.
\end{proof}

\begin{lemma}[Change of tree]\label{lem:rotation}
Let $T,T'$ have the same ordered leaves and root. The families
$\{P^T_\lambda\}$ and $\{P^{T'}_\mu\}$ are related by an
invertible matrix. The same holds for the splitting families.
\end{lemma}
\begin{proof}
A rotation $((T_1T_2)T_3)\leftrightarrow(T_1(T_2T_3))$ is one of
\autoref{eq:R-aaa}--\autoref{eq:R-mmm} applied at the three subtree roots.
The first seven matrices are invertible monomial matrices; the last
is invertible by \autoref{lem:Fourier}. Any tree can be changed to
the left comb: rotate at the root until the right subtree is a leaf,
then repeat on the left subtree. For splitting maps use
\autoref{lem:splitting-relations}.
\end{proof}

Choose a left comb for each word and omit its superscript.
For $\lambda\in\Lambda(u;t)$ and $\mu\in\Lambda(v;t)$ put
\begin{equation}\label{eq:normal-form}
 E^t_{\mu,\lambda}=I_\mu P_\lambda:u\longrightarrow v.
\end{equation}
We abbreviate the chosen trees by fans, with ordered boundaries $u,v$:
\begin{equation*}
E^t_{\mu,\lambda}=\treepair{u}{v}{t}{\lambda}{\mu}.
\end{equation*}
The lower tree projects $u$ onto its copy $\lambda$ of the simple
$t$. The upper includes $t$ into the copy $\mu$ in $v$.

\subsection{The presentation theorem}

Cancellation gives the product of two tree pairs. Reassociation
makes their span closed under tensor products.

\begin{proposition}\label{prop:span}
For all words $u,v$, the maps in \autoref{eq:normal-form} span
$\Hom_{\DS}(u,v)$.
\end{proposition}
\begin{proof}
Let $M(u,v)$ be the stated span. \autoref{lem:trees} puts the
identities in $M$. The generating projections and inclusions also
belong to $M$: use a one-vertex tree on one side and a strand on the other (the empty tree for a global-unit output).

Orthogonality gives closure under composition:
\begin{equation}\label{eq:matrix-product}
 (I_\eta P_\mu)(I_{\mu'}P_\lambda)
 =\delta_{\mu,\mu'}I_\eta P_\lambda.
\end{equation}

On three red strands,
\begin{equation*}
 \matrixunit{a}{b}\circ\matrixunit{c}{d}
 =\delta_{b,c}\matrixunit{a}{d}.
\end{equation*}
The middle trees cancel from their leaves; the remaining lower and
upper channels are $d$ and $a$.

For tensor products, resolve the intermediate pair $r,s$:
\begin{align*}
 &(I_\mu P_\lambda)\otimes(I_{\mu'}P_{\lambda'})\\
 &\quad=\sum_{t\in\Out(r,s)}
 \underbrace{(I_\mu\otimes I_{\mu'})i_{rs}^{t}}_{\text{grafted splitting tree}}\hspace{0.5cm}
 \underbrace{p_{rs}^{t}(P_\lambda\otimes P_{\lambda'})}_{\text{grafted fusion tree}}.
\end{align*}
Change both grafted trees to the chosen combs by
\autoref{lem:rotation}. Each summand lies in $M$, also for a
global-unit root.
\end{proof}

\begin{theorem}[Tree basis and equivalence]\label{thm:presentation}
For all words $u,v$, the tree pairs $E^t_{\mu,\lambda}$ form a
basis of $\Hom_{\DS}(u,v)$. Evaluation is fully faithful and
strong semigroupal, and induces $\Add(\DS)\simeq\CS$, a monoidal
equivalence when $S$ is a monoid. Moreover,
\begin{equation}\label{eq:hom-dimension}
 \dim\Hom_{\DS}(u,v)
 =\sum_{t\in\mathcal L}|\Lambda(u;t)|\,|\Lambda(v;t)|.
\end{equation}
\end{theorem}
\begin{proof}
Put $\mathcal C=\CS$. Evaluation is well-defined by
\autoref{lem:evaluation}. For a fixed tree,
\[
 \mathcal E(w)\cong
 \bigoplus_{t\in\mathcal L}\ \bigoplus_{\lambda\in\Lambda(w;t)}t.
\]
The image of $I_\mu P_\lambda$ is the matrix unit from copy
$\lambda$ to copy $\mu$ of $t$. These images form a basis of the
target Hom-space. \autoref{prop:span} now proves that the
functor is fully faithful and gives \autoref{eq:hom-dimension}.

Every simple of $\mathcal C$ is the image of its strand, or of the empty
word for a global unit. The additive extension is therefore essentially surjective.
\end{proof}

\begin{remark}
By \autoref{thm:presentation}, every idempotent in $\Add(\DS)$
already splits, so no further completion by idempotents is needed.
The diagram category remains strict; the nontrivial associators
in $\CS$ are recorded by the tensor maps of evaluation.
\end{remark}

\begin{proposition}[The black and local subcalculi]\label{cor:pointed-embedding}
The assignments $\mu_{s,t}\mapsto p_{st}^{st}$ and
$\Delta_{s,t}\mapsto i_{st}^{st}$ give a fully faithful semigroupal
functor $\mathcal P(S)\to\DS$, monoidal when $S$ is a monoid.
Its additive image is $\Vec_S$.
The local category $\Add(\D)$ embeds fully faithfully and semigroupally
in $\Add(\DS)$ with image $\CA$. Its empty word maps to $e$,
so the inclusion preserves the ambient unit exactly when $S=A$.
\end{proposition}
\begin{proof}
For a black word there is a unique admissible comb, with root the
product of its leaves. Thus \autoref{thm:presentation} gives exactly the Hom-spaces in
\autoref{lem:pointed-normal}, with the same nonzero basis maps.
The tensor relations restrict to \autoref{eq:pointed-inverse} and
\autoref{eq:pointed-assoc} with $\omega=1$, so this is
$\mathcal P(S)=\mathcal P(S,1)$.
Apply \autoref{thm:pointed-equivalence}.
For the local inclusion, all roots and channels remain in
$A\sqcup\{m\}$, so the same basis proves fullness and faithfulness.
At the local empty word, its tensor maps are the identifications
$eX\cong X\cong Xe$ in $\CA$.
\end{proof}

\begin{proposition}\label{cor:simples}
The simples of $\Add(\DS)$ are $S\sqcup\{m\}$.
The unique positive real multiplicative dimension function on these
rules is $d_{\mathrm{pos}}(s)=1$ and $d_{\mathrm{pos}}(m)=\sqrt N$.
On the local fusion category these are the Frobenius--Perron dimensions
$\FPdim(a)=1$ and $\FPdim(m)=\sqrt N$, distinct from categorical
dimensions in $k$.
\end{proposition}

\begin{proof}
The basis theorem gives one-dimensional endomorphism spaces for these
labels and zero Hom-spaces between distinct labels. Every word splits
into them by \autoref{lem:trees}. Positivity and $sm=m$ force
$d_{\mathrm{pos}}(s)=1$; then $m^2=\bigoplus_a a$ forces
$d_{\mathrm{pos}}(m)^2=N$. Conversely these values respect all rules,
and restrict to the positive dimension function of the local fusion ring.
\end{proof}

\Needspace{13\baselineskip}
\subsection{Tensor products and mixed words}

Counting the free channels in a comb gives the tensor powers of $m$.

\begin{proposition}\label{prop:powers}
For $r\ge1$ and $s\ge0$,
\begin{equation*}
 m^{2r}\cong\bigoplus_{a\in A}a^{\oplus N^{r-1}},\qquad
 m^{2s+1}\cong m^{\oplus N^s}.
\end{equation*}
The decompositions are realized by the tree projections and inclusions.
Consequently
\begin{equation}\label{eq:End-powers}
 \End(m^{2r})\cong\bigoplus_{a\in A}\Mat_{N^{r-1}}(k),\qquad
 \End(m^{2s+1})\cong\Mat_{N^s}(k).
\end{equation}
In particular $\dim\End(m^n)=N^{n-1}$ for $n\ge1$, whereas
$\End(m^0)=\End(1)=k$ only when $S$ has a global unit.
\end{proposition}
\begin{proof}
The even formula starts with the channel decomposition at $r=1$.
Tensor with $m$ and use $am=m$ to obtain $N^r$ copies of $m$;
tensor once more to obtain $N^r$ copies of each $a$.
The case $s=0$ is $m$ itself. \autoref{lem:trees} supplies the
decomposition maps, giving \autoref{eq:End-powers}. Equivalently, count the
$r-1$ free comb labels at an even root and the $s$ free labels
at the odd root $m$.
\end{proof}

If a word contains a red strand, all black labels can be absorbed
into red strands. The number of red strands then determines its
decomposition.

\begin{proposition}[Mixed words]\label{prop:mixed-words}
Let $h(w)$ count the red strands of $w$. If $h(w)>0$, multiplication
and absorption give an isomorphism $\kappa_w:w\to m^{h(w)}$.
For $p=h(u)$ and $q=h(v)$,
\begin{equation*}
 \dim\Hom_{\DS}(u,v)=
 \begin{cases}
 \delta_{\pi(u),\pi(v)},&p=q=0,\\
 N^{(p+q)/2-1},&p,q>0,\quad p\equiv q\pmod2,\\
 {\bf1}_A(\pi(u))N^{r-1},&p=0,\quad q=2r>0,\\
 {\bf1}_A(\pi(v))N^{r-1},&q=0,\quad p=2r>0,\\
 0,&\text{otherwise}.
 \end{cases}
\end{equation*}
Here $\pi$ is the product of a pure black word and $\mathbf1_A$ its
membership indicator. Empty products occur only for monoids and
have value the global $1$.
\end{proposition}
\begin{proof}
Fuse each maximal block of black labels, then absorb it into the next $m$,
working from left to right. Absorb a terminal block into the last $m$.
All these maps are invertible by \autoref{eq:resolution}; their
composite defines $\kappa_w$.

For two pure black words use \autoref{cor:pointed-embedding};
otherwise use \autoref{prop:powers}. Two positive even powers
have $N$ common roots, two odd powers have the single root $m$,
and opposite parities have none. If one boundary is pure black and the other contains $2r>0$ red
strands, their only possible common root is the black product.
It contributes $N^{r-1}$ precisely when it lies in $A$. Summing products of
multiplicities proves the formula. The tree matrix units, conjugated
by the absorption maps, give bases for these Hom-spaces.
\end{proof}

\begin{example}[The idempotent boundary]\label{ex:ambient-mixed}
For \(E_2=\{1,e\}\) with \(e^2=e\) and \(A=\{e\}\), the restricted rule is
\(mm=e\). Thus
\[
 \Hom(1,mm)=0,\qquad \Hom(e,mm)=k\,\splitv{m}{m}{e}.
\]
For \(S=A=G_2=\{1,g\}\) with \(g^2=1\), both black labels occur in \(mm\).
More generally, a proper group ideal in a monoid gives
\(\Hom(1,m^{2r})=\Hom(m^{2r},1)=0\) for \(r>0\).
\end{example}

\section{Endomorphism algebras}\label{sec:channels}

We now describe the endomorphism algebras more explicitly.
In the ordinary TY case these are closely related to clock--shift,
generalized Clifford and twisted tensor-product algebras
\cite{EvGa-TY-Potts,GuKiRoZh-twisted}.
We derive the operators directly from the TY channel projectors in the
black/red calculus, uniformly for our semigroup extensions.
For related diagrammatic Fourier and para constructions, see
\cite{HePeTe-planar,JaLi-planar-para}.

\subsection{Channel idempotents and matrix units}

We focus on powers of $m$. The group identity is $e$ and $N=|A|$ is invertible in $k$.
Numerical examples are over $\C$. When $S=A$, we may write $e$ as $1$.

\begingroup
\renewcommand{\quotlabel}[2]{\ifstrequal{#1}{#2}{e}{\ifstrequal{#1}{e}{#2}{#1^{-1}#2}}}
\renewcommand{\quotedge}[3]{\ifstrequal{#1}{#2}{\edge{e}{#3}}{\ifstrequal{#1}{e}{\edge{#2}{#3}}{\edge{a}{#3}}}}

\begin{lemma}[Reduction to the group ideal]\label{thm:ambient-red}
For $n\ge1$, the full inclusion $\CA\subset\CS$ gives
\[
 B_n:=\End_{\CS}(m^n)=\End_{\CA}(m^n).
\]
This identification preserves the tree matrix units and all local
operators.
\end{lemma}

\begin{proof}
Every positive power of $m$ lies in $\CA$. Its channels,
associators and tree products are therefore the same in both categories. The
presentation below also gives independence of $\tau$.
\end{proof}

\begin{remark}
The notation $B_n$ refers to powers of $m$; $B_n(S)$ in \autoref{eq:semigroup-end} refers to the sum of the black simples.
\end{remark}

Resolving two red strands into the channel $a$ and back gives the
projector $e_a$. Cancellation makes these projectors orthogonal:
\begin{equation*}
\begin{gathered}
 e_a=i_ap_a=\channel{m}{m}{a},\qquad
 \zchannels{2}{1/b,1/a}=\delta_{a,b}\channel{m}{m}{a},\qquad
 \sum_{a\in A}\channel{m}{m}{a}=\ztlid{2}.
\end{gathered}
\end{equation*}

\begin{lemma}
The $e_a$ form 
a basis of $\End(mm)$ and $\End(mm)\cong k^N$.
\end{lemma}

\begin{proof}
Immediate.
\end{proof}

On $m^3$ put
$L_a=e_a\otimes\id_m$ and $R_b=\id_m\otimes e_b$, and set
$M_{a,c}=I_a^LP_c^L$. Their expanded diagrams are
\[
 M_{a,c}=\matrixunit{a}{c}.
\]
The lower tree selects channel $c$ and fuses to $m$; the upper tree
splits through channel $a$.

\begin{proposition}\label{prop:adjacent}
For $a,b,c\in A$,
\begin{equation}\label{eq:adjacent}
 L_aR_bL_c=\frac{\chi(ac^{-1},b)}{N}M_{a,c}.
\end{equation}
In particular,
\begin{equation}\label{eq:adjacent-diagonal}
 L_aR_bL_a=N^{-1}L_a,\qquad
 M_{a,c}=N L_aR_eL_c.
\end{equation}
Hence $\{L_a,R_b:a,b\in A\}$ generates $\End(m^3)$ and $\End(m^3)\cong\Mat_N(k)$.
\end{proposition}

\begin{proof}
Expanding every channel gives
\begin{equation*}
\zchannels{3}{1/c,2/b,1/a}
 =\frac{\chi(ac^{-1},b)}{N}\matrixunit{a}{c}.
\end{equation*}
For example, $(L_aR_b)^2=N^{-1}L_aR_b$. Hence
$NL_aR_b$ is an idempotent.

Resolve the middle projector into its right splitting and fusion
trees. The two remaining overlaps are
\begin{equation}\label{eq:adjacent-overlaps}
 \oppoverlap{b}{a}=\tau\chi(a,b)\wire{m},\qquad
 \coverlap{c}{b}=\tau\chi(c,b)^{-1}\wire{m}.
\end{equation}
The $F^{-1}$ and $F$ entries multiply to $N^{-1}\chi(ac^{-1},b)$,
leaving $M_{a,c}$. Take $c=a$ for the projector sandwich and $b=e$
for all matrix units.
\end{proof}

\begin{example}[Three strands for the group of order two]
Take $A=G_2=\{1,g\}$, $g^2=1$, $\chi(g,g)=-1$ and
$\tau=\nu/\sqrt2$.

Write $L_x=e_x\otimes\id_m$, $R_x=\id_m\otimes e_x$ and
$M_{x,y}=I_x^LP_y^L$ for $x,y\in\{1,g\}$:
\[
 M_{1,1}=\zmatrixunit{1}{1},\quad
 M_{g,1}=\zmatrixunit{g}{1},\quad
 M_{1,g}=\zmatrixunit{1}{g},\quad
 M_{g,g}=\zmatrixunit{g}{g}.
\]
In the left-comb basis ordered by $(1,g)$,
\[
 F=\frac{\nu}{\sqrt2}
 \begin{pmatrix}1&1\\1&-1\end{pmatrix},\qquad
 L_1=\begin{pmatrix}1&0\\0&0\end{pmatrix},\qquad
 R_1=\frac12\begin{pmatrix}1&1\\1&1\end{pmatrix}.
\]

Since $N=2$, the general sandwich relation becomes
\[
 L_xR_yL_z=\frac12\chi(xz,y)M_{x,z}.
\]
Thus the two middle channels give opposite off-diagonal matrix units:
\[
 \zchannels{3}{1/1,2/1,1/g}
 =\frac12\zmatrixunit{g}{1},\qquad
 \zchannels{3}{1/1,2/g,1/g}
 =-\frac12\zmatrixunit{g}{1}.
\]
Indeed, the coefficient is
$\tau\chi(x,y)\tau\chi(z,y)=\frac12\chi(xz,y)$.

An off-diagonal tree pair can be nonzero and still have square zero.
For $V=2L_aR_1L_1=M_{g,1}$,
\begin{equation*}
 \zmatrixunit{g}{1}\circ\zmatrixunit{g}{1}=0,\qquad
 \zmatrixunit{1}{g}\circ\zmatrixunit{g}{1}=\zchannels{3}{1/1}.
\end{equation*}
Indeed, the middle contraction in the first product is
$P_1^LI_g^L=0$; in the second it is $P_g^LI_g^L=\id_m$.
More generally,
$M_{x,y}M_{z,w}=\delta_{y,z}M_{x,w}$.
\end{example}

\begin{example}[Adjacent channels for $\Z/3\Z$]
For the datum in \autoref{ex:Z3-recoupling}, put $e=1$. Then
\begin{equation*}
\zchannels{3}{1/1,2/g,1/g}=\frac\zeta3\matrixunit{g}{1},\qquad
 \zchannels{3}{1/g,2/g,1/g}=\frac13\zchannels{3}{1/g},
\end{equation*}
as one can easily check.
\end{example}

\subsection{Fusion paths and Fourier operators}

The successive labels along a splitting comb form a fusion path.
For $m^{2r+1}$, the path is determined by
$\mathbf a=(a_1,\ldots,a_r)\in A^r$. After the first $2j$ leaves
the root is $a_j$; after $2j+1$ leaves it is $m$. Set $a_0=e$.
For $m^{2r}$ with fixed root $t$, set $a_r=t$, leaving
$a_1,\ldots,a_{r-1}$ free. Denote the splitting map by
$|\mathbf a\rangle$. On six leaves this splitting map is
\begin{equation*}
|a_1,a_2;t\rangle=\zsixstate{a_1}{a_2}{t}{none}.
\end{equation*}

Write $e_b^{(i)}$ for the $b$-channel projector on strands $i,i+1$,
and define the Fourier operator $u_i(x)$ on the same two strands by
\begin{equation}\label{eq:Fourier-channels}
 \channel{m}{m}{b}=e_b,\qquad
 \zlocal{2}{1/u(x)}
 =\sum_{b\in A}\chi(x,b)\channel{m}{m}{b}.
\end{equation}
Placing these diagrams on strands $i,i+1$ gives $e_b^{(i)}$ and
$u_i(x)$.

\begin{lemma}[Fourier inversion]\label{lem:Fourier-inversion}
For $1\le i<n$,
\[
 e_b^{(i)}
 =\frac1N\sum_{x\in A}\chi(x,b)^{-1}u_i(x).
\]
\end{lemma}

\begin{proof}
This is character orthogonality, \autoref{eq:orthogonality}.
\end{proof}

\begin{lemma}[Local action on paths]\label{lem:path-action}
For all admissible indices,
\begin{align*}
 e_b^{(2j-1)}|\mathbf a\rangle
 &=\delta_{b,a_{j-1}^{-1}a_j}|\mathbf a\rangle,
 \\
 e_b^{(2j)}|\mathbf a\rangle
 &=\frac1N\sum_{c\in A}\chi(ca_j^{-1},b)
 |a_1,\ldots,a_{j-1},c,a_{j+1},\ldots\rangle.
\end{align*}
Consequently
\begin{align}
 u_{2j-1}(x)|\mathbf a\rangle
 &=\chi(x,a_{j-1}^{-1}a_j)|\mathbf a\rangle,
 \label{eq:path-u-odd}\\
 u_{2j}(x)|\mathbf a\rangle
 &=|a_1,\ldots,a_{j-1},x^{-1}a_j,a_{j+1},\ldots\rangle.
 \label{eq:path-u-even}
\end{align}
\end{lemma}
\begin{proof}
For an odd position, contract the prefix to $h=a_{j-1}$. Reassociate
the local tree by the splitting form of \autoref{eq:R-amm}:
\[
 \leftsplit{h}{m}{m}{m}{a_j}
 =\rightsplit{h}{m}{m}{h^{-1}a_j}{a_j}.
\]
The pair channel is $h^{-1}a_j$, with coefficient $1$.
A $b$-projector therefore gives
$\delta_{b,h^{-1}a_j}$, and its Fourier sum acts diagonally:
\begin{equation*}
 \stateaction{h}{m}{m}{m}{a}{2}{u(x)}
 =\chi(x,h^{-1}a)\leftsplit{h}{m}{m}{m}{a}.
\end{equation*}
For an even position, contract the prefix to $m$. Resolve the
right pair, then recouple back:
\begin{equation}\label{eq:graphical-local-projector}
 \stateaction{m}{m}{m}{a}{m}{2}{e_b}
 =\frac1N\sum_c\chi(ca^{-1},b)\leftsplit{m}{m}{m}{c}{m}.
\end{equation}
The coefficient is the product $(F^{-1})_{c,b}F_{b,a}$ from
\autoref{eq:adjacent-overlaps}. Fourier summation forces $xca^{-1}=e$:
\begin{equation*}
 \stateaction{m}{m}{m}{a}{m}{2}{u(x)}
 =\leftsplit{m}{m}{m}{x^{-1}a}{m}.
\end{equation*}
Reattach the prefix and suffix. For the first pair, $a_0=e$ uses
the local comparison $em\cong m$.
\end{proof}

\begin{example}[Fourier operators for $\Z/3\Z$]
\label{ex:local-Z3}\label{ex:Z3-Weyl}
Work over $\C$. Let
$A=\langle g\mid g^3=e\rangle$ and
$\chi(g,g)=\zeta=e^{2\pi i/3}$.
On three strands,
\[
 \stateaction{m}{m}{m}{g}{m}{1}{u(g)}
 =\zeta\leftsplit{m}{m}{m}{g}{m},\qquad
 \stateaction{m}{m}{m}{g}{m}{2}{u(g)}
 =\leftsplit{m}{m}{m}{e}{m}.
\]
Thus the first Fourier operator is diagonal, while the second shifts
the channel. (For Fourier transforms, matrix units and Pauli operators, compare \cite[Sections 6.3--6.5 and 9]{JaLi-planar-para}.) For example,
\[
 \zchannels{3}{1/e,2/e,1/g}
 =\tau^2\matrixunit{g}{e}
 =\frac13\matrixunit{g}{e},
 \qquad
 \matrixunit{e}{g}\circ\matrixunit{g}{e}
 =\matrixunit{e}{e}
 =\zchannels{3}{1/e}.
\]

In the path basis $\{|e\rangle,|g\rangle,|g^2\rangle\}$, put
\[
 Z=u_1(g)=
 \begin{pmatrix}1&0&0\\0&\zeta&0\\0&0&\zeta^2\end{pmatrix},
 \qquad
 X=u_2(g^2)=
 \begin{pmatrix}0&0&1\\1&0&0\\0&1&0\end{pmatrix}.
\]
Then $X|a\rangle=|ga\rangle$ and $ZX=\zeta XZ$. Hence the nine
operators $X^jZ^k$, $0\le j,k<3$, form a basis of $\End(m^3)$.
For instance,
\[
 L_e=\frac13(1+Z+Z^2),\qquad
 M_{g,e}=XL_e=\frac13(X+XZ+XZ^2).
\]
The commutator reduces diagrammatically to
\begin{equation}\label{eq:Z3-commutator}
 \zlocal{3}{2/X^2,1/Z^2,2/X,1/Z}
 =\zeta\,\zlocal{3}{2/X^2,1/Z^3,2/X}
 =\zeta\,\ztlid{3}.
\end{equation}
The coupons expand into the channel projectors of
\autoref{eq:Fourier-channels}.
\end{example}

Fourier operators also compare the two ways of absorbing a black strand.
For $s\in S$, put
\[
 d_s^L=p_{ms}^{m}\otimes\id_m,\qquad
 d_s^R=\id_m\otimes p_{sm}^{m}:msm\longrightarrow mm
\]
in the strict diagram category. The first absorbs $s$ to the left,
the second to the right. Changing between them gives a Fourier operator on $mm$.

\begin{lemma}[Moving an absorption]\label{lem:absorption}
For every $s\in S$,
\begin{equation*}
 d_s^L=\left(\sum_{x\in A}\chi(\phi(s),x)^{-1}e_x\right)d_s^R,
 \qquad d_s^L(d_s^R)^{-1}=u_1(\phi(s)^{-1}).
\end{equation*}
\end{lemma}

\begin{proof}
The fusion relation \autoref{eq:R-mam} gives
\[
 \lefttree{m}{s}{m}{m}{x}
 =\chi\big(\phi(s),x\big)^{-1}\righttree{m}{s}{m}{m}{x},
\]
or $p_xd_s^L=\chi(\phi(s),x)^{-1}p_xd_s^R$.
Multiply by $i_x$ and sum. The definition of the Fourier operator
gives the second equality.
\end{proof}

For $s=a\in A$ this is the expanded absorption comparison
\begin{equation*}
 \leftabsorptionpic
 =\sum_x\chi(a,x)^{-1}\rightabsorptionpic.
\end{equation*}

\subsection{An algebra presentation}

The Fourier operators give a particularly small diagrammatic
presentation. Over $\C$, its abstract algebra relations and ordered basis are those of \cite[Section 3.1 and Proposition 3.2]{GuKiRoZh-twisted}, with twisting bicharacter $\chi^{-1}$. Here the generators are realized by TY channel operators.

\begin{theorem}\label{thm:centralizers}
For $n\ge2$, the local operators of \autoref{eq:Fourier-channels}
present $B_n=\End(m^n)$ by the relations
\begin{align}
 \zlocal{2}{1/u(e)}
 &=\ztlid{2},
 \zlocal{2}{1/u(y),1/u(x)}
=\zlocal{2}{1/u(xy)},
 \label{eq:u-group}
 \\[1ex]
 \zlocal{3}{2/u(y),1/u(x)}
 &=\chi(x,y)^{-1}
   \zlocal{3}{1/u(x),2/u(y)},
 \label{eq:u-near}
 \\[1ex]
 \zlocal{4}{3/u(y),1/u(x)}
 &=\zlocal{4}{1/u(x),3/u(y)}.
 \label{eq:u-far}
\end{align}
Here $x,y\in A$ and $1\le i,j<n$: the first relations act at one site $i$, the second at sites $i,i+1$ with $i<n-1$, and disjoint sites commute for $|i-j|>1$.
Their ordered monomials
\begin{equation}\label{eq:u-basis}
 u_1(x_1)u_2(x_2)\cdots u_{n-1}(x_{n-1}),
 \qquad (x_1,\ldots,x_{n-1})\in A^{n-1},
\end{equation}
form a basis. This implies
\begin{equation}\label{eq:centralizer-blocks}
 \End(m^{2r+1})\cong\Mat_{N^r}(k),\qquad
 \End(m^{2r})\cong\bigoplus_{t\in A}\Mat_{N^{r-1}}(k).
\end{equation}
\end{theorem}

\begin{proof}
Channel orthogonality gives \autoref{eq:u-group}, while disjoint
supports give \autoref{eq:u-far}. For adjacent supports, one Fourier
operator shifts the channel seen by the other. The coefficient changes
by $\chi(x,y)^{-1}$, giving \autoref{eq:u-near}.

The three relations put every word into the ordered form
\autoref{eq:u-basis}, so the presented algebra has dimension at most
$N^{n-1}$.

For generation, put
\[
 v_j(x)=u_1(x)u_3(x)\cdots u_{2j-1}(x).
\]
By \autoref{eq:path-u-odd}, $v_j(x)$ acts by $\chi(x,a_j)$.
Hence
\[
 \frac1N\sum_x\chi(x,c)^{-1}v_j(x)
\]
projects onto paths with $a_j=c$, while $u_{2j}(x)$ shifts that
coordinate by $x^{-1}$. Products of these projections and shifts give
all matrix units inside each root block.

For $n=2r$, the operator $v_r(x)$ acts by $\chi(x,t)$ on root $t$, so
\begin{equation}\label{eq:central-root}
 z_t=\frac1N\sum_x\chi(x,t)^{-1}v_r(x)
\end{equation}
projects onto that root block. There are $N^r$ paths for odd $n$, and
$N^{r-1}$ paths in each of the $N$ even root blocks. Hence
\autoref{eq:centralizer-blocks} follows, and the total dimension is
$N^{n-1}$. The presented algebra therefore has the correct dimension
and maps isomorphically onto $\End(m^n)$.
\end{proof}

\subsection{Root blocks and simple modules}\label{sec:sandwich}

Fix $n\ge1$ and put 
$\Lambda_n(t)=\Lambda(m^n;t)$ and
$S_n=\{t:\Lambda_n(t)\ne\emptyset\}$.
Cutting a tree pair at its simple root $t$ leaves a fusion tree
below and a splitting tree above. At the cut we have $\End(t)=k\id_t$:
\begin{equation*}
 E^t_{\mu,\lambda}=I_\mu\,\id_t\,P_\lambda
 =\begin{tikzpicture}[scale=.72,
 every node/.style={font=\scriptsize}]
  \path[region](-.8,.3)--(.8,.3)--(0,1)--cycle;
  \path[fill=treegraylight](-.8,2.7)--(.8,2.7)--(0,2)--cycle;
  \foreach \x in {-.65,0,.65}{\draw[m](\x,0)--(\x,.3);
    \draw[m](\x,2.7)--(\x,3);}
  \edge{t}{(0,1)--(0,2)}
  \draw[gray,dashed](-1.1,1.5)--(1.1,1.5);
  \node[coupon] at (0,1.5){$\id_t$};
  \node at (0,.5){$P_\lambda$};\node at (0,2.5){$I_\mu$};
  \node[below] at (0,0){$m^n$};\node[above] at (0,3){$m^n$};
 \end{tikzpicture}.
\end{equation*}
The shaded fans denote the chosen trees, with root $m$ or $t\in A$.
Gluing the middle trees gives the scalar pairing
(cf. \autoref{lem:trees})
\begin{equation*}
 \Hom(m^n,t)\otimes\Hom(t,m^n)\longrightarrow\End(t),
 \qquad P_\lambda\otimes I_\mu\longmapsto
 P_\lambda I_\mu=\delta_{\lambda,\mu}\id_t.
\end{equation*}
The multiplication acts separately on the lower and upper tree
indices. This is the sandwich cellular structure, with one scalar
algebra at each root, just as in \autoref{sec:pointed}. The spaces
spanned by the upper splitting trees give the simple modules.

\begin{proposition}[Sandwich cellular basis]\label{prop:sandwich}
The tree pairs form a sandwich cellular basis of $B_n$ in the sense
of \cite[Definition~2A.3]{Tu-sandwich-cellular}. Its sandwich cell datum is
\begin{equation*}
\begin{gathered}
 P=S_n\ \text{with the antichain order},\qquad
 T(t)=B(t)=\Lambda_n(t),\\
 H_t=\End(t)\cong k,\qquad B_t=\{\id_t\},\qquad
 C(\mu,\id_t,\lambda)=E^t_{\mu,\lambda}.
\end{gathered}
\end{equation*}
The $J$-cells, defined by mutual reachability under two-sided
multiplication, are precisely the root blocks.
\end{proposition}
\begin{proof}
\autoref{thm:presentation} gives the basis. By
\autoref{eq:matrix-product},
\begin{equation}\label{eq:sandwich-product}
 E^t_{\mu,\lambda}E^s_{\nu,\kappa}
 =\begin{cases}
 E^t_{\mu,\kappa},&t=s,\ \lambda=\nu,\\
 0,&\text{otherwise}.
 \end{cases}
\end{equation}
For $f=\sum_{s,\eta,\kappa}f^s_{\eta,\kappa}E^s_{\eta,\kappa}$ this gives
\[
 fE^t_{\mu,\lambda}=\sum_\eta f^t_{\eta,\mu}E^t_{\eta,\lambda},
 \qquad
 E^t_{\mu,\lambda}f=\sum_\kappa f^t_{\lambda,\kappa}E^t_{\mu,\kappa}.
\]
The left coefficients are independent of $\lambda$, and the right
coefficients of $\mu$. There are no higher terms in the antichain order.
The cell algebra at $t$ has the $B_n$-bimodule factorization
\begin{equation*}
\begin{aligned}
 \Hom(t,m^n)\otimes_{H_t}\Hom(m^n,t)
 &\xrightarrow{\ \sim\ }\operatorname{span}_{k}
       \{E^t_{\mu,\lambda}:\mu,\lambda\in\Lambda_n(t)\},\\
 I_\mu\otimes P_\lambda&\longmapsto I_\mu P_\lambda.
\end{aligned}
\end{equation*}
Both factors are free over $H_t$, with the splitting and fusion bases.
This verifies (AC1)--(AC3). The product formula also shows that two
basis elements lie in the same $J$-cell exactly when their roots agree.
\end{proof}

A left cell fixes $(t,\lambda)$ and varies
$\mu$: it is a column of splitting trees. A right cell fixes
$(t,\mu)$ and varies $\lambda$: it is a row of fusion trees.
An $H$-cell is the intersection of a left and a right cell;
here it is a single basis element. Choosing a
diagonal one gives the idempotent corner
\[
 \End(t)\xrightarrow{\ \sim\ }
 E^t_{\lambda,\lambda}B_nE^t_{\lambda,\lambda},
 \qquad h\longmapsto I_\lambda hP_\lambda.
\]

\begin{remark}
This is also a Graham--Lehrer cellular datum~\cite{GrLe-cellular} for the
$k$-linear anti-involution
$(E^t_{\mu,\lambda})^\star=E^t_{\lambda,\mu}$.
\autoref{eq:sandwich-product} gives
$(fg)^\star=g^\star f^\star$ and $\star^2=\id$; triangularity follows
as above.
\end{remark}

\begin{proposition}[Simple modules]\label{cor:path-simples}
A complete set of pairwise nonisomorphic simple left $B_n$-modules (with action by postcomposition) is
\begin{equation*}
 L_t=\Hom(t,m^n)=\bigoplus_{\lambda\in\Lambda_n(t)}k I_\lambda
 \qquad(t\in S_n).
\end{equation*}
Their path bases and dimensions are
\[
\begin{array}{c|c|c|c}
 n&\text{simple modules}&\text{basis indices}&\text{dimension}\\\hline
 2r+1\ (r\ge0)&L_m\ \text{(one)}&A^r&N^r\\
 2r\ (r\ge1)&L_t,\ t\in A\ \text{($N$ in total)}&A^{r-1}&N^{r-1}
\end{array}
\]
For a monoid, $n=0$ additionally gives $B_0=k$ and the single simple
$L_1=k$; otherwise there is no ambient empty power.
\end{proposition}
\begin{proof}
In the root-$t$ block,
$E^t_{\mu,\lambda}I_\nu=\delta_{\lambda,\nu}I_\mu$;
all other blocks act by zero. Matrix units therefore extract any
nonzero coefficient of a vector and send it to any basis vector.
Thus $L_t$ is simple. The central primitive idempotents are
\begin{equation*}
 z_t=\sum_{\lambda\in\Lambda_n(t)}E^t_{\lambda,\lambda},
 \qquad z_s|_{L_t}=\delta_{s,t}\id_{L_t}.
\end{equation*}
For even $n$, these are exactly \autoref{eq:central-root}; for odd $n$,
$z_m=1$. They distinguish the modules, and
\autoref{eq:centralizer-blocks} proves exhaustion. This is the
$H$-reduction of \cite[Theorem~2A.17]{Tu-sandwich-cellular}: each root has the single
simple $H_t$-module $k$. Counting the paths gives the displayed dimensions.
\end{proof}

The path vector $|\mathbf a\rangle$ in \autoref{lem:path-action}
is $I_\lambda\in L_t$. By \autoref{eq:path-u-odd} and
\autoref{eq:path-u-even}, odd generators act diagonally (as clocks) and even generators
shift $a_j\mapsto x^{-1}a_j$. For even powers, $a_r=t$ stays fixed.

\begin{example}[The doubled-up group]
For the semigroup in \autoref{ex:proper-group-ideal}, $N=2$.
The red algebras are
\[
 \End_{\CS}(m^{2r+1})\cong\Mat_{2^r}(\C),\qquad
 \End_{\CS}(m^{2r})\cong
 \Mat_{2^{r-1}}(\C)\oplus\Mat_{2^{r-1}}(\C),
\]
where $r\geq 1$.
\end{example}

\begin{example}[The order-two monoids]
For the group $G_2=\{1,g\}$ with $g^2=1$, the red square has two
channels,
\[
 \wire{m}\,\wire{m}
 =\channel{m}{m}{1}+\channel{m}{m}{g}.
\]
They are orthogonal idempotents:
\[
 \channel{m}{m}{x}\circ\channel{m}{m}{y}
 =\delta_{x,y}\channel{m}{m}{x},
 \qquad x,y\in\{1,g\}.
\]
Thus
$B_2\cong k\oplus k$.
In general, $B_{2r+1}\cong\Mat_{2^r}(k),\qquad
 B_{2r}\cong\Mat_{2^{r-1}}(k)\oplus\Mat_{2^{r-1}}(k)$.

For the idempotent monoid $E_2=\{1,e\}$ with $e^2=e$, the group ideal is
$A=\{e\}$. There is only one red channel,
\[
 \wire{m}\,\wire{m}=\channel{m}{m}{e},
\]
and hence
\[
 \channel{m}{m}{e}\circ\channel{m}{m}{e}
 =\channel{m}{m}{e}
 =\id_{m^2}.
\]
Therefore $B_2\cong k$. More generally,
$B_n\cong k$ for every $n\ge1$.
\end{example}

\subsection{Four-strand blocks}

Pairwise fusion on four strands gives
\[
 m^4\cong\bigoplus_{a,b\in A}ab
       \cong\bigoplus_{t\in A}t^{\oplus N}.
\]
For fixed root $t$, the first pair label $a$ determines the second
by $b=a^{-1}t$. The fusion and splitting maps are
\begin{equation*}
 Q_a^t=p_{a,a^{-1}t}^{t}(p_a\otimes p_{a^{-1}t}),\qquad
 J_a^t=(i_a\otimes i_{a^{-1}t})i_{a,a^{-1}t}^{t}.
\end{equation*}
The maps $Q_a^t$ form a basis of $\Hom(m^4,t)$, with diagrams
\begin{equation*}
Q_a^t=\pairfusion{a}{a^{-1}t}{t},\qquad J_a^t=\fourtree{4}{a}{a^{-1}t}{t}.
\end{equation*}
The matrix units $J_b^tQ_a^t$ satisfy
\begin{equation}\label{eq:four-matrix}
 (J_b^tQ_a^t)(J_d^uQ_c^u)
 =\delta_{t,u}\delta_{a,d}J_b^tQ_c^t,
 \qquad
 \End(m^4)\cong\bigoplus_{t\in A}\Mat_N(k).
\end{equation}

The corresponding pair of trees is
\begin{equation*}
 J_b^tQ_a^t=\fourunit{b}{a}{t}.
\end{equation*}
The lower tree fuses the pairs through $a$ and $a^{-1}t$ to root
$t$; the upper tree splits through $b$ and $b^{-1}t$. Thus $t$ fixes
the block, $a$ its column and $b$ its row. Cancelling the middle
trees gives \autoref{eq:four-matrix}. The $N$ splitting trees $J_a^t$
form a basis of $L_t$.

By \autoref{eq:R-amm}, this pair basis agrees with the comb basis.
\autoref{lem:path-action} gives
\begin{equation*}
 (e_a^{(1)})_{x,y}=\delta_{x,y}\delta_{x,a},\quad
 (e_b^{(2)})_{x,y}=\frac{\chi(xy^{-1},b)}N,\quad
 (e_c^{(3)})_{x,y}=\delta_{x,y}\delta_{c,x^{-1}t}.
\end{equation*}
Multiplication in the total-output-$t$ block gives
\begin{equation}\label{eq:four-sandwich}
 z_t e_a^{(1)}e_b^{(2)}e_c^{(3)}
 =\frac{\chi(act^{-1},b)}N J_a^tQ_{tc^{-1}}^t.
\end{equation}
The first and last projections select row $a$ and column $tc^{-1}$;
the middle matrix contributes the coefficient. In particular,
\begin{equation*}
 z_t=\sum_a e_a^{(1)}e_{a^{-1}t}^{(3)},\qquad
 J_a^tQ_d^t=N z_t e_a^{(1)}e_e^{(2)}e_d^{(1)}.
\end{equation*}
For instance,
\[
 (J_b^tQ_a^t)(J_a^tQ_d^t)=J_b^tQ_d^t,
 \qquad (J_b^tQ_d^t)^2=\delta_{b,d}J_b^tQ_d^t.
\]

\begin{example}[A root projection for $\Z/3\Z$]
Use $\chi(g,g)=\zeta=e^{2\pi i/3}$.
On four strands put $V=u_1(g)u_3(g)$. It acts by $1,\zeta,\zeta^2$
on the roots $e,g,g^2$, respectively. Thus
\[
 z_g=\tfrac13(1+\zeta^2V+\zeta V^2)
\]
projects onto root $g$. Its product with
$3e_a^{(1)}e_e^{(2)}e_d^{(1)}$ is the matrix unit
$J_a^gQ_d^g$ in this $3\times3$ block.
\end{example}

\endgroup

For the order-two Jones/Clifford comparison behind the following path models, see, for example, \cite{IoLeZh-Jones}.

\begin{example}[$\Z/2\Z$ again]\label{ex:Ising-four}
For $A=G_2=\{1,g\}$ with $g^2=1$, the first block decompositions are
\[
\begin{gathered}
 m^3\cong m^{\oplus2},\qquad
 m^4\cong1^{\oplus2}\oplus g^{\oplus2},\qquad
 m^5\cong m^{\oplus4},\\
 \dim\End(m^3)=4,\qquad\dim\End(m^4)=8,\qquad\dim\End(m^5)=16.
\end{gathered}
\]
Put
\[
 Z=\begin{pmatrix}1&0\\0&-1\end{pmatrix},\qquad
 X=\begin{pmatrix}0&1\\1&0\end{pmatrix}.
\]
With $s_i=u_i(g)$, the $m^3$ path space has $s_1=Z$ and
$s_2=X$; hence
$ZX=-XZ$ and $Z^2=X^2=1$. On $m^4$ the three generators act as
\[
 (s_1,s_2,s_3)=
 \begin{cases}(Z,X,Z),&t=1,\\(Z,X,-Z),&t=g.\end{cases}
\]
Thus $s_1s_3$ distinguishes the two root blocks, and their
central idempotents are
$z_1=\tfrac12(1+s_1s_3)$ and
$z_g=\tfrac12(1-s_1s_3)$.
In the channel basis, \autoref{eq:four-sandwich} gives
\begin{equation}\label{eq:Ising-four-product}
\zchannels{4}{3/1,2/g,1/1}
 =\frac12\fourunit{1}{1}{1}-\frac12\fourunit{1}{g}{g}.
\end{equation}
In root $1$ its matrix is $\tfrac12\left(\begin{smallmatrix}1&0\\0&0\end{smallmatrix}\right)$;
in root $g$ it is $-\tfrac12\left(\begin{smallmatrix}0&1\\0&0\end{smallmatrix}\right)$.
\end{example}

\subsection{Six-strand blocks}

Continue with $A=G_2=\{1,g\}$, $g^2=1$, and $s_i=u_i(g)$.
Six strands give two free channel labels in each root block.
There are four fusion paths to each root $t\in\{1,g\}$. Denote
their splitting maps by $I_{x,y}^t=|x,y;t\rangle$, where the first
two and first four leaves fuse to $x,y\in\{1,g\}$, respectively:
\begin{equation*}
 I_{1,1}^t=\zsixstate{1}{1}{t}{none},\qquad
 I_{g,1}^t=\zsixstate{g}{1}{t}{none},\qquad
 I_{g,g}^t=\zsixstate{g}{g}{t}{none}.
\end{equation*}
The fourth path is $I_{1,g}^t$. Let $P_{x,y}^t$ be the matching
fusion map, so $P_{x,y}^tI_{z,w}^t=\delta_{x,z}\delta_{y,w}\id_t$.
The root edge is erased when $t=1$.

\begin{lemma}\label{lem:z2-six}
In the ordered basis $\{|1,1;t\rangle,|1,g;t\rangle,
|g,1;t\rangle,|g,g;t\rangle\}$, put
\[
 Z=\begin{pmatrix}1&0\\0&-1\end{pmatrix},\qquad
 X=\begin{pmatrix}0&1\\1&0\end{pmatrix},\qquad
 \kappa_t=\chi(g,t).
\]
Then the five local generators act by
\begin{equation*}
 s_1=Z\otimes1,\quad s_2=X\otimes1,\quad
 s_3=Z\otimes Z,\quad s_4=1\otimes X,\quad
 s_5=\kappa_t(1\otimes Z).
\end{equation*}
In particular, $\End(m^6)\cong\Mat_4(\C)\oplus\Mat_4(\C)$.
\end{lemma}
\begin{proof}
\autoref{lem:path-action} gives the diagonal scalars
$\chi(g,x)$, $\chi(g,xy)$ and $\chi(g,yt)$ for $s_1,s_3,s_5$.
The operators $s_2$ and $s_4$ replace $x$ and $y$ by $gx$ and
$gy$, respectively. Four paths to each root give the block decomposition.
\end{proof}

The even generators change the internal labels and preserve the root:
\begin{equation*}
 \zsixstate{1}{1}{t}{2}=\zsixstate{g}{1}{t}{none},\qquad
 \zsixstate{g}{1}{t}{4}=\zsixstate{g}{g}{t}{none}.
\end{equation*}
Also $s_3I_{g,1}^t=-I_{g,1}^t$, while $s_3I_{g,g}^t=I_{g,g}^t$.

\begin{example}[A rank-one map on six strands]\label{ex:z2-six-unit}
Consider the three disjoint channel projections
\begin{equation*}
 D_t=e_1^{(1)}e_1^{(3)}e_t^{(5)}
 =\zpairprojector{t}
 =I_{1,1}^tP_{1,1}^t.
\end{equation*}
The three channels force $x=1$, $y=1$ and root $t$, respectively. Thus
\[
 V_t=s_2s_4D_t=I_{g,g}^tP_{1,1}^t,\qquad
 W_t=D_ts_4s_2=I_{1,1}^tP_{g,g}^t.
\]
It follows that
\begin{equation*}
 V_t^2=0,\qquad W_tV_t=D_t,\qquad
 V_tW_t=I_{g,g}^tP_{g,g}^t.
\end{equation*}
The product $s_1s_3s_5$ acts by $\kappa_t$ on the whole root
block. Hence its central projections are
 $z_1=\tfrac12(1+s_1s_3s_5)$ and
 $z_g=\tfrac12(1-s_1s_3s_5)$.
\end{example}

\section{Duality and traces}\label{sec:duality}

Can the new red strand bend? A bend has to start or end at a tensor
unit. The answer is therefore local: the TY part on \(A\sqcup\{m\}\)
has unit \(e\), while the full semigroup category may have a different
global unit \(1\).

\subsection{Local and global duals}

Recall that local means on $A$ and global on $S$.

\begin{proposition}[Local and global duals]\label{prop:ambient-duals}
For a monoid $S$, neither $m$ nor any $a\in A$ has an ambient dual
when $A\subsetneq S$. A black simple $s$ has one exactly
when $s\in S^\times$, with dual $s^{-1}$ and the ordinary black
pivotal calculus. Thus $\CS$ is rigid exactly when $S=A$.
For nonunital $S$, $\CS$ has no unit with which to define rigidity;
$\CA$ is rigid with unit $e$.
\end{proposition}

\begin{proof}
A group ideal containing the global $1$ equals $S$. In the proper
case, tensoring $m$ or a label of $A$ with any object has no
global-unit summand. Evaluation to $1$ is zero and cannot satisfy a
snake. For a black label $s$, a nonzero evaluation requires a black
summand $t$ with $ts=1$ or $st=1$. In a finite monoid a one-sided
inverse is two-sided, so $s$ must be a unit. Conversely its inverse
gives \autoref{eq:pointed-bends}. The duals in $\CA$ are described below.
\end{proof}

\begin{example}[The idempotent monoid: the missing bend]
For \(E_2=\{1,e\}\) with \(e^2=e\), there is one nonunit label, \(e\), and its rules are
\begin{equation*}
 \mu=\fuse{e}{e}{e},\qquad \Delta=\splitv{e}{e}{e},\qquad
 \bubble{e}{e}{e}{e}=\wire{e},\qquad
 \channel{e}{e}{e}=\wire{e}\,\wire{e}.
\end{equation*}
The associator and multiplication were described in \autoref{sec:pointed}. But
\[
 \Hom(e\otimes e,1)=\Hom(1,e\otimes e)=0.
\]
The split and merge have an \(e\)-leg, which cannot be erased.
They therefore give no categorical \(e\)-loop.

The same point is even more visible in the restricted red extension
of \autoref{ex:idempotent-restricted}. There \(A=\{e\}\) and
\(m^2=e\), so the only possible red bends are
\[
 \splitv{m}{m}{e}:e\longrightarrow mm,
 \qquad
 \fuse{m}{m}{e}:mm\longrightarrow e.
\]
These are bends for the local unit \(e\), not for the global unit \(1\).
\end{example}

For \autoref{ex:proper-group-ideal}, $mm=e\oplus a$, so
$\Hom_{\CS}(1,mm)=\Hom_{\CS}(mm,1)=0$.

The local red bends are the maps
\begin{equation*}
 \coev_m^A=\splitv{m}{m}{e}:e\longrightarrow mm,\qquad
 \ev_m^A=\tau^{-1}\fuse{m}{m}{e}:mm\longrightarrow e.
\end{equation*}
Write $\ell_m=p_{em}^m$ and $r_m=p_{me}^m$ for local unit comparisons.
The first snake is
\[
 r_m(\id_m\otimes\ev_m^A)\alpha_{m,m,m}
 (\coev_m^A\otimes\id_m)\ell_m^{-1}=\id_m.
\]
The other reverses the bend and uses $\alpha^{-1}$. Including the
local unit comparisons, the two tree diagrams (the local snake identities) are
\begin{equation*}
 \tau^{-1}\coverlap{e}{e}=\wire{m}
 =\tau^{-1}\oppoverlap{e}{e}.
\end{equation*}
In each picture, the cup and cap select the $e$-channel. The selected
Fourier entry is $\tau$, which cancels the cap factor $\tau^{-1}$.

For the bending calculations, we work in $\CA\simeq\Add(\D)$ and
write $1_A=e$ as $1$. To read the pictures in $\CS$, put back the
$e$-boundary and local unit comparisons.

\begin{remark}
Adjunction can be formulated even when no genuine unit object is
present. A bicategorical framework for this is developed in
\cite{KoMaZh-adjunction}; see also \cite{St-adjunction} for the
corresponding theory of adjunctions in semigroup categories.
Our situation is slightly different: the group ideal has its own local
unit, while ambient duality is measured against the global unit.
\end{remark}

\subsection{Cups and caps}

We use left duals, with $\coev_X:1\to X\otimes X^*$ and
$\ev_X:X^*\otimes X\to1$. On the simple labels put
\begin{equation}\label{eq:duality}
 a^*=a^{-1},\quad m^*=m,\qquad
 \coev_a=i_{a,a^{-1}}^1,\quad \ev_a=p_{a^{-1},a}^1,
 \qquad \coev_m=i_1,\quad \ev_m=\tau^{-1}p_1.
\end{equation}
The corresponding cups and caps are
\begin{equation*}
\begin{gathered}
 \cupcap{cup}{m}{m^*}=\splitv{m}{m}{1},
 \cupcap{cap}{m^*}{m}=\tau^{-1}\fuse{m}{m}{1},
 \coev_a=\cupcap{cup}{a}{a^{-1}},
 \ev_a=\cupcap{cap}{a^{-1}}{a}.
\end{gathered}
\end{equation*}

\begin{lemma}[Snake identities]\label{lem:snakes}
The maps \autoref{eq:duality} make $\Add(\D)$ rigid.
\end{lemma}

\begin{proof}
For group labels, use \autoref{lem:pointed-duality}.
In the skeletal TY model, the two $m$-snakes are
\[
\begin{aligned}
 (\id_m\otimes\ev_m)\alpha_{m,m,m}
 (\coev_m\otimes\id_m)&=\id_m,\\
 (\ev_m\otimes\id_m)\alpha_{m,m,m}^{-1}
 (\id_m\otimes\coev_m)&=\id_m.
\end{aligned}
\]
The cup and cap select the $(1,1)$ entries of $F$ and $F^{-1}$,
respectively. Both give $\tau^{-1}\tau\id_m=\id_m$:
\begin{equation}\label{eq:snake-picture}
 \leftsnakepic
 =\wire{m}=
 \rightsnakepic.
\end{equation}
Duals of words reverse their order; duality extends to finite direct sums.
\end{proof}

The factor on the cap is necessary: with cup $i_1$ and cap $p_1$
the snake would be $\tau\id_m$. The resulting raw loop is
\begin{equation*}
\pivloop{m}{none}=\ev_m\coev_m=\tau^{-1}.
\end{equation*}

\subsection{The Temperley--Lieb operators}

We denote by $\mathrm{TL}_n(\delta)$
the usual Temperley--Lieb algebra with parameter $\delta$ and $n$ strands.

\begin{proposition}\label{prop:TL}
For $n\ge2$ let $e_1^{(i)}$ be the unit-channel idempotent on positions
$i,i+1$ of $m^n$, and put
 $U_i=\tau^{-1}e_1^{(i)},\qquad \delta=\tau^{-1}$.
(In ambient notation, $U_i=\tau^{-1}e_e^{(i)}$.) They satisfy
\begin{equation}\label{eq:TL}
 U_i^2=\delta U_i,\qquad
 U_iU_{i+1}U_i=U_i,\qquad
 U_{i+1}U_iU_{i+1}=U_{i+1},\qquad
 U_iU_j=U_jU_i\quad(|i-j|>1).
\end{equation}
They therefore define a representation of $\mathrm{TL}_n(\delta)$ on
$m^n$.
\end{proposition}

\begin{proof}
In the fixed cup--cap normalization, stacking two generators creates
one raw loop:
\[
 \ztl{2}{1,1}=\pivloop{m}{none}\,\ztl{2}{1}
 =\delta\ztl{2}{1}.
\]
For adjacent generators, resolve the three unit channels and use
\autoref{eq:adjacent-diagonal}:
\[
 \ztl{3}{1,2,1}
 =\tau^{-3}\zchannels{3}{1/1,2/1,1/1}
 =\frac{\tau^{-3}}N\zchannels{3}{1/1}
 =\ztl{3}{1}.
\]
Here $N\tau^2=1$. The other sandwich uses the same cancellation
with the right pair fixed. Disjoint cup--cap diagrams commute by
interchange.
\end{proof}

With the cup and cap of \autoref{eq:duality}, the operators are
\begin{equation*}
U_i=\id^{\otimes(i-1)}\otimes\ztl{2}{1}\otimes\id^{\otimes(n-i-1)}.
\end{equation*}
For $N>2$, these operators do not generate the whole endomorphism
algebra: already on two strands, $1,U_1$ span a two-dimensional
subalgebra of the $N$-dimensional algebra $\End(mm)$.

For the order-four datum \autoref{eq:Z4}, these two dimensions are $2$ and $4$.

\subsubsection{The order-two case}

This is closely related to the familiar Ising/Jones--Clifford picture, see e.g.
\cite{EvGa-TY-Potts,IoLeZh-Jones}.
Over $\C$, let \(A=G_2=\{1,g\}\) with \(g^2=1\), \(\chi(g,g)=-1\) and
\(\tau=\nu/\sqrt2\). Put \(\delta=\tau^{-1}=\nu\sqrt2\).
With the normalization above, the cup is \(i_1\), the cap is
\(\delta p_1\), and \(U=\delta e_1\). Hence
\begin{equation}\label{eq:z2-resolution}
 \ztlid{2}=\channel{m}{m}{1}+\channel{m}{m}{g},\qquad
 \channel{m}{m}{g}=\ztlid{2}-\delta^{-1}\ztl{2}{1}.
\end{equation}
The raw red loop is \[\zloop{raw}=\delta.\]
In this case the cup--cap operators already generate the whole red
endomorphism algebra.
Set $U_i=\delta e_1^{(i)}$. The first two Temperley--Lieb relations
take the form
\begin{equation*}
 \ztl{2}{1,1}=\delta\ztl{2}{1},\qquad
 \ztl{3}{1,2,1}=\ztl{3}{1}.
\end{equation*}
On three strands there is one further relation:
\begin{equation*}
 \ztlid{3}
 -\delta\left(\ztl{3}{1}+\ztl{3}{2}\right)
 +\ztl{3}{2,1}+\ztl{3}{1,2}=0.
\end{equation*}
The last two terms are $U_1U_2$ and $U_2U_1$, respectively.
The cups and caps in these drawings have the fixed normalisation
of \autoref{eq:z2-resolution}.

\begin{proposition}\label{prop:z2-TL-quotient}
For $n\ge2$, the algebra $\End(m^n)$ is generated by
$U_1,\ldots,U_{n-1}$, subject to the Temperley--Lieb relations
\autoref{eq:TL} and the local relations
\begin{equation*}
 1-\delta(U_i+U_{i+1})+U_iU_{i+1}+U_{i+1}U_i=0
 \qquad(1\le i\le n-2).
\end{equation*}
Its dimension is $2^{n-1}$.
\end{proposition}
\begin{proof}
Put $s_i=e_1^{(i)}-e_{g}^{(i)}=\delta U_i-1$. \autoref{thm:centralizers}, with $x=y=g$, presents $\End(m^n)$ by
\begin{equation*}
 s_i^2=1,\qquad s_is_{i+1}=-s_{i+1}s_i,\qquad
 s_is_j=s_js_i\quad(|i-j|>1).
\end{equation*}
Since $\delta^2=2$, the equality $s_i^2=1$ is equivalent to
$U_i^2=\delta U_i$, and
\[
 s_is_{i+1}+s_{i+1}s_i
 =2\bigl(1-\delta(U_i+U_{i+1})
          +U_iU_{i+1}+U_{i+1}U_i\bigr).
\]
Far commutativity is preserved by this change of generators.
These relations also imply the sandwich relation: for adjacent
$s=s_i$ and $t=s_{i+1}$,
\[
 (1+s)(1+t)(1+s)=2(1+s),
\]
so $U_iU_{i+1}U_i=U_i$. Hence the two presentations are
equivalent. Their ordered basis is
$s_1^{\epsilon_1}\cdots s_{n-1}^{\epsilon_{n-1}}$ with
$\epsilon_i\in\{0,1\}$.
\end{proof}

The generators and their adjacent relation are
\begin{equation*}
 \zlocal{2}{1/s}=\channel{m}{m}{1}-\channel{m}{m}{g},\qquad
 \zlocal{3}{2/s,1/s}=-\zlocal{3}{1/s,2/s}.
\end{equation*}
The coupon $s$ always acts on the two strands that meet it.

\begin{example}[Reducing a word]\label{ex:z2-word}
On four strands, move $s_1$ past $s_3$ and then past $s_2$:
\[
 s_1s_2s_3s_1s_2s_3
 =s_1s_2s_1s_3s_2s_3
 =-s_2s_3s_2s_3=1.
\]
With all four boundary strands shown,
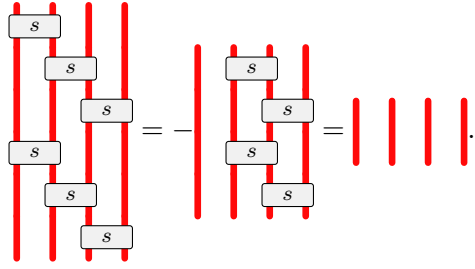
\begin{equation*}
 \zlocal{4}{3/s,2/s,1/s,3/s,2/s,1/s}
 =-\zlocal{4}{3/s,2/s,3/s,2/s}=\ztlid{4}.
\end{equation*}
The word is reduced.
\end{example}

\subsection{Pivotal structures on the local category}

A pivotal structure makes it possible to turn coupons and close
diagrams. For the chosen duality, it is determined by one scalar on $m$.

\begin{proposition}\label{prop:pivotal}
For the duality \autoref{eq:duality}, the pivotal structures are
\begin{equation}\label{eq:pivotal}
 j_a=\id_a\quad(a\in A),\qquad j_m=\rho\id_m,
 \qquad\rho\in k^\times,\quad\rho^2=1.
\end{equation}
All are spherical; there is one choice in characteristic $2$ and two otherwise. Their dimensions are
\begin{equation}\label{eq:dimensions}
 d_a=1,\qquad d_m=\rho\tau^{-1}.
\end{equation}
Over $\C$, the pivotal structure with positive dimensions is
$\rho=\nu$ and has $d_m=\sqrt N$.
\end{proposition}

\begin{proof}
For the duality \autoref{eq:duality}, the tensor comparison
$s^*r^*\to(rs)^*$ has the following coefficients on the $t^*$-channel:
\[
 \begin{array}{c|cccc}
 (r,s;t)&(a,b;ab)&(a,m;m)&(m,a;m)&(m,m;a)\\\hline
 \text{coefficient}&1&1&1&\tau^{-1}
 \end{array}.
\]
The last entry is $\tau^{-2}(F^{-1})_{1,a}=\tau^{-1}$;
the other entries come from the unit and permutation associators.
Each agrees with the coefficient for $(s^*,r^*;t^*)$, so the
double-dual tensor comparison is the identity. Thus identity maps
on simples extend to a pivotal structure. Any other differs by a tensor-compatible automorphism $\gamma$
of the identity functor. Its components on simples are scalars.
The rule $a\otimes m=m$ forces $\gamma_a=1$, and the unit summand
of $m\otimes m$ forces $\gamma_m^2=1$. Conversely, either allowed
value of $\gamma_m$ respects every fusion vertex.

For the right pivotal trace use
\begin{equation*}
 \tr_j(f)=\ev_{X^*}(j_Xf\otimes\id_{X^*})\coev_X.
\end{equation*}
Substituting \autoref{eq:duality} gives \autoref{eq:dimensions}.
The left dimensions are also $1$ on group labels and
$\rho\tau^{-1}$ on $m=m^*$. By semisimplicity it suffices to
compare traces on simple identities, so all these structures are spherical.
\end{proof}

On an $m$-loop, the pivotal factor is drawn as a $\rho$-coupon:
\begin{equation}\label{eq:loop-flag}
\pivloop{m}{\rho}=\rho\tau^{-1},\qquad\pivloop{a}{none}=1.
\end{equation}
Here the cap is $\ev_{m^*}$ and the cup is $\coev_m$.
The right evaluation and coevaluation are
\begin{equation*}
 \widetilde\ev_m=\rho\tau^{-1}p_1,
 \qquad \widetilde\coev_m=\rho i_1.
\end{equation*}
Over $\C$, for $\tau<0$ the raw loop is $-\sqrt N$; the canonical pivotal flag
changes its value to $\sqrt N$.

\begin{remark}[Global dimension and separability]\label{rem:separability}
The spherical dimensions give
\[
 \dim\big(\TY(A,\chi,\tau)\big)=\sum_{a\in A}d_a^2+d_m^2=2N\in k.
\]
A fusion category is called separable when its global dimension
is nonzero~\cite[Introduction]{Et-faithful-lifting}. Here this is
equivalent to $2N\ne0$ in $k$. For odd $N$ in
characteristic $2$, it is semisimple and rigid but the lifting
theorem \cite[Theorem 9.3]{EtNiOs-fusion} does not apply.
(Separability here concerns global dimension, independently of
whether a braiding is nondegenerate.)
\end{remark}

\begin{example}[The two loop normalizations for \(\Z/2\Z\)]
For the canonical pivotal choice \(\rho=\nu\), the raw and spherical
red loops are
\begin{equation*}
 \zloop{raw}=\nu\sqrt2,\qquad
 \zloop{\nu}=\sqrt2.
\end{equation*}
Thus changing the sign of \(\tau\) changes the raw cup--cap loop but
not the positive spherical dimension.
\end{example}

\subsection{Pivotal isotopy and vertex slides}\label{sec:pivotal-isotopy}

We straighten snakes and slide vertices around cups and caps, as in
\cite[(18)--(19)]{RoTu-symmetric-webs}. The chosen vertex bases account
for the scalar factors below.

\begin{lemma}[Straightening bends]\label{lem:snakes2}
All four snake diagrams straighten to the identity. For the right
duality of $m$,
\begin{equation*}
 \isopivsnake{1}
 =\rho^2\wire{m}
 =\wire{m}
 =\isopivsnake{-1}.
\end{equation*}
Straightening preserves the edge labels and whether the bend is left
or right.
\end{lemma}

\begin{proof}
The left snakes are \autoref{eq:snake-picture}, and the corresponding
group-labelled snakes use $a^*=a^{-1}$. The right snakes use
$\widetilde\coev$ and $\widetilde\ev$; their two flags contribute
$\rho^2=1$.
\end{proof}

Put $\epsilon_a=1$ and $\epsilon_m=\rho$. Then
$\widetilde\coev_r=\epsilon_r i_{r^*,r}^1$ and
$\widetilde\ev_r=\epsilon_r\ev_{r^*}$ on simples. For a merge
$p=p_{rs}^t$, bending its left or right input gives
\[
 \mathcal L(p)
 =(\id_{r^*}\otimes p)
   (\widetilde\coev_r\otimes\id_s)
 :s\longrightarrow r^*t,
 \qquad
 \mathcal R(p)
 =(p\otimes\id_{s^*})
   (\id_r\otimes\coev_s)
 :r\longrightarrow ts^*.
\]

\begin{lemma}[Cup and cap slides]\label{prop:isotopy-slides}
The cup slides are
\begin{equation}\label{eq:isotopy-cup-slides}
\begin{aligned}
 \bentvertex{0}{r}{s}{t}
 =\kappa_L(r,s;t)\,\splitv{r^*}{t}{s},\quad
 \bentvertex{1}{r}{s}{t}
 =\kappa_R(r,s;t)\,\splitv{t}{s^*}{r},
\end{aligned}
\end{equation}
where
\begin{equation}\label{eq:isotopy-coefficients}
\begin{array}{c|c|c|c|c}
 r&s&t&\kappa_L(r,s;t)&\kappa_R(r,s;t)\\\hline
 a&b&ab&1&1\\
 a&m&m&1&1\\
 m&a&m&\rho&1\\
 m&m&a&\rho\tau&\tau
\end{array}
\qquad(a,b\in A).
\end{equation}
The inverse cap slides have reciprocal coefficients:
\begin{equation*}
\begin{aligned}
 \bentvertex{2}{r}{s}{t}
 =\kappa_L(r,s;t)^{-1}\,\fuse{r}{s}{t},\quad
 \bentvertex{3}{r}{s}{t}
 =\kappa_R(r,s;t)^{-1}\,\fuse{r}{s}{t}.
\end{aligned}
\end{equation*}
The lower splitting vertices are $i_{r^*,t}^s$ and
$i_{t,s^*}^r$, respectively.
\end{lemma}

\begin{proof}
The two cup slides in \autoref{eq:isotopy-cup-slides} give
\[
 \kappa_L(r,s;t)=\epsilon_r F^{r^*rs;s}_{t,1},
 \qquad
 \kappa_R(r,s;t)
 =\bigl((F^{rss^*;r})^{-1}\bigr)_{t,1}.
\]
The scalar and permutation entries are $1$, while the three-$m$
entry is $\tau$. This gives \autoref{eq:isotopy-coefficients}.
In the first slide the factor $\epsilon_r$ is already part of the
right cup.

Undoing either bend produces a snake, so \autoref{lem:snakes2} gives
the reciprocal coefficients for the cap slides.
\end{proof}

\begin{example}[Bending a channel vertex]\label{ex:isotopy-channel}
The last row of \autoref{eq:isotopy-coefficients} gives, for every
$a\in A$,
\begin{equation*}
 \bentvertex{1}{m}{m}{a}
 =\tau\splitv{a}{m}{m},\qquad
 \bentvertex{0}{m}{m}{a}
 =\rho\tau\splitv{m}{a}{m}.
\end{equation*}
Over $\C$, for $A=G_2$ with $\tau=1/\sqrt2$ and $\rho=1$, both
coefficients are $1/\sqrt2$, and the inverse cap slides have
coefficient $\sqrt2$. With $\tau=-1/\sqrt2$ and the canonical choice
$\rho=-1$, the right and left cup slides instead have coefficients
$-1/\sqrt2$ and $1/\sqrt2$, respectively.
\end{example}

\begin{lemma}[Turning coupons and closing diagrams]
\label{lem:turning-coupons}
Turning a coupon twice gives its double dual. Under the pivotal
identifications,
\begin{equation*}
 \isopivotal{0}=\isopivotal{1},
 \qquad
 j_Yf=f^{**}j_X
\end{equation*}
for every $f:X\to Y$. Hence a full turn preserves the map.

Together with progressive isotopy, interchange,
\autoref{lem:snakes2} and \autoref{prop:isotopy-slides}, coupons may be
moved past cups and caps. Bent labels are dualized and the incident
edge order is preserved. Sphericality identifies left and right
closures.
\end{lemma}

\begin{proof}
The first statement is the naturality of the pivotal structure.
Decomposing a turn or closure into snakes and cup or cap slides gives
the remaining claims. The scalar factors are those of
\autoref{eq:isotopy-coefficients}.
\end{proof}

Loop values are given in \autoref{eq:loop-flag}. Crossings require the
braiding and framed isotopies of \autoref{sec:braiding}.

Even though $m^*=m$, the pivotal factors remain. We can see this by
rotating a cup.

\begin{proposition}[Rotation of the unit cup]\label{prop:cup-rotation}
For $f:1\to mm$, bend its first output by
\[
 h_f=(\ev_m\otimes\id_m)\alpha^{-1}_{m^*,m,m}
       (\id_{m^*}\otimes f):m^*\longrightarrow m,
 \qquad
 \mathcal R_2(f)
 =(\id_m\otimes j_m^{-1})(h_f\otimes\id_{m^{**}})
       \coev_{m^*}.
\]
For the unit cup,
\begin{equation*}
 \mathcal R_2(i_1)=\rho i_1,
 \qquad
 \rotatedcup{\rho}
 =\rho\cupcap{cup}{m}{m}.
\end{equation*}
Thus rotation has trace $\rho$ on $\Hom(1,mm)$. Over $\C$, this is
the usual second Frobenius--Schur indicator, and for the canonical
pivotal structure $\nu_2(m)=\nu$.
\end{proposition}

\begin{proof}
By the snake identity, $h_{i_1}=\id_m$. Since $m^*=m$ and
$j_m^{-1}=\rho\id_m$, the rotated cup is $\rho i_1$.
As $\Hom(1,mm)$ is one-dimensional with basis $i_1$, its trace is
$\rho$. Over $\C$, this agrees with the second indicator of
\cite[Section 2.2]{Shi-TY-indicators}.
\end{proof}

\begin{remark}
The rotation uses only cups, caps and the pivotal structure.
\end{remark}

Over $\C$, for $A=\Z/3\Z$, the two signs of $\tau$ leave the channel matrices
unchanged but give opposite canonical cup rotations.

\subsection{Closing red diagrams}

Restoring the local boundary $e$, closure with the chosen pivotal bends gives,
for $n\ge1$,
\[
 \operatorname{Tr}_j^A:\End_{\CS}(m^n)\longrightarrow
 \End_{\CS}(e)=k\id_e,\qquad f\longmapsto\tr_j(f)\id_e,
\]
where $\tr_j(f)$ denotes its scalar coefficient. The remaining
$e$-edge records that this is a local trace. In scalar calculations
inside $\CA$ we continue to write $e$ as $1$. For the identity on $m$,
\begin{equation*}
 \operatorname{Tr}_j^A(\id_m)
 =\rho\tau^{-1}\bubble{m}{m}{e}{e}=d_m\,\wire{e}.
\end{equation*}
Closing only some strands leaves an endomorphism of the remaining
positive red power. The channel evaluations and cyclicity therefore apply to
$B_n=\End_{\CS}(m^n)$, as does the normalized trace below.

\begin{lemma}\label{lem:channel-trace}
With any pivotal choice,
\begin{equation*}
 \tr_j(e_a)=1,\qquad
 \operatorname{ptr}_{2,j}(e_a)=d_m^{-1}\id_m,
 \qquad \tr_j(M_{a,c})=\delta_{a,c}d_m.
\end{equation*}
Here $\operatorname{ptr}_{2,j}$ closes the second strand of $mm$.
\end{lemma}

\begin{proof}
Close the two strands and move the cut past the splitting vertex.
The split--fusion bubble cancels, leaving the channel strand:
\begin{equation*}
 \closedchannel{a}{\rho}=\pivloop{a}{none}=1.
\end{equation*}
With only the second strand closed, the result belongs to
$\End(m)=k\id_m$. If its scalar is $z$, closing the remaining
strand gives $zd_m=1$; hence $z=d_m^{-1}$. For a three-strand
tree pair the same cyclic cancellation gives
\[
 \ztrace{3}{M_{a,c}}{\rho}=\delta_{a,c}\pivloop{m}{\rho}
 =\delta_{a,c}d_m.
\]
The coupons record the pivotal factors in each closure.
\end{proof}

The partial closure is
\begin{equation*}
\partialchannel{a}{\rho}=\frac1{d_m}\wire{m}.
\end{equation*}
Closing the channel resolution gives $\sum_a\tr_j(e_a)=N=d_m^2$.
Taking traces in \autoref{eq:adjacent} gives
\begin{equation*}
 \tr_j(L_aR_b)=\frac{d_m}{N},\qquad
 \tr_j(L_aR_bL_c)=\delta_{a,c}\frac{d_m}{N}.
\end{equation*}

\begin{example}[Closing an order-two matrix unit]
Work over $\C$ with the canonical pivotal structure.
Take $A=G_2$. For the matrix units, cyclicity
contracts the same two trees:
\begin{equation*}
 \ztrace{3}{M_{g,1}}{\nu}=0,\qquad
 \ztrace{3}{M_{1,1}}{\nu}=\zloop{\nu}=\sqrt2.
\end{equation*}
Thus \(\tr(M_{x,y})=\delta_{x,y}\sqrt2\). In particular a nonzero
off-diagonal matrix unit can have both square and closure zero.
\end{example}

The four-strand operator in \autoref{eq:Ising-four-product} has trace
$1/2$: its off-diagonal term contributes zero. On six strands, the
maps in \autoref{ex:z2-six-unit} satisfy
\[
 \tr(D_t)=1,\qquad \tr(V_t)=0,\qquad \tr(z_1)=\tr(z_a)=4.
\]

\begin{example}[The two loop normalizations]\label{ex:negative-loop}
Work over $\C$.
Take $A=\langle g\mid g^3=1\rangle$, $\tau=-1/\sqrt3$ and the canonical pivotal
choice $\rho=-1$. Then
\[
 \pivloop{m}{none}=-\sqrt3,\qquad
 \pivloop{m}{-1}=\sqrt3,\qquad
 \partialchannel{g}{-1}=\frac1{\sqrt3}\wire{m}.
\]
Closing the last equality gives $1$, the quantum dimension of the
$g$-channel. Resolving two straight strands gives
$d_m^2=\sum_{a\in A}d_a=3$. The unit-cup rotation is
$-i_1$, by \autoref{prop:cup-rotation}.
More generally, the canonical dimensions are
\[
 d(m^{2r})=N^r,\qquad d(m^{2r+1})=N^r\sqrt N,
\]
as one easily sees.
\end{example}

\subsection{Traces of Fourier operators}

Divide the trace on $m^n$ by $d_m^n$ so that the identity has value
$1$. In the ordered Fourier basis, this normalized trace is particularly
simple.

\begin{proposition}[Trace of the ordered basis]\label{prop:Markov}
For $n\ge1$ put $T_n(f)=d_m^{-n}\tr_j(f)$ for $f\in\End(m^n)$. Then
\begin{equation}\label{eq:monomial-trace}
 T_n\bigl(u_1(x_1)\cdots u_{n-1}(x_{n-1})\bigr)
 =\prod_{i=1}^{n-1}\delta_{x_i,1}.
\end{equation}
For $f\in\End(m^n)$ and $b\in A$,
\begin{equation}\label{eq:Markov}
 T_{n+1}\bigl((f\otimes\id_m)e_b^{(n)}\bigr)=N^{-1}T_n(f).
\end{equation}
\end{proposition}
\begin{proof}
Fourier-expand the last two-strand coupon and close its last strand:
\begin{equation*}
\begin{aligned}
 \partialbox{u(x)}{\rho}
 &=\sum_b\chi(x,b)\,\partialchannel{b}{\rho}\\[1ex]
 &=\frac1{d_m}\sum_b\chi(x,b)\wire{m}
 =d_m\delta_{x,1}\wire{m}.
\end{aligned}
\end{equation*}
The penultimate equality uses \autoref{lem:channel-trace}; the
last uses character orthogonality and $d_m^2=N$. In an ordered
monomial, every earlier coupon acts on the remaining strands, so
it factors out of this partial closure. Remove strands from right
to left. Each removal contributes $d_m\delta_{x_i,1}$, and the
last loop contributes $d_m$. Division by $d_m^n$ proves
\autoref{eq:monomial-trace}, including $n=1$.

For \autoref{eq:Markov}, close the final strand of the last channel
projector instead. Its bubble contributes $d_m^{-1}$, so the
unnormalized closure is $d_m^{-1}\tr_j(f)$. Normalization gives
$d_m^{-2}T_n(f)=N^{-1}T_n(f)$.
\end{proof}

For example, the closure of $e_a^{(1)}e_b^{(2)}e_c^{(3)}$ is
$N^{-1}$: in \autoref{eq:four-sandwich}, only $t=ac$ contributes
to the trace, with coefficient $1/N$. For a general closed word
in the $u_i(x)$, order the factors using \autoref{eq:u-near} and then
apply \autoref{eq:monomial-trace}.

\begin{example}[Closing an absorption comparison]
For arbitrary $s\in S$, \autoref{lem:absorption} gives
\[
 \operatorname{Tr}_j^A\bigl(d_s^L(d_s^R)^{-1}\bigr)
 =\sum_{x\in A}\chi\big(\phi(s),x\big)^{-1}\id_e
 =N\delta_{\phi(s),e}\id_e.
\]
For $s=a\in A$, its scalar form in our local notation is
\[
 \tr\bigl(d_a^L(d_a^R)^{-1}\bigr)
 =\sum_x\chi(a,x)^{-1}=N\delta_{a,1}.
\]
Over $\C$, for $A=\langle g\mid g^3=1\rangle$ and
$\chi(g^j,g^k)=\zeta^{jk}$, where $\zeta=e^{2\pi i/3}$,
the comparison for $a=g$ is
$e_1+\zeta^2e_g+\zeta e_{g^2}$; its closure is zero.
\end{example}

\begin{example}[Closures of reduced words]
Work over $\C$ with the canonical pivotal structure.
The commutator in \autoref{eq:Z3-commutator} is $\zeta\id_{m^3}$,
so its closure is $3\sqrt3\,\zeta$.
For $A=\Z/4\Z$ with the datum \autoref{eq:Z4}, put
$s=u_1(1)$ and $t=u_2(1)$ in additive group notation. Then
\[
 st=-i\,ts,\qquad sts^{-1}t^{-1}=-i\id_{m^3},\qquad
 \tr(sts^{-1}t^{-1})=-8i.
\]
For the order-two group, \autoref{ex:z2-word} has closure $d_m^4=4$,
whereas $\tr(s_1s_2s_3)=0$ by \autoref{prop:Markov}.
\end{example}

\section{Braidings, twists and closures}\label{sec:braiding}

Black crossings were already classified in \autoref{sec:semigroup-structures}.
A crossing involving $m$ must also pass through the absorption and
channel vertices.

The local TY braiding classification is due to Siehler
\cite{Si-braided-near-group}; here we classify its extension through
$\phi$ to the ambient semigroup category.

\begin{example}[The idempotent black crossing]
For \(E_2=\{1,e\}\) with \(e^2=e\), the label \(e\) cannot bend, but
it can cross:
\begin{equation*}
 \crossing{e}{e}=\channel{e}{e}{e}=\wire{e}\,\wire{e}.
\end{equation*}
Indeed, \autoref{eq:semigroup-hexagon} gives
\(b(e,e)=b(e,e)^2\), hence \(b(e,e)=1\). Thus crossing and duality
already separate in the smallest idempotent example.
\end{example}

\subsection{Crossings and hexagons}

Our positive crossing has the lower-left strand passing over the
lower-right strand. Write
$R_{r,s}^u\id_u=p_{sr}^u c_{r,s}i_{rs}^u$ for its coefficient in
channel $u$.

An inverse crossing has reversed boundary order and reciprocal
channel scalars. A change of vertex bases as in \autoref{prop:gauge}
changes the crossing coefficients by
\begin{equation*}
 (R')_{r,s}^{u}=\frac{z_{sr}^{u}}{z_{rs}^{u}}R_{r,s}^{u},
\end{equation*}
as follows by substitution in $p_{sr}^u c_{r,s}i_{rs}^u$.

A crossing must pass through a fusion vertex in either direction.
In the strict diagram category, the two hexagons are the local moves
\begin{equation*}
\begin{aligned}
 \hexmove{0}{0}{r}{s}{t}{u}
 =\hexmove{0}{1}{r}{s}{t}{u}
 \bigl(u\in\Out(s,t)\bigr),\quad
 \hexmove{1}{0}{r}{s}{t}{u}
 =\hexmove{1}{1}{r}{s}{t}{u}
 \bigl(u\in\Out(r,s)\bigr).
\end{aligned}
\end{equation*}
In the categorical model, restoring the brackets gives
\begin{equation}\label{eq:hexagons-with-associators}
\begin{aligned}
 c_{r,s\otimes t}
 &=\alpha^{-1}_{s,t,r}(\id_s\otimes c_{r,t})\alpha_{s,r,t}
   (c_{r,s}\otimes\id_t)\alpha^{-1}_{r,s,t},\\
 c_{r\otimes s,t}
 &=\alpha_{t,r,s}(c_{r,t}\otimes\id_s)\alpha^{-1}_{r,t,s}
   (\id_r\otimes c_{s,t})\alpha_{r,s,t}.
\end{aligned}
\end{equation}

An invertible black crossing $st\to ts$ requires $st=ts$, since
distinct simples have zero Hom-space. Thus $S$ is commutative.
On $\CA$, substituting the associators into
\autoref{eq:hexagons-with-associators} gives
$R_{a,b}^{ab}=\chi(a,b)$ and $R_{b,a}^{ab}=\chi(a,b)^{-1}$.
Symmetry of $\chi$ therefore gives $\chi(a,b)^2=1$, whence
$\chi(a^2,b)=1$ for every $b$. Nondegeneracy implies $a^2=e$.

For $N>1$, this forces $A$ to be a nontrivial elementary abelian $2$-group. Since
$N\ne0$ in $k$, we have $\operatorname{char}k\ne2$,
$\chi(a,b)\in\{1,-1\}$ and $F=F^{-1}$.

\subsection{Quadratic refinements and crossings}

The mixed hexagons determine the two absorption crossings by the
same function $q$. Their compatibility with multiplication is the
quadratic refinement identity below. The three-red hexagon will
determine one further scalar $\beta$.

\begin{definition}
A \textbf{quadratic refinement} is a function
$q:A\to k^\times$ such that
\begin{equation}\label{eq:q}
 q(e)=1,\qquad q(ab)=q(a)q(b)\chi(a,b).
\end{equation}
Choose $\beta\in k^\times$ with
\begin{equation}\label{eq:beta}
 \beta^2=\tau G(q),\qquad G(q)=\sum_{a\in A}q(a).
\end{equation}
Putting $b=a$ in \autoref{eq:q} gives $q(a)^2=\chi(a,a)$.
\end{definition}

Refinements exist over $k$. Choose a basis $a_1,\ldots,a_\ell$ of
$A$ over $\mathbb F_2$ and scalars $z_i^2=\chi(a_i,a_i)$. Then
\[
 q\left(\prod_i a_i^{x_i}\right)
 =\prod_i z_i^{x_i}\prod_{i<j}\chi(a_i,a_j)^{x_ix_j},
 \qquad x_i\in\{0,1\},
\]
satisfies \autoref{eq:q}. The ratio of two refinements is a character,
so there are $|A|$ refinements for fixed $\chi$.

\begin{lemma}\label{lem:Gauss}
The refinement satisfies $q(a)^4=1$, and
\begin{equation*}
 \sum_b q(b)\chi(t,b)=\frac{G(q)}{q(t)},\qquad
 G(q^{-1})G(q)=N.
\end{equation*}
In particular \autoref{eq:beta} has two solutions if $\operatorname{char}k\ne2$, and one in characteristic $2$ (where $A=\{e\}$).
\end{lemma}

\begin{proof}
The values follow from $q(a)^2=\chi(a,a)$. For the first identity use
$q(tb)=q(t)q(b)\chi(t,b)$ and sum over $b$. For the second, substitute
$a=tb$ in $G(q)G(q^{-1})$ to obtain
\[
 \sum_{t,b}q(tb)q(b)^{-1}
 =\sum_t q(t)\sum_b\chi(t,b)=N.
\]
Since $N\ne0$, the Gauss sum is nonzero. Algebraic closure gives
two roots for $\beta$ outside characteristic $2$, and one in characteristic $2$.
Over $\C$, $q^{-1}=\overline q$ also gives $|G(q)|^2=N$.
\end{proof}

For commutative $S$ and $s,t\in S$, the candidate
crossings determined by $q,\beta$ are
\begin{align*}
 \crossing{s}{t}&=\chi\big(\phi(s),\phi(t)\big)\resolvedcross{s}{t}{st},\quad
 \crossing{s}{m}=q\big(\phi(s)\big)\resolvedcross{s}{m}{m},\\[1ex]
 \crossing{m}{s}&=q\big(\phi(s)\big)\resolvedcross{m}{s}{m},\quad
 \crossing{m}{m}=\sum_{a\in A}\beta q(a)^{-1}\channel{m}{m}{a}.
\end{align*}
\begin{theorem}[Braiding classification]\label{thm:braid}
The category $\CS$ admits a braiding precisely when $S$ is commutative
and $a^2=e$ for every $a\in A$. All braidings in the fixed normal
gauge are parametrized by \autoref{eq:q}--\autoref{eq:beta}, with
\begin{equation}\label{eq:R}
\begin{gathered}
 R_{s,t}^{st}=\chi\big(\phi(s),\phi(t)\big),\qquad
 R_{s,m}^{m}=R_{m,s}^{m}=q\big(\phi(s)\big)\quad(s,t\in S),\\
 R_{m,m}^{a}=r(a):=\beta q(a)^{-1}\quad(a\in A).
\end{gathered}
\end{equation}
When $S$ is commutative, restriction to $\CA$ is a bijection on braidings. For $N>1$ there
are $2|A|$ choices for fixed data, labels and normal basis, before equivalences.
\end{theorem}

\begin{proof}
On $\CA$, the two absorptions have the same scalar $q$. The other mixed equations give
\autoref{eq:q} and
\begin{equation}\label{eq:mixed-hex}
 q(a)r(ab)=\chi(a,b)r(b).
\end{equation}
Setting $b=e$ gives $r(a)=\beta/q(a)$ with $\beta=r(e)$. Indeed,
$q(a)\beta/q(ab)=\beta\chi(a,b)/q(b)$.

For three $m$-strands, put $Q=\operatorname{diag}(q(a))$ and
$R=\operatorname{diag}(r(a))$. The hexagon becomes
\begin{equation}\label{eq:hex-mmm}
 FQF=RFR.
\end{equation}
Entry by entry, \autoref{lem:Gauss} gives
\begin{equation*}
 (FQF)_{a,c}=\tau^2\sum_b q(b)\chi(ac,b)
 =\frac{\tau^2G(q)}{q(ac)}
 =\frac{\tau\beta^2\chi(a,c)}{q(a)q(c)}=(RFR)_{a,c}.
\end{equation*}
Conversely, its $a=c=e$ entry forces \autoref{eq:beta}.
Substitution for the seven input types other than $mmm$ gives
\autoref{eq:q}, $q(a)^2=\chi(a,a)$ and \autoref{eq:mixed-hex}.
This gives the TY crossings of \cite[Theorem 1.2]{Si-braided-near-group}.

Write $b(s,t)=R_{s,t}^{st}$, $L(s)=R_{s,m}^m,R(s)=R_{m,s}^m$. The first hexagon on $(s,t,m)$, with $tm=m$, is
\begin{equation*}
 L(s)=L(s)b(s,t)\chi\big(\phi(t),\phi(s)\big)^{-1}.
\end{equation*}
Its crossing-through-fusion diagram is
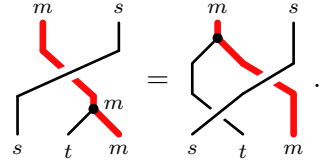
\[
 \hexmove{0}{0}{s}{t}{m}{m}
 =\hexmove{0}{1}{s}{t}{m}{m}.
\]
Since $L(s)\ne0$, symmetry gives
$b(s,t)=\chi(\phi(s),\phi(t))$ for all $s,t\in S$.
The other mixed hexagons give
\begin{equation*}
\begin{aligned}
 L(st)&=L(s)L(t)\chi\big(\phi(s),\phi(t)\big)^{-1},\quad
 R(st)&=R(s)R(t)\chi\big(\phi(s),\phi(t)\big).
\end{aligned}
\end{equation*}
Set $s=t=e$ to obtain $L(e)=R(e)=1$, then set $t=e$ to obtain
$L(s)=L(\phi(s))$ and $R(s)=R(\phi(s))$.
Both local values are $q$, proving \autoref{eq:R}.

Conversely, for commutative $S$, formula \autoref{eq:R} gives natural
invertible crossings with the required ambient boundaries.
Under the faithful functor $F_\phi$ every hexagon is a local TY
hexagon, hence holds in $\CS$. For a monoid, $\phi(1)=e$ gives
the unit crossing identities.
\end{proof}

\begin{example}[A hexagon coefficient for $\Z/2\Z$]\label{ex:Ising-hexagon}
Work in $\CA$ over $\C$, with $A=G_2=\{1,g\}$ and $g^2=1$.
Let $\chi(g,g)=-1$, $\tau=1/\sqrt2$, $q(g)=-i$ and
$\beta=e^{-\pi i/8}$. Resolve the two routes through the unit
initial and final channels. Their coefficients are
\[
 \underbrace{\frac12q(1)+\frac12q(g)}_{\text{recouple, cross, recouple}}
 =\frac{1-i}{2}
 =\underbrace{\beta\,\frac1{\sqrt2}\,\beta}_{\text{cross, recouple, cross}}.
\]
These are the $(1,1)$ entries of \autoref{eq:hex-mmm};
both intermediate channels contribute.
\end{example}

\begin{example}[Crossing a label outside the group ideal]
Use the datum of \autoref{ex:proper-group-ideal}, with
$q(a)=-i$ and $\beta=e^{-\pi i/8}$. Since $\phi(s)=a$,
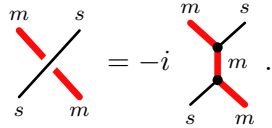
\[
 \crossing{s}{m}=-i\resolvedcross{s}{m}{m}.
\]
This crossing is defined in $\CS$ even though $m$ has no ambient dual.
\end{example}

\begin{example}[The one-channel red crossing]
Return to the restricted idempotent extension
\(e^2=e\), \(em=me=m\), \(m^2=e\).
If \(\alpha_{m,m,m}=\tau\id_m\), with \(\tau^2=1\), then
\[
 \crossing{m}{m}=\beta\channel{m}{m}{e},
 \qquad \beta^2=\tau.
\]
Crossings involving \(e\) or the global unit have coefficient \(1\).
The red strand is therefore braided, but its local cup still has source
\(e\), not the global unit \(1\).
\end{example}

The order-three data of \autoref{ex:Z3-recoupling} and the order-four
datum \autoref{eq:Z4} therefore give unbraided categories: their groups
contain elements whose square is not the identity.

\subsection{Braid matrices}

Next up:

\begin{proposition}[Three-strand braid matrices]\label{prop:braid-matrices}
In the splitting-tree basis $\{I_a^L:m\to m^3\}_{a\in A}$, let
$b_1=c_{m,m}\otimes\id_m$ and $b_2=\id_m\otimes c_{m,m}$.
Their matrices are
\begin{equation*}
 (b_1)_{a,c}=\delta_{a,c}\frac{\beta}{q(a)},\qquad
 b_2=F^{-1}RF,\qquad (b_2)_{a,c}=\frac{\tau}{\beta}q(ac).
\end{equation*}
They satisfy the braid relation.
\end{proposition}

\begin{proof}
The first crossing acts on the left pair and contributes
$r(a)$. For the second, apply the local projector reduction
\autoref{eq:graphical-local-projector} to each crossing channel:
\[
 (b_2)_{a,c}=\frac\beta N\sum_bq(b)^{-1}\chi(ac,b)
 =\frac\beta N G(q^{-1})q(ac)=\frac\tau\beta q(ac).
\]
For the three-crossing words, resolve their trees. The first
has coefficient $r(a)(\tau/\beta)q(ac)r(c)\allowbreak=\tau\beta\chi(a,c)$.
The second has coefficient
\[
 \frac{\tau^2}{\beta}\sum_b\frac{q(ab)q(bc)}{q(b)}
 =\frac{\tau^2}{\beta}q(a)q(c)\frac{G(q)}{q(ac)}
 =\tau\beta\chi(a,c).
\]
Both expansions therefore give
\begin{equation*}
 \braidthree{1,2,1}
 =\sum_{a,c}\tau\beta\chi(a,c)\matrixunit{a}{c}
 =\braidthree{2,1,2}.
\end{equation*}
The proof is complete.
\end{proof}

\autoref{lem:path-action} and \autoref{lem:Gauss} give
\begin{align*}
 b_{2j-1}|\mathbf a\rangle
 =\frac{\beta}{q(a_{j-1}^{-1}a_j)}|\mathbf a\rangle,
 \quad
 b_{2j}|\mathbf a\rangle
 =\frac{\tau}{\beta}\sum_c q(ca_j)
 |a_1,\ldots,a_{j-1},c,a_{j+1},\ldots\rangle.
\end{align*}
Here every group element has order at most two.

For the order-two examples work over $\C$ with $A=G_2=\{1,g\}$,
$g^2=1$, $\chi(g,g)=-1$, $\tau=1/\sqrt2$, and fix
\[
 \nu=1,\qquad q(g)=-i,\qquad \beta=e^{-\pi i/8}.
\]
Thus $\delta=\sqrt2$. The four nontrivial crossing types resolve as
\begin{equation*}
\begin{gathered}
 \crossing{g}{g}=-\resolvedcross{g}{g}{1},\qquad
 \crossing{g}{m}=-i\resolvedcross{g}{m}{m},\qquad
 \crossing{m}{g}=-i\resolvedcross{m}{g}{m},\\[1ex]
 \crossing{m}{m}=\beta\bigg(
   \channel{m}{m}{1}+i\channel{m}{m}{g}\bigg).
\end{gathered}
\end{equation*}
The two mixed resolutions are maps between different words; their
vertices record the required absorption maps.

\begin{example}[Three-strand braid matrices]
Let \(b_i\) be the positive crossing on positions \(i,i+1\).
In the splitting basis \((I_1^L,I_g^L)\), \(b_1=R\) and
\begin{equation*}
 R=\begin{pmatrix}e^{-\pi i/8}&0\\0&e^{3\pi i/8}\end{pmatrix},\qquad
 b_2=\frac{e^{-\pi i/8}}2
 \begin{pmatrix}1+i&1-i\\1-i&1+i\end{pmatrix}.
\end{equation*}
For instance,
 \(b_2I_1^L=\frac{e^{-\pi i/8}}2((1+i)I_1^L+(1-i)I_g^L)\).
\end{example}

\begin{lemma}[A skein calculation]\label{lem:z2-skein}
The positive and negative $m$-crossings have the cup--cap
expansions
\begin{equation*}
\begin{aligned}
 \crossing{m}{m}
 =\beta\left(i\ztlid{2}+\frac{1-i}{\sqrt2}\ztl{2}{1}\right),\quad
 \zinversecross
 =\beta^{-1}\left(-i\ztlid{2}+\frac{1+i}{\sqrt2}\ztl{2}{1}\right).
\end{aligned}
\end{equation*}
The two crossings cancel:
\begin{equation*}
 \zsignedbraid{1,-1}=\ztlid{2}.
\end{equation*}
\end{lemma}

\begin{proof}
Substitute $e_1=U/\sqrt2$ and $e_{g}=1-U/\sqrt2$ into
$c_{m,m}=\beta(e_1+ie_{g})$ and
$c_{m,m}^{-1}=\beta^{-1}(e_1-ie_{g})$.
Multiplying the two expressions, the coefficient of $1$ is $1$
and the coefficient of $U$, after using $U^2=\sqrt2U$, is
 $\frac{i(1+i)-i(1-i)}{\sqrt2}
 +\frac{(1-i)(1+i)}2\sqrt2=0$.
\end{proof}

\begin{example}[Two crossings change a channel]\label{ex:z2-braid-square}
Let \(b_i\) be the positive crossing on positions \(i,i+1\).
Since the square of its \(g\)-channel scalar is \(-\beta^2\),
\begin{equation*}
 \zsignedbraid{1,1}
 =\beta^2\bigg(\channel{m}{m}{1}-\channel{m}{m}{g}\bigg),
 \qquad b_i^2=\beta^2s_i.
\end{equation*}
On three strands, \(s_2\) exchanges the two left-comb channels.
Thus
\begin{equation*}
 \zbraidedstate{1}=\beta^2\leftsplit{m}{m}{m}{g}{m},\qquad
 \zbraidedstate{g}=\beta^2\leftsplit{m}{m}{m}{1}{m}.
\end{equation*}
A further pair returns to the original channel, with scalar
\(\beta^4=-i\).
\end{example}

\subsection{Balancing and ribbon twists}

The twist on a tensor product must agree with twisting its factors
and then crossing them twice. A \emph{semigroupal balancing} is a
natural automorphism
$\theta$ of $\mathrm{Id}_{\CS}$ satisfying
\begin{equation}\label{eq:balance}
 \theta_{X\otimes Y}
 =c_{Y,X}c_{X,Y}(\theta_X\otimes\theta_Y).
\end{equation}
For a monoid require $\theta_1=\id_1$. Balancing also forces
$\theta_e=\id_e$ without a global unit.
Draw a positive full twist as a $\theta$-coupon. On a simple
channel $u\subset r\otimes s$, naturality and \autoref{eq:balance} give
\begin{equation*}
\twistfusionpic
 =
\balancedfusionpic.
\end{equation*}
Its scalar form is $\theta_u=\theta_r\theta_sR_{r,s}^uR_{s,r}^u$.

\begin{proposition}[Balancing and ribbon structures]\label{prop:twists}
For $N>1$, each ambient braiding $(q,\beta)$ admits exactly two
balancing twists:
\begin{equation*}
 \theta_s=\chi\big(\phi(s),\phi(s)\big)\id_s\quad(s\in S),\qquad
 \theta_m=\rho\beta^{-1}\id_m,\qquad\rho^2=1.
\end{equation*}
Their restrictions to $\CA$ are ribbon and induce the local pivotal
structures \autoref{eq:pivotal}. If $S$ is a monoid, $\CS$ is ribbon exactly when $S=A$.
For nonunital $S$, the ribbon assertion concerns only $\CA$.
Over $\C$, the local pivotal choice with positive dimensions has
\begin{equation}\label{eq:canonical-twist}
 \theta_m^{\mathrm{can}}=\nu\beta^{-1}\id_m.
\end{equation}
\end{proposition}

\begin{proof}
On $s\otimes m=m$, balancing forces
$\theta_m=\theta_s\theta_mq(\phi(s))^2$, hence
$\theta_s=\chi(\phi(s),\phi(s))$. In particular $\theta_e=1$.
The $e$-channel of $mm$ gives $1=\theta_m^2\beta^2$.
Conversely the diagonal bicharacter is multiplicative on $S$:
its extra factor is $\chi(\phi(s),\phi(t))^2=1$.
The black double crossing is $1$, absorption obeys the displayed equation,
and every red-red channel satisfies
\[
 \theta_m^2r(a)^2=q(a)^{-2}=\chi(a,a)=\theta_a.
\]
Thus all balancing equations hold; $\phi(1)=e$ gives the global-unit
condition when needed.

Every simple of $\CA$ is self-dual. Its duality is
$k$-linear, so $(\lambda\id_X)^*=\lambda\id_{X^*}$ and
$\theta_{X^*}=(\theta_X)^*$. These restrictions are ribbon.
\autoref{prop:ambient-duals} gives the ambient rigidity
obstruction.

Using the notation of \autoref{sec:duality}, we identify
the pivotal structure via the Drinfeld map
\[
 u_X=(\ev_X\otimes\id_{X^{**}})
 (c_{X,X^*}\otimes\id_{X^{**}})
 \alpha^{-1}_{X,X^*,X^{**}}
 (\id_X\otimes\coev_{X^*}).
\]
The cup and cap select the unit channel, giving
\begin{equation*}
 u_m=\tau^{-1}r(1)(F^{-1})_{1,1}\id_m=\beta\id_m.
\end{equation*}
For $a$, we obtain $u_a=\chi(a,a)\id_a$. Since
$j_X=u_X\theta_X$, this gives $j_a=1,j_m=\rho$.
The canonical choice follows from \autoref{prop:pivotal}.
\end{proof}

The twist relations are
\begin{equation*}
 \twistpic{s}=\chi\big(\phi(s),\phi(s)\big)\wire{s},\qquad
 \twistpic{m}=\rho\beta^{-1}\wire{m}.
\end{equation*}
The Drinfeld map in $\CA$ has output $m^{**}=m$:
\begin{equation*}
 \drinfeldpic
 =\beta\wire{m}.
\end{equation*}
Closing the second strand of a red crossing gives the twist:

\begin{lemma}[Positive curl]\label{lem:positive-curl}
For either local pivotal choice, the partial closure lies in
$\End_{\CS}(m)$ and satisfies
\begin{equation*}
 \operatorname{ptr}_{2,j}(c_{m,m})=\theta_m
 =\rho\beta^{-1}\id_m.
\end{equation*}
\end{lemma}

\begin{proof}
Resolve the crossing and apply \autoref{lem:channel-trace}:
\[
 \operatorname{ptr}_{2,j}(c_{m,m})
 =\frac{\beta}{d_m}\sum_a q(a)^{-1}\id_m
 =\rho\tau\beta\frac{N}{G(q)}\id_m
 =\rho\beta^{-1}\id_m.
\]
We used $G(q)G(q^{-1})=N$, $\beta^2=\tau G(q)$ and $\tau^2N=1$.
The diagram, with top cap $\ev_{m^*}$ preceded by $j_m$, is
\begin{equation}\label{eq:positive-curl-picture}
\begin{aligned}
 \positivecurl
 &=\sum_a\frac\beta{q(a)}\,\partialchannel{a}{\rho}\\[1ex]
 &=\frac\beta{d_m}G(q^{-1})\wire{m}
 =\rho\beta^{-1}\wire{m}=\twistpic{m}.
\end{aligned}
\end{equation}
The proof is complete.
\end{proof}

Over $\C$, for $A=G_2=\{1,g\}$, $g^2=1$, and $\tau=\nu/\sqrt2$,
the braidings are specified by
\[
 q(g)\in\{i,-i\},\qquad \beta^2=\tau(1+q(g)).
\]
For either choice of $\beta$, the crossing and canonical twist
on two $m$-strands and on a simple strand are
\begin{equation*}
 c_{m,m}=\beta(e_1+q(g)^{-1}e_g),\qquad
 \theta_g=-1,\qquad\theta_m=\nu\beta^{-1}.
\end{equation*}
The choices of $\nu$ and $q(g)$ give the following canonical twists.
Each row also has the simultaneous sign change of $\beta$ and
$\theta_m^{\mathrm{can}}$.
\begin{center}
\begin{tabular}{cccc}
\toprule
$\nu$ & $q(g)$ & one choice of $\beta$ & $\theta_m^{\mathrm{can}}$\\
\midrule
$1$ & $-i$ & $e^{-\pi i/8}$ & $e^{\pi i/8}$\\
$1$ & $i$ & $e^{\pi i/8}$ & $e^{-\pi i/8}$\\
$-1$ & $-i$ & $e^{3\pi i/8}$ & $e^{5\pi i/8}$\\
$-1$ & $i$ & $e^{5\pi i/8}$ & $e^{3\pi i/8}$\\
\bottomrule
\end{tabular}
\end{center}
Changing the pivotal structure reverses both $d_m$ and $\theta_m$.
Changing $\beta$, on the other hand, changes the braiding.

For the choice $\nu=\rho=1$, $q(g)=-i$ and
$\beta=e^{-\pi i/8}$ used above, the twists are
\[
 \twistpic{g}=-\wire{g},\qquad
 \twistpic{m}=\beta^{-1}\wire{m}.
\]

\begin{example}[The three-strand full twist]\label{ex:z2-full-twist}
Put $h=(1+i)/2$ and $k=(1-i)/2$, so that
$b_i=\beta(h+ks_i)$. Using
$s_1s_2=-s_2s_1$ gives
\[
 \beta^{-3}b_1b_2b_1
 =h(h^2+k^2)+2h^2k\,s_1+k(h^2-k^2)s_2.
\]
Now $h^2+k^2=0$ and $2h^2k=k(h^2-k^2)=h$. Interchanging
$s_1$ and $s_2$ therefore gives
\begin{equation}\label{eq:z2-braid-triple}
 b_1b_2b_1=\frac{(1+i)\beta^3}{2}(s_1+s_2)
          =b_2b_1b_2.
\end{equation}
Since $(s_1+s_2)^2=2$, squaring \autoref{eq:z2-braid-triple} gives
$i\beta^6=\beta^2$, because $\beta^4=-i$. Diagrammatically,
\begin{equation*}
 \braidthree{2,1,2,1,2,1}=\beta^2\ztlid{3}.
\end{equation*}
Balancing gives the same answer: every summand of $m^3$ is $m$,
so the full twist acts by $\theta_m/\theta_m^3=\beta^2$, with the
denominator accounting for the three individual ribbon twists.
\end{example}

\begin{example}[The one-channel case]
For $A=\{e\}$, a braiding still requires $S$ commutative.
The hexagons give $b(s,t)=L(s)=R(s)=1$, $q(e)=1$, $\beta^2=\tau$,
$\tau^2=1$ and $\theta_m=\rho\beta^{-1}$ with $\rho^2=1$.
There are two choices of $\beta$ and of $\rho$ outside characteristic
$2$; in characteristic $2$, $\tau=\beta=\rho=1$.
For \(E_2=\{1,e\}\) with \(e^2=e\) and \(mm=e\), the category is braided
and balanced; its red cup has source \(e\), so it is not ribbon.
\end{example}

\begin{example}[The trivial group]
Work over $\C$.
If $A=\{1\}$, the same formulas still make sense, but $m^2=1$ and $m$
is invertible. Here $\tau=\nu=\pm1$, $F=(\tau)$, $q(1)=1$ and
$\beta^2=\tau$. The canonical twist is
$\theta_m=\tau\beta^{-1}=\beta$. Thus $\tau=1$ gives
$\beta=\pm1$, whereas $\tau=-1$ gives $\beta=\pm i$. The latter
choices have $c_{m,m}^2=-1$. In particular, the factor $\nu$ in
\autoref{eq:canonical-twist} is necessary even in this case.
\end{example}

\subsection{Local Hopf links}

Closing a double crossing gives a Hopf link. We evaluate it with
the local trace $\operatorname{Tr}_j^A$. The double crossings,
also called monodromies, resolve as
\begin{equation}\label{eq:monodromy-pictures}
 \doublecross{a}{b}=\wire{a}\,\wire{b},\qquad
 \doublecross{a}{m}=\chi(a,a)\wire{a}\,\wire{m},\qquad
 \doublecross{m}{m}=\beta^2\sum_a\chi(a,a)\channel{m}{m}{a}.
\end{equation}

An object is \emph{transparent} if its double crossing with every
object is the identity.

\begin{proposition}\label{prop:Hopf}
Choose $\rho^2=1$ and let
$S_{X,Y}=\tr(c_{Y,X}c_{X,Y})$, the unnormalized Hopf link value. Then
\begin{equation*}
 S_{a,b}=1,\qquad S_{a,m}=S_{m,a}=d_m\chi(a,a),\qquad
 S_{m,m}=\beta^2\sum_{a\in A}\chi(a,a).
\end{equation*}
The last sum is $N$ if $\chi$ is alternating, and is $0$ otherwise.
An invertible label $a$ is transparent precisely when $\chi(a,a)=1$.
\end{proposition}
\begin{proof}
Close the three relations \autoref{eq:monodromy-pictures}. The trace of
each $e_a$ is $1$, whereas the trace of $\id_{am}$ is $d_m$.
The character $a\mapsto\chi(a,a)$ sums to $N$ if trivial and to
zero otherwise. An invertible object has trivial monodromy with
all invertibles. Its transparency is therefore determined by
monodromy with $m$, as in the middle relation of
\autoref{eq:monodromy-pictures}.
\end{proof}

Writing $D_{X,Y}=c_{Y,X}c_{X,Y}$, the closure is
\begin{equation*}
 \hopfclosurepic
 =S_{X,Y}.
\end{equation*}
The return strands are duals; the coupons are $j_X,j_Y$.

For $A=G_2=\{1,g\}$ with $\tau=1/\sqrt2$, take $\rho=1$.

\begin{example}[The Hopf link matrix]
The Hopf values in the order $(1,g,m)$ are
\begin{equation*}
 S=\begin{pmatrix}1&1&\sqrt2\\1&1&-\sqrt2\\\sqrt2&-\sqrt2&0\end{pmatrix}.
\end{equation*}
The monodromy of $g$ with $m$ is $-\id_m$, excluding both from
the transparent simples. Thus the category is nondegenerate; also
$\det S=-8$. 
\end{example}

\subsection{Local two-strand braid closures}

Let $H_k$ be the scalar obtained by closing $c_{m,m}^k$ with pivotal bends.
For $k>0$, it closes the positive two-strand braid with the framing shown:
\begin{equation*}
 H_1=\closedtwobraid{1},\qquad
 H_2=\closedtwobraid{2},\qquad
 H_3=\closedtwobraid{3}.
\end{equation*}

\begin{proposition}\label{prop:torus}
For every integer $k$,
\begin{equation}\label{eq:torus-values}
 H_k=\beta^k\sum_a q(a)^{-k}
 =\beta^k
 \begin{cases}
 N,&k\equiv0\pmod4,\\
 G(q^{-1}),&k\equiv1\pmod4,\\
 \sum_a\chi(a,a),&k\equiv2\pmod4,\\
 G(q),&k\equiv3\pmod4.
 \end{cases}
\end{equation}
For either local ribbon structure, $H_1=d_m\theta_m$. The zero-framed
positive trefoil, coloured by $m$, has value
\begin{equation*}
 \theta_m^{-3}H_3=d_m\,\beta^8.
\end{equation*}
These are unnormalized values: the zero-framed unknot has value
$d_m$. Over $\C$, the canonical choice gives $d_m=\sqrt N$.
\end{proposition}
\begin{proof}
By orthogonality, $c_{m,m}^k=\sum_a(\beta/q(a))^k e_a$,
including negative $k$. Taking the trace uses $\tr(e_a)=1$;
the four cases then follow from $q(a)^4=1$.
The identities $G(q^{-1})=N/G(q)$ and $\beta^2=\tau G(q)$ give
$H_1=\tau N/\beta=\tau^{-1}/\beta=d_m\theta_m$.
The three-crossing closure is a trefoil with framing $3$.
Multiplication by $\theta_m^{-3}$ removes the framing:
\[
 \theta_m^{-3}H_3
 =\rho\beta^3\,\beta^3G(q)
 =\frac{\rho}{\tau}\beta^8=d_m\,\beta^8.
\]
The proof is complete.
\end{proof}

For even $k$, the two components have no self-crossings, so no
framing correction is needed: $H_2$ is the zero-framed Hopf value,
and $H_4=N\beta^4$ is the value at linking number $2$.

\begin{example}[From a crossing to the zero-framed unknot]\label{ex:unknot-reduction}
Work over $\C$.
For the Ising choice in \autoref{ex:Ising-hexagon}, use the
canonical pivotal structure $\rho=1$. Close a single positive
crossing and resolve its two channels:
\begin{equation*}
\begin{aligned}
 \closedtwobraid{1}
 &=\beta\left(\closedchannel{1}{1}+i\closedchannel{g}{1}\right)\\[1ex]
 &=\beta(1+i)=\sqrt2e^{\pi i/8}.
\end{aligned}
\end{equation*}
This local ribbon closure is an unknot with framing $+1$. Remove that framing by the
inverse twist scalar:
\[
 e^{-\pi i/8}\closedtwobraid{1}=\sqrt2=\pivloop{m}{1}.
\]
The zero-framed unknot has value $d_m=\sqrt2$, or $1$ after
division by $d_m$.
\end{example}

For the order-two datum $q(g)=-i$, $\beta=e^{-\pi i/8}$, take
$\rho=1$. The pivotal coupons are then identities.

\begin{lemma}\label{lem:z2-closure}
For every integer $r$,
\begin{equation*}
 H_r=\beta^r(1+i^r).
\end{equation*}
\end{lemma}

\begin{proof}
Orthogonality gives $c_{m,m}^r=\beta^r(e_1+i^re_g)$, also for
negative $r$, since both eigenvalues are nonzero. Each channel
has trace $1$.
\end{proof}

\begin{example}[Hopf and trefoil closures]\label{ex:Ising-trefoil}
For two crossings, the two channels cancel:
\begin{equation*}
 \closedtwobraid{2}
 =\beta^2\left(\ztrace{2}{e_1}{1}-\ztrace{2}{e_{g}}{1}\right)
 =\beta^2(1-1)=0.
\end{equation*}
This is the Hopf link with both components labelled $m$.
The one- and three-crossing closures give
\begin{equation*}
 \closedtwobraid{1}=\sqrt2\,e^{\pi i/8},\qquad
 \closedtwobraid{3}=\sqrt2\,e^{-5\pi i/8}.
\end{equation*}
The first is a framing-$1$ unknot, with value
$\theta_m d_m=\sqrt2e^{\pi i/8}$. The second is a framing-$3$
trefoil; removing the framing gives
\begin{equation*}
 \theta_m^{-3}\closedtwobraid{3}
 =e^{-3\pi i/8}\sqrt2e^{-5\pi i/8}=-\sqrt2.
\end{equation*}
Dividing by the unknot value gives $-1$.

Four crossings give two components with linking number $2$ and
zero self-framing. The two channels add:
\begin{equation*}
 \closedtwobraid{4}
 =\beta^4\left(\ztrace{2}{e_1}{1}+\ztrace{2}{e_{g}}{1}\right)
 =-2i.
\end{equation*}
More generally, the two channels cancel whenever $r\equiv2\pmod4$.
Since $\beta^8=-1$ and $i^8=1$, the framed closure values satisfy
\[
 H_{4k+2}=0,\qquad H_{r+8}=-H_r,\qquad H_{r+16}=H_r
 \qquad(k,r\in\Z),
\]
as one easily checks.
\end{example}

\subsection{Quadratic refinements in coordinates}

Here \(k=\C\).
Let $A=(\Z/2\Z)^n$ and let $B$ be an invertible symmetric binary
matrix with $\chi(x,y)=(-1)^{x^{\mathsf T}By}$. The following formula
lists every quadratic refinement. Binary coordinates are represented
by the integers $0,1$.

\begin{lemma}\label{lem:binary-q}
All refinements of $\chi$ are
\begin{equation*}
 q_\ell(x)=
 i^{\sum_j B_{jj}x_j}
 (-1)^{\sum_{j<k}B_{jk}x_jx_k+\ell^{\mathsf T}x},
 \qquad\ell\in(\Z/2\Z)^n.
\end{equation*}
For each refinement there are two choices of $\beta$ in the fixed
normal gauge.
\end{lemma}

\begin{proof}
For a diagonal term use
$x_j\mathbin{\oplus}y_j=x_j+y_j-2x_jy_j$ to see that its contribution
to $q(x+y)/(q(x)q(y))$ is $(-1)^{B_{jj}x_jy_j}$. An off-diagonal
term contributes $(-1)^{B_{jk}(x_jy_k+x_ky_j)}$. Their product is
$(-1)^{x^{\mathsf T}By}$, and the $\ell$-term is a character. Thus
every displayed function is a refinement. Conversely the ratio of two
refinements is a character of $A$, hence has the form
$(-1)^{\ell^{\mathsf T}x}$. \autoref{lem:Gauss} gives the two square
roots for $\beta$.
\end{proof}
This counts refinements with fixed labels and normal basis.
Isometries of $(A,\chi)$ may identify braided equivalence classes.

For \(A=(\Z/2\Z)^2\), order the labels as
\((0,e_1,e_2,e_1+e_2)\).

\begin{example}[Alternating bicharacter]\label{ex:V4-alt}
Set $\chi(x,y)=(-1)^{x_1y_2+x_2y_1}$. Then
\begin{equation*}
 F=\frac{\nu}{2}
 \begin{pmatrix}1&1&1&1\\1&1&-1&-1\\1&-1&1&-1\\1&-1&-1&1\end{pmatrix}.
\end{equation*}
For $\nu=1$, the refinement $q(x)=(-1)^{x_1x_2}$ has values
$(1,1,1,-1)$ and Gauss sum $2$. Choose $\beta=1$. The crossing and
canonical twist are
\[
 c_{m,m}=e_0+e_{e_1}+e_{e_2}-e_{e_1+e_2},\qquad\theta_m=1.
\]
Every double crossing is the identity: the $mm$-channel scalars square
to $1$, and $\chi(a,a)=1$ for every $a$. Hence $S_{m,m}=4$.

With the same $\chi$ and $\tau$, instead take
$q'(x)=(-1)^{x_1x_2+x_1+x_2}$, with values $(1,-1,-1,-1)$.
Now $G(q')=-2$, so we may choose $\beta=i$. Then
\[
 c_{m,m}=i(e_0-e_{e_1}-e_{e_2}-e_{e_1+e_2}),\qquad
 \theta_m=-i,\qquad c_{m,m}^2=-\id_{mm}.
\]
Now $S_{m,m}=-4$. The invertible objects are transparent and $m$ is
not. Thus this planar category admits both symmetric and nonsymmetric
braidings.
\end{example}

\begin{example}[Nonalternating bicharacter]\label{ex:V4-dot}
Set $\chi(x,y)=(-1)^{x_1y_1+x_2y_2}$. Then
\begin{equation*}
 F=\frac{\nu}{2}
 \begin{pmatrix}1&1&1&1\\1&-1&1&-1\\1&1&-1&-1\\1&-1&-1&1\end{pmatrix}.
\end{equation*}
For $\nu=1$ choose $q(x)=i^{x_1+x_2}$, with values $(1,i,i,-1)$.
Its Gauss sum is $2i$, so take $\beta=e^{\pi i/4}$. Thus
\[
 R=e^{\pi i/4}\operatorname{diag}(1,-i,-i,-1),\qquad
 \theta_m=e^{-\pi i/4},\qquad
 (\theta_0,\theta_{e_1},\theta_{e_2},\theta_{e_1+e_2})=(1,-1,-1,1).
\]
We have $S_{m,m}=0$. The nontrivial invertible $e_1+e_2$ is
transparent, so the category is degenerate. The mixed recoupling
also differs from \autoref{ex:V4-alt}:
$\chi(e_1,e_1)=-1$, whereas $\chi(a,a)=1$ for every $a$
in the alternating case.
\end{example}

The Hopf link distinguishes these three crossing choices, while
the trefoil does not. With $\nu=1$, \autoref{eq:torus-values}
gives:
\begin{center}
\begin{tabular}{lcccc}
\toprule
$q$ in the order $(0,e_1,e_2,e_1+e_2)$ & $\beta$ & $H_2$ & $H_4$
 & zero-framed trefoil\\
\midrule
$(1,1,1,-1)$ & $1$ & $4$ & $4$ & $2$\\
$(1,-1,-1,-1)$ & $i$ & $-4$ & $4$ & $2$\\
$(1,i,i,-1)$ & $e^{\pi i/4}$ & $0$ & $-4$ & $2$\\
\bottomrule
\end{tabular}
\end{center}
For example, in the second row $H_3=i^3(1-1-1-1)=2i$ and
$\theta_m=-i$, so $\theta_m^{-3}H_3=2$. In the last row,
$H_4=(e^{\pi i/4})^4\sum_a1=-4$. Since $\beta^8=1$ in all
three rows, the trefoil value is $2$ throughout.

\end{document}